\documentclass{amsart}

\usepackage{geometry}
\usepackage{amssymb, esint, graphicx}
\usepackage{verbatim}
\usepackage[colorlinks=true,urlcolor=blue, citecolor=red,linkcolor=blue,linktocpage,pdfpagelabels, bookmarksnumbered,bookmarksopen]{hyperref}
\usepackage[hyperpageref]{backref}

\newtheorem{lemma}{Lemma}[section]
\newtheorem{theorem}[lemma]{Theorem}
\newtheorem{corollary}[lemma]{Corollary} 

\newtheorem{proposition}[lemma]{Proposition}
\newtheorem{openpb}{Open problem}
\newtheorem*{bwn}{BWN Conjecture}
\newtheorem*{bis}{Theorem \ref{thm:stability_lim} bis}

\theoremstyle{definition}
\newtheorem{definition}[lemma]{Definition}
\newtheorem{remark}[lemma]{Remark}
\newtheorem{example}[lemma]{Example}
\newtheorem*{ack}{Acknowledgements}

\numberwithin{equation}{section}

\author[Brasco]{Lorenzo Brasco}
\author[De Philippis]{Guido De Philippis}
\title{Spectral inequalities in quantitative form}

\address[L.\ Brasco]{Dipartimento di Matematica e Informatica
	\newline\indent
	Universit\`a degli Studi di Ferrara,
	Ferrara, Italy}
	\address{{\it and } Institut de Math\'ematiques de Marseille
	\newline\indent
	Aix-Marseille Universit\'e,
	Marseille, France}
\email{lorenzo.brasco@unife.it}

\address[G.\ De Philippis]{ SISSA, Via Bonomea 265, 34136 Trieste, Italy.}
\email{guido.dephilippis@sissa.it}

\begin{document}

\begin{abstract}
We review some results about quantitative improvements of sharp inequalities for eigenvalues of the Laplacian.
\end{abstract}

\maketitle

\begin{center}
\begin{minipage}{11cm}
\small
\tableofcontents
\end{minipage}
\end{center}

\section{Introduction}

\subsection{The problem}
Let $\Omega\subset\mathbb{R}^N$ be an open set. We consider the Laplacian operator $-\Delta$ on $\Omega$ under various boundary conditions. When the relevant spectrum happens to be discrete, it is an interesting issue to provide sharp geometric estimates on associated spectral quantities like the {\it ground state energy} (or first eigenvalue), the {\it fundamental gap} or more general functions of the eigenvalues. More precisely,
in this manuscript we will consider the following eigenvalue problems for the Laplacian:
\vskip.2cm
\begin{tabular}{cl}
{\bf Dirichlet conditions} &\qquad {\bf Robin conditions}\\
$\left\{\begin{array}{rcll}
-\Delta u&=&\lambda\,u,& \mbox{ in }\Omega,\\
u&=&0, & \mbox{ on }\partial\Omega,
\end{array}
\right.$&\qquad $\left\{\begin{array}{rcll}
-\Delta u&=&\lambda\,u,& \mbox{ in }\Omega,\\
\alpha\,u+\dfrac{\partial u}{\partial \nu}&=&0, & \mbox{ on }\partial\Omega,
\end{array}\quad (\alpha>0)
\right.$\\
&\\
{\bf Neumann conditions} &\qquad {\bf Steklov conditions}\\
$\left\{\begin{array}{rcll}
-\Delta u&=&\mu\,u,& \mbox{ in }\Omega,\\
u&=&0, & \mbox{ on }\partial\Omega,
\end{array}
\right.$ &\qquad $\left\{\begin{array}{rcll}
-\Delta u&=&0,& \mbox{ in }\Omega,\\
\dfrac{\partial u}{\partial \nu}&=&\sigma\,u, & \mbox{ on }\partial\Omega.
\end{array}
\right.$
\end{tabular}
\vskip.2cm\noindent
We denote by $\lambda_1(\Omega)$, $\lambda_1(\Omega,\alpha)$, $\mu_2(\Omega)$ and $\sigma_2(\Omega)$ the corresponding first (or first nontrivial\footnote{Observe that in the Neumann and Steklov cases, $0$ is always the first eigenvalue, associated to constant eigenfunctions. Thus we use the convention that $\sigma_1(\Omega)=\mu_1(\Omega)=0$. Also observe that the Robin case can be seen as an interpolation between Neumann (corresponding to $\alpha=0$) and Dirichlet conditions (when $\alpha=+\infty$).}) eigenvalue. We refer to the next sections for the precise definitions of these eiegnvalues and their properties.
For these spectral quantities, we have the following well-known sharp inequalities:
\vskip.2cm\noindent
{\bf Dirichlet case}
\begin{equation}
\label{faber}
|\Omega|^{2/N}\,\lambda_1(\Omega)\ge |B|^{2/N}\,\lambda_1(B),\qquad (\mbox{\it Faber-Krahn inequality})
\end{equation}
\vskip.2cm\noindent
{\bf Robin case}
\begin{equation}
\label{robin}
|\Omega|^{2/N}\,\lambda_1(\Omega,\alpha)\ge |B|^{2/N}\,\lambda_1(B,\alpha),\qquad (\mbox{\it Bossel-Daners inequality})
\end{equation}
\vskip.2cm\noindent
{\bf Neumann case}
\begin{equation}
\label{neumann}
 |B|^{2/N}\,\mu_2(B)\ge |\Omega|^{2/N}\,\mu_2(\Omega),\qquad (\mbox{\it Szeg\H{o}-Weinberger inequality})
\end{equation}
\vskip.2cm\noindent
{\bf Steklov case}
\begin{equation}
\label{brock}
|B|^{1/N}\,\sigma_2(B)\ge|\Omega|^{1/N}\,\sigma_2(\Omega), \qquad (\mbox{\it Brock-Weinstock inequality})
\end{equation}
where $B$ denotes an $N-$dimensional open ball. In all the previous estimates, equality holds only if $\Omega$ is a ball.
\par
The fact that balls can be characterized as the only sets for which equality holds in \eqref{faber}-\eqref{brock} naturally leads to consider the question of the {\it stability} of these inequalities. More precisely, one would like to improve \eqref{faber}, \eqref{robin}, \eqref{neumann} and \eqref{brock}, by adding in the right-hand sides a remainder term measuring the deviation of a set $\Omega$ from spherical symmetry.
\par 
For example, as for inequality \eqref{faber}, a typical quantitative Faber-Krahn inequality would read as follows
\begin{equation}
\label{enforced}
|\Omega|^{2/N}\,\lambda_1(\Omega)- |B|^{2/N}\,\lambda_1(B)\ge g(d(\Omega)),
\end{equation}
where:
\begin{itemize} 
\item $g:[0,+\infty)\to[0,+\infty)$ is some modulus of continuity, i.\,e. a positive continuous increasing function, vanishing at $0$ only;
\vskip.2cm
\item $\Omega\mapsto d(\Omega)$ is some scaling invariant {\it asymmetry functional}, i.\,e. a functional defined over sets such that
\[
d(t\,\Omega)=d(\Omega),\ \mbox{ for every }t>0\qquad \mbox{ and }\qquad d(\Omega)=0 \mbox{ if and only if } \Omega \mbox{ is a ball}. 
\]
\end{itemize}
Moreover, it would desirable to have quantitative enhancements which are ``the best possible'', in a sense. This means that not only we have \eqref{enforced} for every set $\Omega$,
but that it is possible to find a sequence of open sets $\{\Omega_n\}_{n\in\mathbb{N}}\subset\mathbb{R}^N$ such that
\[
\lim_{n\to\infty}\Big(|\Omega_n|^{2/N}\, \lambda_1(\Omega_n)-|B|^{2/N}\, \lambda_1(B)\Big)=0,
\]
and
\[
|\Omega_n|^{2/N}\, \lambda_1(\Omega_n)-|B|^{2/N}\, \lambda_1(B)\simeq g(d(\Omega_n)),\qquad \mbox{ as } n\to \infty.
\]
In this case, we would say that \eqref{enforced} is {\it sharp}.
In other words, the quantitative inequality \eqref{enforced} is sharp if it becomes asymptotically an equality, at least for particular shapes having small {\it deficits}.
\par
The quest for quantitative improvements of spectral inequalities has attracted an increasing interest in the last years. To the best of our knowledge, such a quest started with the papers \cite{MR1266215} by Hansen and Nadirashvili and \cite{MR1168980} by Melas. Both papers concern the Faber-Krahn inequality, which is indeed the most studied case. The reader is invited to consult Section \ref{sec:7} for more bibliographical references and comments.
\par
The aim of this manuscript is to give quite a complete picture on recent results about quantitative improvements of sharp inequalities for eigenvalues of the Laplacian. Apart from the inequalities for the first eigenvalues presented above, we will also take into account some other inequalities involving the second eigenvalue in the Dirichlet case, as well as the {\it torsional rigidity}. We warn the reader from the very beginning that the presentation will be limited to the Euclidean case. For the case of manifolds, we added some comments in Section \ref{sec:7}. 

\subsection{Plan of the paper} Each section is as self-contained as possible. Where it has not been possible to provide all the details, we have tried to provide precise references.
\par
In Section \ref{sec:2} we consider the case of the Faber-Krahn inequality \eqref{faber}, while the stability of the Szeg\H{o}-Weinberger and Brock-Weinstock inequalities is treated in Section \ref{sec:4} and \ref{sec:5}, respectively. For each of these sections, we first present the relevant stability result and then discuss its sharpness.
\par
Section \ref{sec:3} is a sort of {\it divertissement}, which shows some applications of the quantitative Faber-Krahn inequality to estimates for the so called {\it harmonic radius}. This part of the manuscript is essentially new and is placed there because some of the results presented will be used in Section \ref{sec:4}.
\par
Section \ref{sec:6} is devoted to present the proofs of other spectral inequalities, involving the second Dirichlet eigenvalue $\lambda_2$ as well. Namely, we consider the Hong-Krahn-Szego inequality for $\lambda_2$ and the Ashbaugh-Benguria inequality for the ratio $\lambda_2/\lambda_1$.
\par
Then in Section \ref{sec:7} we present some comments on further bibliographical references, applications and miscellaneous stability results on some particular classes of Riemannian manifolds.
\par
The work is complemented by $4$ appendices, containing technical results which are used throughout the paper.

\subsection{An open issue} We conclude the Introduction by pointing out that at present no quantitative stability results are available for the case of the Bossel-Daners inequality. We thus start by formulating the following
\begin{openpb}
Prove a quantitative stability estimate of the type \eqref{enforced} for the Bossel-Daners inequality for the first eigenvalue of the Robin Laplacian $\lambda_1(\Omega,\alpha)$.
\end{openpb}

\begin{ack}
This project was started while L.\,B. was still at Aix-Marseille Universit\'e, during a 6 months {\it  CNRS d\'el\'egation} period. He wishes to thank all his former colleagues at I2M institution. L.\,B. has been supported by the {\it Agence Nationale de la Recherche}, through the project ANR-12-BS01-0014-01 {\sc Geometrya}, G.\,D.\,P. is supported by the MIUR SIR-grant {\sc Geometric Variational Problems} (RBSI14RVEZ). 
\par
Both authors wish to warmly thank Mark S. Ashbaugh, Erwann Aubry and Nikolai Nadirashvili for interesting discussions and remarks on the subject, as well as their collaborators Giovanni Franzina, Aldo Pratelli, Berardo Ruffini and Bozhidar Velichkov. Special thanks go to Nicola Fusco, who first introduced L.\,B. to the realm of stability and quantitative inequalities, during his post-doc position in Naples.
\par
The authors are members of the Gruppo Nazionale per l'Analisi Matematica, la Probabilit\`a
e le loro Applicazioni (GNAMPA) of the Istituto Nazionale di Alta Matematica (INdAM).
\end{ack}

\section{Stability for the Faber-Krahn inequality}
\label{sec:2}

\subsection{A quick overview of the Dirichlet spectrum}

 For an open set $\Omega\subset\mathbb{R}^N$, we indicate by $W^{1,2}_0(\Omega)$ the completion of $C^\infty_0(\Omega)$ with respect to the norm
\[
 u\mapsto \left(\int_\Omega |\nabla u|^2\,dx\right)^\frac{1}{2},\qquad u\in C^\infty_0(\Omega).
\] 
The first eigenvalue of the Dirichlet Laplacian is defined by
\[
\lambda_1(\Omega):=\inf_{u\in C^\infty_0(\Omega)\setminus\{0\}} \frac{\displaystyle\int_\Omega |\nabla u|^2\, dx}{\displaystyle\int_\Omega |u|^2\, dx}.
\] 
In other words, this is the sharp constant in the Poincar\'e inequality
\[
c\,\int_{\Omega} |u|^2\,dx\le \int_\Omega |\nabla u|^2\,dx,\qquad u\in C^\infty_0(\Omega).
\]
Of course, it may happen that $\lambda_1(\Omega)=0$ if $\Omega$ does not support such an inequality.
\par
The infimum above is attained on $W^{1,2}_0(\Omega)$ whenever the embedding $W^{1,2}_0(\Omega)\hookrightarrow L^2(\Omega)$ is compact. In this case, the Dirichlet Laplacian has a discrete spectrum $\{\lambda_1(\Omega),\lambda_2(\Omega),\lambda_3(\Omega),\dots\}$ and
successive Dirichlet eigenvalues can be defined accordingly. Namely, $\lambda_k(\Omega)$ is obtained by minimizing the Rayleigh quotient above, among functions orthogonal (in the $L^2(\Omega)$ sense) to the first $k-1$ eigenfunctions. 
Dirichlet eigenvalues have the following scaling property
\[
\lambda_k(t\,\Omega)=t^{-2}\, \lambda_k(\Omega),\qquad t>0.
\]
Compactness of the embedding $W^{1,2}_0(\Omega)\hookrightarrow L^2(\Omega)$ holds for example when $\Omega\subset\mathbb{R}^N$ is an open set with finite measure. 
\par
In this case, it is possible to provide a sharp lower bound on $\lambda_1(\Omega)$ in terms of the measure of the set: this is the celebrated {\it Faber-Krahn inequality} \eqref{faber} recalled in the Introduction. The usual proof of  this inequality relies on the so-called {\it Schwarz symmetrization} (see \cite[Chapter 2]{MR2251558}). The latter consists in associating to each positive function $u\in W^{1,2}_0(\Omega)$ a radially symmetric decreasing function $u^*\in W^{1,2}_0(\Omega^*)$, where $\Omega^*$ is the ball centered at the origin such that $|\Omega^*|=|\Omega|$. The function $u^*$ is {\it equimeasurable} with $u$, that is
\[
|\{x\, :\, u(x)>t\}|=|\{x\, :\, u^*(x)>t\}|,\qquad \mbox{ for every } t\ge 0,
\]
so that in particular every $L^q$ norm of the function $u$ is preserved.
More interestingly, one has the {\it P\'olya-Szeg\H{o} principle} (see the Subsection \ref{sec:HNM})
\begin{equation}
\label{PS}
\int_{\Omega^*} |\nabla u^*|^2\, dx\le \int_\Omega |\nabla u|^2\, dx,
\end{equation}
from which the Faber-Krahn inequality easily follows.
\par
For a connected set $\Omega$, the first eigenvalue $\lambda_1(\Omega)$ is simple. In other words, there exists $u_1\in W^{1,2}_0(\Omega)\setminus\{0\}$ such that every solution to
\[
-\Delta u=\lambda_1(\Omega)\, u,\quad \mbox{ in }\Omega,\qquad\qquad u=0,\quad \mbox{ on }\partial\Omega,
\] 
is proportional to $u_1$.
For a ball $B_r$ of radius $r$, the value $\lambda_1(B_r)$ can be explicitely computed, together with its corresponding eigenfunction. The latter is given by the radial function (see \cite{MR2251558})
\[
u(x):=|x|^\frac{2-N}{2}\,J_\frac{N-2}{2}\left(\frac{j_{(N-2)/2,1}}{r}\,|x|\right).
\]
Here $J_\alpha$ is a Bessel function of the first kind, solving the ODE
\[
g''(t)+\frac{1}{t}\,g'(t)+\left(1-\frac{\alpha^2}{t^2}\right)g(t)=0\,,
\]
and $j_{\alpha,1}$ denotes the first positive zero of $J_{\alpha}$. We have
\begin{equation}
\label{lambda1palla}
\lambda_1(B_r)=\left(\frac{j_{(N-2)/2,1}}{r}\right)^2.
\end{equation}
\subsection{Semilinear eigenvalues and torsional rigidity} More generally, for an open set $\Omega\subset\mathbb{R}^N$ with finite measure, we will consider its {\it first semilinear eigenvalue} of the Dirichlet Laplacian
\begin{equation}
\label{autolavoro}
\lambda^q_{1}(\Omega)=\min_{u\in W^{1,2}_0(\Omega)\setminus\{0\}} \frac{\displaystyle\int_\Omega |\nabla u|^2\, dx}{\displaystyle\left(\int_\Omega |u|^q\, dx\right)^\frac{2}{q}}=\min_{u\in W^{1,2}_0(\Omega)}\left\{\int_\Omega |\nabla u|^2\, dx\, :\, \|u\|_{L^q(\Omega)}=1\right\},
\end{equation}
where the exponent $q$ satisfies
\begin{equation}
\label{q}
1\le q<2^*:=\left\{\begin{array}{rl}
\displaystyle \frac{2\,N}{N-2},&\mbox{ if } N\ge 3,\\
+\infty,& \mbox{ if }N=2.\\
\end{array}
\right.
\end{equation}
For every such an exponent $q$ the embedding $W^{1,2}_0(\Omega)\hookrightarrow L^q(\Omega)$ is compact, thus the above minimization problem is well-defined.
The shape functional $\Omega\mapsto \lambda_1^q(\Omega)$ verifies the scaling law
\[
\lambda^q_{1}(t\,\Omega)=t^{N-2-\frac{2}{q}\, N}\, \lambda_1^q(\Omega),
\]
the exponent $N-2-(2\,N)/q$ being negative. 
Still by means of Schwarz symmetrization, the following general family of Faber-Krahn inequalities can be derived 
\begin{equation}
\label{FKgen}
|\Omega|^{\frac{2}{N}+\frac{2}{q}-1}\, \lambda_1^q(\Omega)\ge |B|^{\frac{2}{N}+\frac{2}{q}-1}\, \lambda_1^q(B),
\end{equation}
where $B$ is any $N-$dimensional ball. Again, equality in \eqref{FKgen} is possible if and only if $\Omega$ is a ball, up to a set of zero capacity. Of course, when $q=2$ we are back to $\lambda_1(\Omega)$ defined above.
We also point out that the quantity 
\[
T(\Omega):=\frac{1}{\lambda^1_{1}(\Omega)}=\max_{u\in W^{1,2}_0(\Omega)\setminus\{0\}} \frac{\left(\displaystyle\int_\Omega u\, dx\right)^2}{\displaystyle\int_\Omega |\nabla u|^2\, dx}, 
\] 
is usually referred to as the {\it torsional rigidity} of the set $\Omega$. In this case, we can write \eqref{FKgen} in the form
\begin{equation}
\label{sv}
|B|^{-\frac{N+2}{N}}\,T(B)\ge|\Omega|^{-\frac{N+2}{N}}\,T(\Omega).
\end{equation}
This is sometimes called {\it Saint-Venant inequality}. We recall that for a ball of radius $R>0$ we have
\begin{equation}
\label{torsopalla}
T(B_R)=\frac{1}{\lambda_1^1(B_R)}=\frac{\omega_N}{N\,(N+2)}\, R^{N+2}.
\end{equation}
\begin{remark}[Torsion function]
\label{oss:torsion}
The torsional rigidity $T(\Omega)$ can be equivalently defined through an unconstrained convex problems, i.e.
\begin{equation}
\label{energiaaa}
-T(\Omega)=\min_{u\in W^{1,2}_0(\Omega)} \left\{\int_\Omega |\nabla u|^2\,dx-2\,\int_\Omega u\,dx\right\}.
\end{equation}
Indeed, it is sufficient to observe that for every $u\in W^{1,2}_0(\Omega)$ and $t>0$, the function $t\,u$ is still admissible and thus by Young's inequality
\[
\begin{split}
\min_{u\in W^{1,2}_0(\Omega)} \left\{\int_\Omega |\nabla u|^2\,dx-2\,\int_\Omega u\,dx\right\}&=\min_{u\in W^{1,2}_0(\Omega)} \min_{t>0} \left\{t^2\,\int_\Omega |\nabla u|^2\,dx-2\,t\,\int_\Omega u\,dx\right\}\\
&=\min_{u\in W^{1,2}_0(\Omega)} -\frac{\displaystyle\left(\int_\Omega u\,dx\right)^2}{\displaystyle\int_\Omega |\nabla u|^2\,dx},
\end{split}
\]
which proves \eqref{energiaaa}.
The unique solution $w_\Omega$ of the problem on the right-hand side in \eqref{energiaaa} is called {\it torsion function} and it satisfies
\[
-\Delta w_\Omega=1,\quad \mbox{ in }\Omega,\qquad w_\Omega=0,\quad \mbox{ on }\partial\Omega.
\]
From \eqref{energiaaa} and the equation satisfied by $w_\Omega$, we thus also get
\[
T(\Omega)=\int_\Omega w_\Omega\,dx.
\]
\end{remark}

\subsection{Some pioneering stability results}
\label{sec:HNM}

In this part we recall the quantitative estimates for the Faber-Krahn inequality by Hansen \& Nadirashvili \cite{MR1266215} and Melas \cite{MR1168980}.
\par
First of all, as the proof of the Faber-Krahn inequality is based on the P\'olya-Szeg\H{o} principle \eqref{PS}, it is better to recall how \eqref{PS} can be proved. By following Talenti (see \cite[Lemma 1]{MR0463908}), the proof combines the {\it Coarea Formula}, the {\it convexity} of the function $t\mapsto t^2$ and 
the Euclidean {\it Isoperimetric Inequality}
\begin{equation}
\label{isosciarpa}
|\Omega|^\frac{1-N}{N}\,P(\Omega)\ge |B|^\frac{1-N}{N}\, P(B).
\end{equation}
Here $P(\,\cdot\,)$ denotes the perimeter of a set. If $u\in W^{1,2}_0(\Omega)$ is a smooth positive function and we set
\[
\Omega_t:=\{x\in\Omega\, :\, u(x)>t\}\qquad \mbox{ and }\qquad \mu(t):=|\Omega_t|,
\]
by using the above mentioned tools, one can infer
\begin{equation}
\label{PSegola}
\begin{split}
\int_\Omega |\nabla u|^2\,dx&\stackrel{\mbox{\tt Coarea}}{=} \int_0^\infty \left(\int_{\{u=t\}} |\nabla u|^2\,\frac{d\mathcal{H}^{N-1}}{|\nabla u|}\right)\,dt\\
&\stackrel{\mbox{\tt Jensen}}{\ge} \int_0^\infty P(\Omega_t)^2\, \frac{1}{\displaystyle\int_{\{u=t\}} |\nabla u|^{-1}\,d\mathcal{H}^{N-1}} dt=\int_0^\infty \frac{P(\Omega_t)^2}{-\mu'(t)}\, dt\\
&\stackrel{\mbox{\tt Isoperimetry}}{\ge} \int_0^\infty \frac{P(\Omega^*_t)^2}{-\mu'(t)}\, dt=\int_{\Omega^*} |\nabla u^*|^2\,dx,
\end{split}
\end{equation}
where $\Omega^*_t$ is the ball centered at the origin such that $|\Omega^*_t|=|\Omega_t|$.
For a smooth function, the equality 
\[
-\mu'(t)=\int_{\{u=t\}} \frac{1}{|\nabla u|}\,d\mathcal{H}^{N-1},\qquad \mbox{ for a.\,e. } t>0, 
\]
follows from Sard's Theorem, but all the passages in \eqref{PSegola} can indeed be justified for a genuine $W^{1,2}_0$ function. We refer the reader to \cite[Section 2]{MR2376285} for more details.
\par
By taking $u$ to be a first eigenfunction of $\Omega$ with unit $L^2$ norm and observing that $u^*$ is admissible for the variational problem defining $\lambda_1(\Omega^*)$, from \eqref{PSegola} one easily gets the Faber-Krahn inequality
\[
\lambda_1(\Omega)\ge \lambda_1(\Omega^*),
\]
as desired.
\vskip.2cm
The idea of Hansen \& Nadirashvili \cite{MR1266215} and Melas \cite{MR1168980} is to replace in \eqref{PSegola} the classical isoperimetric statement \eqref{isosciarpa} with an improved quantitative version. At the time of \cite{MR1266215} and \cite{MR1168980}, quantitative versions of the isoperimetric inequality were availbale only for some particular sets, under the name of {\it Bonnesen inequalities}. These cover simply connected sets in dimension $N=2$ (see Bonnesen's paper \cite{MR1512192}, generalized in \cite[Theorem 2.2]{MR1112666}) and convex sets in every dimension (see \cite[Theorem 2.3]{MR0942426}).
\par
For this reason, both papers treat simply connected sets in dimension $N=2$ or convex sets in general dimensions. We now present their results, without entering at all into the details of the proofs. Rather, in the next subsection we will explain the ideas by Hansen and Nadirashvili and use them to prove a fairly more general result (see Theorem \ref{teo:HN} below).
\begin{theorem}[Melas]
\label{teo:melas}
For every open bounded set $\Omega\subset\mathbb{R}^N$, we define the asymmetry functional
\begin{equation}
\label{melas}
d_{\mathcal{M}}(\Omega):=\min\left\{\max\left\{\frac{|\Omega\setminus B_1|}{|\Omega|},\,\frac{|B_2\setminus \Omega|}{|B_2|}\right\}\, :\, B_1\subset \Omega\subset B_2 \mbox{ balls}\right\}.
\end{equation}
Then we have:
\begin{itemize}
\item if $N=2$, for every $\Omega$ open bounded simply connected set, there exists a disc $B_\Omega\subset \Omega$ such that
\[
|\Omega|\,\lambda_1(\Omega)-|B|\,\lambda_1(B)\ge \frac{1}{C}\,\left(\frac{|\Omega\setminus B_\Omega|}{|\Omega|}\right)^4,
\]
for some universal constant $C>0$; 
\vskip.2cm
\item if $N\ge 2$ for every open bounded convex set $\Omega\subset\mathbb{R}^N$ we have 
\[
|\Omega|^{2/N}\,\lambda_1(\Omega)- |B|^{2/N}\,\lambda_1(B)\ge \frac{1}{C}\,d_{\mathcal{M}}(\Omega)^{2\,N},
\] 
for some universal constant $C>0$.
\end{itemize}
\end{theorem}
\begin{remark}
In dimension $N=2$ Melas' result is indeed more general. If $\Omega\subset\mathbb{R}^2$ is an open bounded set, not necessarily simply connected, then there exists an open disc $B_\Omega$ such that
\[
|\Omega|\,\lambda_1(\Omega)-|B|\,\lambda_1(B)\ge \frac{1}{C}\,\left(\frac{|\Omega\Delta B_\Omega|}{|\Omega\cup B_\Omega|}\right)^4,
\]
for some universal constant $C>0$.
\end{remark}
\noindent
For every open set $\Omega\subset\mathbb{R}^N$ we note
\[
r_\Omega=\sup\{r>0\, :\, \mbox{there exists }x_0\in\Omega \mbox{ such that } B_r(x_0)\subset\Omega\}\qquad \mbox{ and} \qquad R_{\Omega}=\left(\frac{|\Omega|}{\omega_N}\right)^\frac{1}{N}.
\]
The first quantity is usually called {\it inradius of $\Omega$}. This is the radius of the largest ball contained in $\Omega$.
\begin{theorem}[Hansen-Nadirashvili]
\label{teo:HNoriginal}
For every open set $\Omega\subset\mathbb{R}^N$ with finite measure, we define the asymmetry functional
\begin{equation}
\label{d1}
d_\mathcal{N}(\Omega):=1-\frac{r_\Omega}{R_{\Omega}}.
\end{equation}
Then we have:
\begin{itemize}
\item if $N=2$ and $\Omega$ is simply connected,
\[
|\Omega|\,\lambda_1(\Omega)- |B|\,\lambda_1(B)\ge \left(\frac{\pi\,j_{0,1}^2}{250}\right)\,d_\mathcal{N}(\Omega)^3;
\]
\item if $N\ge 3$, there exist $0<\varepsilon<1$ and $C>0$ such that for every $\Omega\subset\mathbb{R}^N$ open bounded convex set satisfying $d_\mathcal{N}(\Omega)<\varepsilon$, we have
\[
|\Omega|^{2/N}\,\lambda_1(\Omega)- |B|^{2/N}\,\lambda_1(B)\ge \frac{1}{C}\,\left\{\begin{array}{cc}
\displaystyle\frac{d_\mathcal{N}(\Omega)^3}{|\log d_\mathcal{N}(\Omega)|}, & \mbox{ if } N=3,\\
&\\
d_\mathcal{N}(\Omega)^\frac{N+3}{2}, & \mbox{ if } N\ge 4.
\end{array}\right.
\]
\end{itemize}
\end{theorem}

\begin{remark}[The role of topology]
It is easy to see that the stability estimates of Theorems \ref{teo:melas} and \ref{teo:HNoriginal} with $d_\mathcal{M}$ and $d_\mathcal{N}$ {\it can not hold true} without some topological assumptions on the sets. For example, by taking the perforated ball
\[
\Omega_\varepsilon=\{x\in\mathbb{R}^N \, :\, \varepsilon<|x|<1\},\qquad 0<\varepsilon<1,
\] 
we have
\begin{equation}
\label{fluccher}
\lim_{\varepsilon\searrow 0} \left(|\Omega_\varepsilon|^{2/N}\,\lambda(\Omega_\varepsilon)- |B|^{2/N}\,\lambda(B)\right)=0,
\end{equation}
while
\[
\lim_{\varepsilon\searrow 0} d_\mathcal{N}(\Omega_\varepsilon)=\frac{1}{2}\qquad \mbox{ and } \qquad \lim_{\varepsilon\searrow 0} d_\mathcal{M}(\Omega_\varepsilon)\ge \frac{1}{2}.
\]
For the limit \eqref{fluccher} see for example \cite[Theorem 9]{MR1338506}. These contradict Theorems \ref{teo:melas} and \ref{teo:HNoriginal}.
Observe that for $N=2$ the set $\Omega_\varepsilon$ is not simply connected, while for $N\ge 3$ it is. Thus in higher dimensions simple connectedness is still not sufficient to have stability with respect to $d_\mathcal{N}$ or $d_\mathcal{M}$.
\end{remark}
If we want to obtain a quantitative Faber-Krahn inequality for general open sets in every dimension, a more flexible notion of asymmetry is the so called {\it Fraenkel asymmetry}, defined by 
\[
\mathcal{A}(\Omega)=\inf\left\{\frac{|\Omega\Delta B|}{|\Omega|} \, :\,  \text{ \(B\) ball such that \(|B|=|\Omega|\)}\right\}.
\]
Observe that for every ball $B$ such that $|B|=|\Omega|$, we have $|\Omega\Delta B|=2\,|\Omega\setminus B|=2\,|B\setminus \Omega|$. This simple facts will be used repeatedly.
\par
It is not difficult to see that this is a weaker asymmetry functional, with respect to $d_\mathcal{N}$ and $d_\mathcal{M}$ above. Indeed, we have the following.
\begin{lemma}[Comparison between asymmetries]
Let $\Omega\subset\mathbb{R}^N$ be an open bounded set. Then we have
\begin{equation}
\label{disuguaglianze}
d_\mathcal{N}(\Omega)\ge \frac{1}{2\,N}\,\mathcal{A}(\Omega),\qquad d_\mathcal{M}(\Omega)\ge \frac{1}{2}\,\mathcal{A}(\Omega)\qquad \mbox{ and }\qquad  d_\mathcal{M}(\Omega)\ge \frac{N}{2^{N}}\, d_\mathcal{N}(\Omega).
\end{equation}
If $\Omega$ is convex, we also have
\[
\mathcal{A}(\Omega)\ge \frac{1}{N}\,\left(\frac{\omega_N}{N}\right)^{1/N}\,\frac{|\Omega|^{1/N}}{\mathrm{diam}(\Omega)}\, d_\mathcal{M}(\Omega)^N.
\]
\end{lemma}
\begin{proof}
By using the elementary inequality 
\[
a^N-b^N\le N\, a^{N-1}\,(a-b),\qquad \mbox{ for } 0\le b\le a, 
\]
and the definitions of $r_\Omega$, $R_\Omega$ and $d_\mathcal{N}(\Omega)$, we have
\begin{equation}
\label{frankd1}
|\Omega|-\omega_N\,r_\Omega^N\le N\, |\Omega|^\frac{N-1}{N}\,\Big(|\Omega|^\frac{1}{N}-\omega_N^\frac{1}{N}\,r_\Omega\Big)= N\, |\Omega|\, d_\mathcal{N}(\Omega).
\end{equation}
We then consider a ball $B_{r_\Omega}(x_0)\subset\Omega$ and take a concentric ball $\widetilde B$ such that $|\widetilde B|=|\Omega|$. By definition of Fraenkel asymmetry and estimate \eqref{frankd1}, we obtain
\[
\mathcal{A}(\Omega)\le 2\, \frac{|\Omega\setminus \widetilde B|}{|\Omega|}\le 2\, \frac{|\Omega\setminus B_{r_\Omega}(x_0)|}{|\Omega|}\le 2\,N\, d_\mathcal{N}(\Omega),
\]
and thus we get the first estimate in \eqref{disuguaglianze}.
\par
For the second one, we take a pair of balls $B_1\subset \Omega\subset B_2$ and consider the ball $\widetilde B_1$ concentric with $B_1$ and such that $|\Omega|=|\widetilde B_1|$. Then we get
\[
\mathcal{A}(\Omega)\le 2\, \frac{|\Omega\setminus \widetilde B_1|}{|\Omega|}\le 2\, \frac{|\Omega\setminus B_1|}{|\Omega|}\le 2\,\max\left\{\frac{|\Omega\setminus B_1|}{|\Omega|},\,\frac{|B_2\setminus \Omega|}{|B_2|}\right\}.
\] 
by taking the infimum over the admissible pairs $(B_1,B_2)$ we get the second inequality in \eqref{disuguaglianze}. 
\par
Finally, for the third one we take again a pair of balls $B_1\subset \Omega\subset B_2$ and observe that if $r_\Omega\le R_\Omega/2$, then  we have
\[
\max\left\{\frac{|\Omega\setminus B_1|}{|\Omega|},\,\frac{|B_2\setminus \Omega|}{|B_2|}\right\}\ge \frac{|\Omega\setminus B_1|}{|\Omega|}\ge 1-\left(\frac{1}{2}\right)^N\ge \left[1-\left(\frac{1}{2}\right)^N\right]\,d_\mathcal{N}(\Omega).
\]
In the last inequality we used that $d_\mathcal{N}(\Omega)<1$. By taking the infimum over admissible couple of balls, we obtain the conclusion. If on the contrary $\Omega$ is such that $r_\Omega> R_\Omega/2$, then 
by definition of $R_\Omega$ and $r_\Omega$ we get
\[
\begin{split}
d_\mathcal{N}(\Omega)=\frac{R_\Omega-r_\Omega}{R_\Omega}=\frac{|\Omega|^{1/N}-|B_{r_\Omega}|^{1/N}}{|\Omega|^{1/N}}&\le \frac{1}{N}\, \frac{|B_{r_\Omega}|^{1/N}}{|\Omega|^{1/N}}\,\frac{|\Omega|-|B_{r_\Omega}|}{|B_{r_\Omega}|}\\
&\le \frac{1}{N}\, \frac{|\Omega|-|B_1|}{|B_{r_\Omega}|}\le \frac{2^N}{N}\, \frac{|\Omega\setminus B_1|}{|\Omega|}\\
&\le \frac{2^{N}}{N}\,\max\left\{\frac{|\Omega\setminus B_1|}{|\Omega|},\,\frac{|B_2\setminus \Omega|}{|B_2|}\right\},
\end{split}
\]
where we used that $|B_1|\le |B_{r_\Omega}|\le |\Omega|$.
Thus we get the conclusion in this case as well. Observe that $1-2^{-N}\ge N\,2^{-N}$ for $N\ge 2$.
\vskip.2cm\noindent
Let us now assume $\Omega$ to be convex. We take a ball $|B|=|\Omega|$ such that $|\Omega\setminus B|/|\Omega|=\mathcal{A}(\Omega)/2$. We can assume that $\mathcal{A}(\Omega)<1/2$, otherwise the estimate is trivial by using the isodiametric inequality and the fact that $d_\mathcal{M}<1$. Then from \cite[Lemma 4.2]{MR2207737} we know that
\[
|\Omega\setminus B|\ge \frac{1}{2\,N}\,\frac{|\Omega|}{\mathrm{diam}(\Omega)^N}\,\mathrm{Haus}(\Omega,B)^N.
\]
Here $\mathrm{Haus}(E_1,E_2)$ denotes the Hausdorff distance between sets, defined by
\begin{equation}
\label{indahaus}
\mathrm{Haus}(E_1,E_2)=\max\left\{\sup_{x\in E_1}\inf_{y\in E_2} |x-y|,\,\sup_{y\in E_2}\inf_{x\in E_1} |x-y|\right\}.
\end{equation}
We then observe that the ball $B_2:=\gamma\,B$ contains $\Omega$, provided
\[
\gamma=\frac{R_\Omega+\mathrm{Haus}(\Omega,B)}{R_\Omega}. 
\]
On the other hand, by Lemma \ref{lm:blasphemy} we have 
\[
r_\Omega\ge R_\Omega-\mathrm{Haus}(\Omega,B).
\]
Let us assume that $\mathrm{Haus}(\Omega,B)<R_\Omega$.
From the definition of $d_\mathcal{M}$, we thus obtain
\[
\begin{split}
d_\mathcal{M}(\Omega)&\le \max\left\{\frac{\gamma^N-1}{\gamma^N},\, 1-\frac{r_\Omega^N}{R_\Omega^N}\right\}\le \max\left\{1-\frac{1}{\gamma^N},\, 1-\left(1-\frac{\mathrm{Haus}(\Omega,B)}{R_\Omega}\right)^N\right\}\\
&=\max\left\{1-\left(1-\frac{\mathrm{Haus}(\Omega,B)}{R_\Omega+\mathrm{Haus}(\Omega,B)}\right)^N,\, 1-\left(1-\frac{\mathrm{Haus}(\Omega,B)}{R_\Omega}\right)^N\right\}\\
&\le N\,\max\left\{\frac{\mathrm{Haus}(\Omega,B)}{R_\Omega+\mathrm{Haus}(\Omega,B)},\,\frac{\mathrm{Haus}(\Omega,B)}{R_\Omega} \right\}\\
&=N\,\frac{\mathrm{Haus}(\Omega,B)}{R_\Omega}\le N\,\left(\frac{N}{\omega_N}\right)^\frac{1}{N}\,\frac{\mathrm{diam}(\Omega)}{|\Omega|^{1/N}}\,\mathcal{A}(\Omega)^{1/N}.
\end{split}
\]
This concludes the proof for $\mathrm{Haus}(\Omega,B)<R_\Omega$. If on the contrary $\mathrm{Haus}(\Omega,B)\ge R_\Omega$, with similar computations we get
\[
\begin{split}
d_\mathcal{M}(\Omega)\le \max\left\{\frac{\gamma^N-1}{\gamma^N},\, 1-\frac{r_\Omega^N}{R_\Omega^N}\right\}&\le\max\left\{1-\left(1-\frac{\mathrm{Haus}(\Omega,B)}{R_\Omega+\mathrm{Haus}(\Omega,B)}\right)^N,\, 1\right\}\\
&\le \max\left\{N\,\frac{\mathrm{Haus}(\Omega,B)}{R_\Omega+\mathrm{Haus}(\Omega,B)},\,1 \right\}\le \frac{N}{2}\,\frac{\mathrm{Haus}(\Omega,B)}{R_\Omega},
\end{split}
\]
and we can conclude as before.
\end{proof}
\begin{remark}
For general open sets, the asymmetries $\mathcal{A}, d_\mathcal{N}$ and $d_\mathcal{M}$ are not equivalent. We first observe that if $\Omega_0=B\setminus \Sigma$, where $B$ is a ball and $\Sigma\subset B$ is a non-empty closed set with $|\Sigma|=0$, then
\[
\mathcal{A}(\Omega_0)=0\qquad \mbox{ while }\qquad  d_\mathcal{N}(\Omega_0)>0,\quad \mbox{ and }\quad d_\mathcal{M}(\Omega_0)>0.
\]
Moreover, there exists a sequence of open sets $\{\Omega_n\}_{n\in\mathbb{N}}\subset\mathbb{R}^N$ such that
\[
\lim_{n\to\infty} d_\mathcal{N}(\Omega_n)=0\qquad\mbox{ and }\qquad \lim_{n\to\infty} d_\mathcal{M}(\Omega_n)>0.
\]
Such a sequence $\{\Omega_n\}_{n\in\mathbb{N}}$ can be constructed by attaching a long tiny tentacle to a ball, for example.
\end{remark}

\subsection{A variation on a theme by Hansen and Nadirashvili}
\label{sec:24}

We will now show how to adapt the ideas by Hansen and Nadirashvili, in order to get a (non sharp) stability estimate for the general Faber-Krahn inequality \eqref{FKgen} and for general open sets with finite measure.
\par
First of all, one needs a quantitative improvement of the isoperimetric inequality which is valid for generic sets and dimensions. Such a (sharp) quantitative isoperimetric inequality has been proved by Fusco, Maggi and Pratelli in \cite[Theorem 1.1]{MR2456887} (see also \cite{MR2980529, MR2672283} for different proofs and \cite{MR3404715} for an exhaustive review of quantitative forms of the isoperimetric inequality). This reads as follows
\begin{equation}
\label{isosciarpa+}
|\Omega|^\frac{1-N}{N}\,P(\Omega)- |B|^\frac{1-N}{N}\, P(B)\ge \beta_N\,\mathcal{A}(\Omega)^2.
\end{equation}
An explicit value for the dimensional constant $\beta_N>0$ can be found in \cite[Theorem 1.1]{MR2672283}. By inserting this information in the proof \eqref{PSegola} of P\'olya-Szeg\H{o} inequality, one would get an estimate of the type
\[
\int_\Omega |\nabla u|^2\,dx-\int_{\Omega^*} |\nabla u^*|^2\,dx\gtrsim \int_0^\infty \mathcal{A}(\Omega_t)^2\,dt.
\]
The difficult point is to estimate the ``propagation of asymmetry'' 
from the whole domain $\Omega$ to the superlevel sets $\Omega_t$ of the optimal function $u$. In other words, we would need to know that
\[
\mathcal{A}(\Omega)\simeq \mathcal{A}(\Omega_t),\qquad \mbox{ for } t>0.
\]
Unfortunately, in general it is difficult to exclude that
\[
\mathcal{A}(\Omega_t)\ll \mathcal{A}(\Omega),\qquad \mbox{ for } t\simeq 0.
\]
This means that the graph of $u$ ``quickly becomes round'' when it detaches from the boundary $\partial\Omega$. This may happen for example if $u$ has a small normal derivative.
For these reasons, improving this idea is very delicate, which usually results in 
a (non sharp) estimate like the ones of Theorems \ref{teo:melas} and \ref{teo:HNoriginal} and the one of Theorem \ref{teo:HN} below. We refer to the discussion of Section \ref{sec:71} for other results of this type, previously obtained by Bhattacharya \cite{MR1836803} and Fusco, Maggi and Pratelli \cite{MR2512200}.
\vskip.2cm
The following expedient result is sometimes useful for stability issues.
It states that if the measure of a subset
$U\subset\Omega$ differs from that of $\Omega$ by an amount comparable to the asymmetry $\mathcal{A}(\Omega)$, then the asymmetry of $U$ can not decrease too much. This is encoded in the following simple result, which is essentially taken from \cite[Section 5]{MR1266215}.
\begin{lemma}[Propagation of asymmetry]
\label{lm:nikolai}
Let $\Omega\subset\mathbb{R}^N$ be an open set with finite measure. Let $U\subset \Omega$ be such that $|U|>0$ and
\begin{equation}
\label{ipo}
\frac{|\Omega\setminus U|}{|\Omega|}\le \frac{1}{4}\, \mathcal{A}(\Omega).
\end{equation}
Then there holds
\begin{equation}
\label{comparison}
\mathcal{A}(U)\ge \frac{1}{2}\, \mathcal{A}(\Omega).
\end{equation}
\end{lemma}
\begin{proof}
Let $B$ be a ball achieving the minimum in the definition of $\mathcal{A}(U)$, by triangular inequality we get
\[
\begin{split}
\mathcal{A}(U)=\frac{|U\Delta B|}{|U|}&\ge \frac{|\Omega|}{|U|}\,\left(\frac{|\Omega\Delta B|}{|\Omega|}-\frac{|U\Delta \Omega|}{|\Omega|}\right)\\
&\ge \frac{|\Omega|}{|U|}\,\left(\frac{|\Omega\Delta B'|}{|\Omega|}-\frac{|B\Delta B'|}{|\Omega|}-\frac{|U\Delta \Omega|}{|\Omega|}\right),
\end{split}
\]
where $B'$ is a ball concentric with $B$ and such that $|B'|=|\Omega|$. By using that
\[
|U\Delta \Omega|=|\Omega\setminus U|=|\Omega|-|U|=|B'\Delta B|,
\]
and the hypothesis \eqref{ipo}, we get the conclusion by further noticing that $|\Omega|\ge |U|$.
\end{proof}
By relying on the previous simple result, we can prove a sort of P\'olya-Szeg\H{o} inequality with remainder term. The remainder term depends on the asymmetry of $\Omega$ and on the level $s$ of the function, whose corresponding superlevel set $\{x : u(x)>s\}$
has a {\it measure defect} comparable to the asymmetry $\mathcal{A}(\Omega)$, i.e. it satisfies \eqref{ipo}.
\begin{lemma}[Boosted P\'olya-Szeg\H{o} principle]
\label{lm:PSbooster}
Let $\Omega\subset\mathbb{R}^N$ be an open set with finite measure, such that $\mathcal{A}(\Omega)>0$. Let $u\in W^{1,2}_0(\Omega)$ be such that $u>0$ in $\Omega$. For every $t>0$ we still denote
\[
\Omega_t=\{x\in\Omega\, :\, u(x)>t\}\qquad \mbox{ and }\qquad \mu(t)=|\Omega_t|.
\]
Let $s>0$ be the level defined by
\begin{equation}
\label{sceltap}
s=\sup\left\{t\, :\, \mu(t)\ge |\Omega|\,\left(1-\frac{1}{4}\,\mathcal{A}(\Omega)\right)\right\}.
\end{equation}
Then we have
\begin{equation}
\label{protop}
\int_\Omega |\nabla u|^2\, dx\ge \int_{\Omega^*} |\nabla u^*|^2\, dx+c_N\, \mathcal{A}(\Omega)\,|\Omega|^{1-\frac{2}{N}}\, s^2.
\end{equation}
The dimensional constant $c_N>0$ is given by
\[
c_N=\frac{4^{1/N}\,\beta_N\,N\,\omega_N^{1/N}}{2},
\]
and $\beta_N$ is the same constant appearing in \eqref{isosciarpa+}.
\end{lemma}
\begin{proof}
We first observe that the level $s$ defined by \eqref{sceltap} is not $0$. Indeed, the function $t\mapsto \mu(t)$ is right-continuous, thus we get
\[
\lim_{t\to 0^+} \mu(t)=\mu(0)=|\{x\in\Omega\, :\, u(x)>0\}|=|\Omega|>|\Omega|\,\left(1-\frac{1}{4}\,\mathcal{A}(\Omega)\right),
\]
where we used the hypothesis on $u$ and the fact that $\mathcal{A}(\Omega)>0$.
\par
By using the sharp quantitative isoperimetric inequality \eqref{isosciarpa+}, we have
\begin{equation}
\label{isoFMP}
P(\Omega_t)\ge P(\Omega^*_t)+\beta_N\, \mu(t)^\frac{N-1}{N}\,\mathcal{A}(\Omega_t)^2,
\end{equation}
while by convexity of the map $\tau\mapsto \tau^2$ we get
\[
\begin{split}
P(\Omega_t)^2&\ge P(\Omega^*_t)^2+2\, P(\Omega^*_t)\, \Big(P(\Omega_t)- P(\Omega^*_t)\Big)\\
&=P(\Omega^*_t)^2+2\, \Big(N\, \omega_N^{1/N}\, \mu(t)^\frac{N-1}{N}\Big)\,\Big(P(\Omega_t)- P(\Omega^*_t)\Big).
\end{split}
\]
By collecting the previous two estimates and reproducing the proof of \eqref{PSegola}, we can now infer
\begin{equation}
\label{basep}
\begin{split}
\int_\Omega |\nabla u|^2\,dx&\ge \int_{\Omega^*} |\nabla u^*|^2\,dx+c\,\int_0^s \mathcal{A}(\Omega_t)^2\, \frac{\left(\mu(t)^\frac{N-1}{N}\right)^{2}}{-\mu'(t)}\, dt,
\end{split}
\end{equation}
where we set 
\[
c=2\,\beta_N\,N\,\omega_N^{1/N}.
\]
We now observe that $\mu$ is a decreasing function, thus we have
\[
\mu(t)> \mu(s)\ge |\Omega|\,\left(1-\frac{1}{4}\,\mathcal{A}(\Omega)\right),\qquad 0<t<s.
\]
This implies that the set $\Omega_t$ verifies the hypothesis of Lemma \ref{lm:nikolai} for $0<t<s$, since
\[
\frac{|\Omega\setminus\Omega_t|}{|\Omega|}=1-\frac{\mu(t)}{|\Omega|}\le 1-\left(1-\frac{1}{4}\,\mathcal{A}(\Omega)\right)=\frac{1}{4}\,\mathcal{A}(\Omega),\qquad 0<t<s.
\]
Thus from \eqref{comparison} we get
\[
\mathcal{A}(\Omega_t)\ge \frac{1}{2}\, \mathcal{A}(\Omega),\qquad 0<t<s.
\]
By inserting the previous information in \eqref{basep} and using that
\[
\mu(t)>\mu(s)\ge |\Omega|\,\left(1-\frac{1}{4}\,\mathcal{A}(\Omega)\right)\ge \frac{|\Omega|}{2},\qquad 0<t<s,
\] 
we get
\[
\begin{split}
\int_\Omega |\nabla u|^2\, dx&\ge \int_{\Omega^*} |\nabla u^*|^2\,dx+\frac{c}{4}\,\mathcal{A}(\Omega)^2\,\left(\frac{|\Omega|^\frac{N-1}{N}}{2^\frac{N-1}{N}} \right)^2\,\int_0^s\frac{1}{-\mu'(t)}\, dt.
\end{split}
\]
We then observe that by convexity of the function $\tau\mapsto \tau^{-1}$, Jensen inequality gives\footnote{In the second inequality, we used that for a monotone non-decreasing function $f$
\[
\int_a^b f'(t)\,dt\le f(b)-f(a),\qquad \mbox{ for } a<b.
\]}
\[
\int_0^s\frac{1}{-\mu'(t)}\, dt\ge \frac{s^2}{\displaystyle\int_0^s -\mu'(t)\, dt}\ge \frac{s^2}{|\Omega|-\mu(s)}\ge 4\,\frac{s^2}{\mathcal{A}(\Omega)\, |\Omega|}, 
\]
thanks to the choice \eqref{sceltap} of $s$. This concludes the proof.
\end{proof}
We can now prove the following quantitative version of the general Faber-Krahn inequality \eqref{FKgen}. The standard case of the first eigenvalue of the Dirichlet Laplacian corresponds to taking $q=2$.
Though the exponent on the Fraenkel asymmetry is not sharp, the interest of the result lies in the computable constant. 
Moreover, the proof is quite simple and it is based on the ideas by Hansen \& Nadirashvili. We also use {\it Kohler-Jobin inequality} (see Appendix \ref{sec:KJ}) to reduce to the case of the torsional rigidity. This reduction trick has been first introduced by Brasco, De Philippis and Velichkov in \cite{MR3357184}.
\begin{theorem}
\label{teo:HN}
Let $1\le q<2^*$, there exists an explicit constant $\tau_{N,q}>0$ such that for every open set $\Omega\subset\mathbb{R}^N$ with finite measure, we have
\begin{equation}
\label{HN2}
|\Omega|^{\frac{2}{N}+\frac{2}{q}-1}\, \lambda_1^q(\Omega)- |B|^{\frac{2}{N}+\frac{2}{q}-1}\, \lambda_1^q(B)\ge \tau_{N,q}\, \mathcal{A}(\Omega)^3.
\end{equation}
\end{theorem}
\begin{proof} 
Since inequality \eqref{HN2} is scaling invariant, we can suppose that $|\Omega|=1$.
By Proposition \ref{prop:gerarchia} with 
\[
g(t)=t^3\qquad \mbox{ and }\qquad d(\Omega)=\mathcal{A}(\Omega),
\]
it is sufficient to prove \eqref{HN2} for the torsional rigidity. In other words, we just need to prove the following {\it quantitative Saint-Venant inequality}
\begin{equation}
\label{HN21}
T(\Omega^*)-T(\Omega)\ge \tau\,\mathcal{A}(\Omega)^3,
\end{equation}
where as always $\Omega^*$ is the ball centered at the origin such that $|\Omega^*|=|\Omega|=1$.
Of course, we can suppose that $\mathcal{A}(\Omega)>0$, otherwise there is nothing to prove. Without loss of generality, we can also suppose that
\begin{equation}
\label{smezzino}
T(\Omega)\ge\frac{T(\Omega^*)}{2}.
\end{equation}
Indeed, if the latter is not satisfied, then \eqref{HN21} trivially holds with constant $\tau=T(\Omega^*)/16$, thanks to the fact that $\mathcal{A}(\Omega)<2$.
\vskip.2cm\noindent
Let $w_\Omega\in W^{1,2}_0(\Omega)$ be the torsion function of $\Omega$ (recall Remark \ref{oss:torsion}), then we know
\[
T(\Omega)=\int_\Omega w_\Omega\, dx=\int_\Omega |\nabla w_\Omega|^2\, dx.
\]
Moreover, by standard elliptic regularity we know that $w_\Omega\in C^\infty(\Omega)\cap L^\infty(\Omega)$.
By recalling that $|\Omega|=1$ and \eqref{smezzino}, we get
\[
\frac{T(\Omega^*)}{2}\le T(\Omega)=\int_\Omega w_\Omega\, dx\le \|w_\Omega\|_{L^\infty(\Omega)}.
\]
We now take $s$ as in \eqref{sceltap}, from \eqref{protop} and the definition of torsional rigidity we get
\[
\begin{split}
T(\Omega)&\le \frac{\left(\displaystyle \int_{\Omega^*} w_\Omega^*\,dx\right)^2}{\displaystyle \int_{\Omega^*} |\nabla w_\Omega^*|^2\, dx+c_N\, \mathcal{A}(\Omega)\, s^2}\le T(\Omega^*)\,\left(1+\frac{c_N\, \mathcal{A}(\Omega)\, s^2}{\displaystyle\int_{\Omega^*} |\nabla w_\Omega^*|^2\, dx}\right)^{-1}, 
\end{split}
\]
that is
\[
\frac{T(\Omega^*)}{T(\Omega)}-1\ge \frac{c_N\, \mathcal{A}(\Omega)\, s^2}{\displaystyle\int_{\Omega^*} |\nabla w_\Omega^*|^2\, dx}.
\]
With simple manipulations, by using $\int_{\Omega^*} |\nabla w_\Omega^*|^2\le \int_{\Omega} |\nabla w_\Omega|^2=T(\Omega)$, we get
\begin{equation}
\label{primadip}
T(\Omega^*)-T(\Omega)\ge c_N\, \mathcal{A}(\Omega)\, s^2.
\end{equation}
We then set 
\begin{equation}
\label{parametrip}
s_0=\frac{T(\Omega^*)}{16}\,\frac{N+2}{3\,N+2}\,\mathcal{A}(\Omega),
\end{equation}
and observe that $s_0<T(\Omega^*)/4$. We have to distinguish two cases.
\vskip.2cm\noindent
{\it First case:} $s\ge s_0$. This is the easy case, as from \eqref{primadip} we directly get \eqref{HN21}, with constant
\[
\tau=\frac{c_N}{256}\,T(\Omega^*)^2\,\left(\frac{N+2}{3\,N+2} \right)^2,
\]
and we recall that $c_N$ is as in Lemma \ref{lm:PSbooster}.
\vskip.2cm\noindent
{\it Second case:} $s< s_0$. In this case, by definition \eqref{sceltap} of $s$ we get
\begin{equation}
\label{battaciaria}
\mu(s_0)<\left(1-\frac{1}{4}\,\mathcal{A}(\Omega)\right).
\end{equation}
We also observe that $\mu(s_0)>0$, since $s_0<T(\Omega^*)/4\le 1/2\, \|w_\Omega\|_{L^\infty}$ by the discussion above.
We now want to work with this level $s_0$. We have
\begin{equation}
\label{hansel}
\begin{split}
\int_{\Omega_{s_0}} (w_\Omega-s_0)_+\,dx=\int_\Omega (w_\Omega-s_0)_+\, dx&\ge \int_{\Omega} w_\Omega\, dx-s_0=T(\Omega)-s_0.
\end{split}
\end{equation}
Observe that the right-hand side is strictly positive, since $s_0<T(\Omega)$ thanks to \eqref{parametrip} and \eqref{smezzino}.
We have $(w_\Omega-s_0)_+\in W^{1,2}_0(\Omega_{s_0})$, thus from the variational characterization of $T(\Omega)$, the Saint-Venant inequality and \eqref{hansel}
\[
\begin{split}
T(\Omega)=\frac{\left(\displaystyle\int_{\Omega} w_\Omega\, dx\right)^2}{\displaystyle\int_{\Omega}|\nabla w_\Omega|^2\, dx}\le\frac{\left(\displaystyle\int_{\Omega} w_\Omega\, dx\right)^2}{\displaystyle\int_{\Omega}|\nabla (w_\Omega-s_0)_+|^2\, dx}&\le T(\Omega_{s_0})\,\left(\frac{\displaystyle\int_\Omega w_\Omega\,dx}{\displaystyle\int_\Omega (w-s_0)_+\,dx}\right)^2\\
&\le T(\Omega^*)\,\mu(s_0)^\frac{N+2}{N}\,\left(1-\frac{s_0}{T(\Omega)}\right)^{-2}.
\end{split}
\]
Since $s_0$ satisfies \eqref{battaciaria}, from the previous estimate and again \eqref{smezzino} we can infer
\begin{equation}
\label{gretel}
\left(1-\displaystyle\frac{1}{4}\,\mathcal{A}(\Omega)\right)^{-\frac{N+2}{N}}\,\left(1-\frac{s_0}{T(\Omega^*)}\right)^{2}-1\le \frac{T(\Omega^*)}{T(\Omega)}-1\le \frac{2}{T(\Omega^*)}\,\Big(T(\Omega^*)-T(\Omega)\Big).
\end{equation}
We then observe that
\[
\left(1-\displaystyle\frac{1}{4}\,\mathcal{A}(\Omega)\right)^{-\frac{N+2}{N}}\ge 1+\frac{N+2}{4\,N}\,\mathcal{A}(\Omega)\qquad \mbox{ and }\qquad 
\left(1-\frac{s_0}{T(\Omega^*)}\right)^{2}\ge 1-\frac{2\,s_0}{T(\Omega^*)}.
\]
Thus from \eqref{gretel} we get
\begin{equation}
\label{quasi}
\begin{split}
T(\Omega^*)-T(\Omega)&\ge \frac{T(\Omega^*)}{2}\,\left[\left(1+\frac{N+2}{4\,N}\,\mathcal{A}(\Omega)\right)\, \left(1-\frac{2\,s_0}{T(\Omega^*)}\right)-1\right].
\end{split}
\end{equation}
We now recall the definition \eqref{parametrip} of $s_0$ and finally estimate
\[
\begin{split}
\left(1+\frac{N+2}{4\,N}\,\mathcal{A}(\Omega)\right)& \left(1-\frac{2\,s_0}{T(\Omega^*)}\right)-1\ge \frac{N+2}{8\,N}\,\mathcal{A}(\Omega).
\end{split}
\]
By inserting this in \eqref{quasi} and recalling that $\mathcal{A}(\Omega)<2$, we get \eqref{HN21} with
\[
\tau=T(\Omega^*)\,\frac{N+2}{64\,N}.
\]
This concludes the proof of \eqref{HN2} and thus of the theorem.
\end{proof}

\begin{remark}[Value of the constant $\tau_{N,q}$]
\label{oss:costanta}
In the previous proof $\Omega^*$ is a ball with measure $1$, then from \eqref{torsopalla}
\[
T(\Omega^*)=\frac{\omega_N^{-2/N}}{N\,(N+2)}.
\]
Thus a possible value for the constant $\tau$ in the quantitative Saint-Venant inequality \eqref{HN21} is
\[
\tau=\frac{\omega_N^{-2/N}}{16\,N\,(N+2)}\,\min\left\{1,\,\frac{c_N}{16}\,\frac{\omega_N^{-2/N}}{N}\,\frac{N+2}{(3\,N+2)^2},\,\frac{N+2}{4\,N}\right\},
\]
with $c_N$ as in Lemma \ref{lm:PSbooster}. Consequently, from Proposition \ref{prop:gerarchia}
we get
\[
\tau_{N,q}=(2^{\vartheta}-1)\,|B|^{\frac{2}{N}+\frac{2}{q}-1}\,\lambda_1^q(B)\,\min\left\{\tau\,\frac{|B|^\frac{N+2}{N}}{T(B)},\,\frac{1}{8}\right\},\qquad \vartheta=\frac{2+\displaystyle\frac{2}{q}\,N-N}{N+2}<1.
\]
for the constant appearing in \eqref{HN2}. 
\par
Let us make some comments about the dependence of $\tau_{N,q}$ on the parameter $q$. It is well-known that for $N\ge 3$ we have
\[
\lim_{q\nearrow 2^*} \lambda_{1}^q(\Omega)=\inf_{u\in W^{1,2}_0(\Omega)}\left\{\int_{\mathbb{R}^N} |\nabla u|^2\, dx\, :\,\|u\|_{L^{2^*}(\mathbb{R}^N)}=1\right\}.
\]
The latter coincides with the best constant in the Sobolev inequality on $\mathbb{R}^N$, a quantity which does not depend on the open set $\Omega$. This implies that the constant $\tau_{N,q}$ must converge to $0$ as $q$ goes to $2^*$. From the explicit expression above, we have 
\[
\gamma_{N,q}\simeq (2^{\vartheta}-1)\simeq (2^*-q),\qquad \mbox{ as $q$ goes to $2^*$.}
\]
The conformal case $N=2$ is a little bit different. In this case we have (see \cite[Lemma 2.2]{MR1232190})
\[
\lim_{q\to+\infty} \lambda_{1}^q(\Omega)=0\qquad \mbox{ and }\qquad \lim_{q\to+\infty} q\, \lambda_{1}^q(\Omega)=8\, \pi\, e,
\] 
for every open bounded set $\Omega$. By observing that for $N=2$ we have $\vartheta=1/q$, the asymptotic behaviour of the constant $\gamma_{2,q}$ is then given by
\[
\gamma_{2,q}\simeq \left(2^\frac{1}{q}-1\right)\, \lambda_{1}^q(B)\simeq \frac{1}{q^2},\qquad \mbox{as $q$ goes to $+\infty$.}
\]
\end{remark}

\subsection{The Faber-Krahn inequality in sharp quantitative form}

As simple and general as it is, the previous result is however not sharp.
Indeed, Bhattacharya and Weitsman \cite[Section 8]{MR1647193} and Nadirashvili \cite[page 200]{MR1476715} indipendently conjectured the following.  
\begin{bwn}
There exists a dimensional constant $C>0$ such that
\[
|\Omega|^{2/N}\,\lambda(\Omega)- |B|^{2/N}\,\lambda(B)\ge \frac{1}{C}\, \mathcal{A}(\Omega)^2.
\]
\end{bwn}
After some attempts and intermediate results, this has been proved by Brasco, De Philippis and Velichkov in \cite{MR3357184}. This follows by choosing \(q=2\) in the statement below, which is again valid in the more general case of the first semilinear eigenvalues. We remark that this time the constant appearing in the estimate {\it is not explicit}. However, we can trace its dependence on $q$, which is the same as that of $\gamma_{N,q}$ in Remark \ref{oss:costanta}.
\begin{theorem}
\label{thm:fkstability}
Let $1\le q<2^*$. There exists a constant  \(\gamma_{N,q}\), depending only on the dimension \(N\) and $q$, such that for every open set \(\Omega\subset \mathbb{R}^N\) with finite measure we have
\begin{equation}
\label{eq:fkstability}
|\Omega|^{\frac{2}{N}+\frac{2}{q}-1}\,\lambda_{1}^q(\Omega)-|B|^{\frac{2}{N}+\frac{2}{q}-1}\, \lambda_{1}^q(B)\ge \gamma_{N,q}\, \mathcal{A}(\Omega)^2.
\end{equation}
\end{theorem}
The proof of this result is quite long and technical. We will briefly describe the main ideas and steps of the proof, referring the reader to the original paper \cite{MR3357184} for all the details.
\par
Let us stress that differently from the previous results, {\it the proof of Theorem \ref{thm:fkstability} does not rely on quantitative versions of the P\'olya-Szeg\H{o} principle}, since this technique seems very hard to implement in sharp form (as explained at the beginning of Subsection \ref{sec:24}). On the contrary, the main core is based on the \emph{selection principle} introduced by Cicalese and Leonardi in \cite{MR2980529} to give a new proof of the aforementioned quantitative isoperimetric inequality \eqref{isoFMP}.
\par
The  selection principle turns out to be a very flexible technique and after the paper \cite{MR2980529} it has been applied to a wide variety of geometric problems, see for instance \cite{MR3077924, MR3357860} and \cite{MR3189461}. 

\vskip.2cm\noindent
Let us now  explain the main steps of the proof of Theorem \ref{thm:fkstability}. 
\vskip.2cm\noindent
\centerline{$\boxed{\mbox{\bf Step 1: reduction to the torsional rigidity}}$} 
\vskip.2cm\noindent
We start by observing that by Proposition \ref{prop:gerarchia} it is sufficient to prove the result for the torsional rigidity. In other words, it is sufficient to prove
\begin{equation}
\label{energyintro}
T(B_1)-T(\Omega)\ge \frac{1}{C}\,\mathcal A(\Omega)^2,\qquad\mbox{ for $\Omega\subset\mathbb{R}^N$ such that}\ |\Omega|=|B_1|,
\end{equation}
where \(C\) is a  dimensional constant and $B_1$ is the ball of radius $1$ centered at the origin. 
\vskip.2cm\noindent
\centerline{$\boxed{\mbox{\bf Step 2: sharp stability for nearly spherical sets}}$} 
 \vskip.2cm\noindent
One then observes that if $\Omega$ is a sufficiently smooth perturbation of $B_1$, then \eqref{energyintro} can be proved by means of a second order expansion argument. More precisely, in this step we consider the following class of sets.
\begin{definition}
An open bounded set $\Omega\subset\mathbb{R}^N$ is said {\it nearly spherical of class $C^{2,\gamma}$ parametrized by \(\varphi\)}, if there exists $\varphi\in C^{2,\gamma}(\partial B_1)$ with $\|\varphi\|_{L^\infty}\le 1/2$ and such that $\partial\Omega$ is represented by
\[
\partial\Omega=\big\{x\in\mathbb{R}^N\, :\, x=(1+\varphi(y))\,y, \mbox{ for } y\in\partial B_1\big\}.
\]
\end{definition}
For nearly spherical sets, we then have the following quantitative estimate. The proof relies on a second order Taylor expansion for the torsional rigidity, see \cite{MR1915674} and \cite[Appendix A]{MR3357184}.
\begin{proposition}
\label{prop:fine}
Let $0<\gamma\le 1$. Then there exists \(\delta_1=\delta_1(N,\gamma)>0\) such that if \(\Omega\) is a nearly spherical set of class $C^{2,\gamma}$ parametrized by $\varphi$ with 
\[
\|\varphi\|_{C^{2,\gamma}(\partial B_1)}\le \delta_1,\qquad |\Omega|=|B_1|\qquad \mbox{ and }\qquad x_\Omega:=\fint_\Omega x\,dx=0,
\] 
then
\begin{equation}
\label{ballstab}
T(B_1)-T(\Omega)\ge \frac{1}{32\,N^2}\,\left\|\varphi\right\|^2_{L^2(\partial B_1)}.
\end{equation}
\end{proposition}
\begin{remark}
It is not difficult to see that \eqref{ballstab} implies \eqref{energyintro} for the class of sets under consideration. Indeed, we have
\[
\left\|\varphi\right\|^2_{L^2(\partial B_1)}\ge \frac{1}{N\,\omega_N}\,\left(\int_{\partial B_1}|\varphi|\,d\mathcal{H}^{N-1}\right)^2,
\]
and
\[
\mathcal{A}(\Omega)\le \frac{|\Omega\Delta B_1|}{|\Omega|}=\frac{1}{N\,\omega_N}\,\int_{\partial B_1} |1-(1+\varphi)^N|\,d\mathcal{H}^{N-1}\simeq \frac{1}{\omega_N}\,\int_{\partial B_1} |\varphi|\,d\mathcal{H}^{N-1}.
\]
Let us record that actually inequality \eqref{ballstab} holds true in a stronger form, where the $L^2$ norm of $\varphi$ is replaced by its $W^{1/2,2}(\partial\Omega)$ norm, see \cite[Theorem 3.3]{MR3357184}.
\end{remark}
\vskip.2cm\noindent
\centerline{$\boxed{\mbox{\bf Step 3: reduction to the small asymmetry regime}}$} 
 \vskip.2cm\noindent
This simple step permits to reduce the task to proving \eqref{energyintro} for sets having suitably small Fraenkel asymmetry. Namely, we have the following result.
\begin{proposition}
Let us suppose that there exist $\varepsilon>0$ and $c>0$ such that 
\[
T(B_1)-T(\Omega)\ge c\,\mathcal A(\Omega)^2,\quad\mbox{ for $\Omega$ such that}\quad |\Omega|=|B_1|\ \mbox{ and }\ \mathcal{A}(\Omega)<\varepsilon.
\]
Then \eqref{energyintro} holds true with
\[
\frac{1}{C}=\min\{c,\, \varepsilon\,\tau\}
\]
where $\tau>0$ is the dimensional constant appearing in \eqref{HN21}.
\end{proposition}
\begin{proof}
Once we have Theorem \ref{teo:HN} at our disposal, the proof is straightforward. Indeed, if $\mathcal{A}(\Omega)\ge \varepsilon$, then by \eqref{HN21} we get
\[
T(B_1)-T(\Omega)\ge \tau\,\mathcal{A}(\Omega)^3\ge \varepsilon\,\tau\,\mathcal{A}(\Omega)^2,
\]
as desired. However, let us point out that for the proof of this result it is not really needed the power law relation given by \eqref{HN21}, it would be sufficient to know that \(\mathcal{A}(\Omega)\to 0\) as \(T(B_1)-T(\Omega)\to 0\).
\end{proof}
\vskip.2cm\noindent
\centerline{$\boxed{\mbox{\bf Step 4: reduction to bounded sets}}$} 
 \vskip.2cm\noindent
We can still make a further reduction, namely we can restrict ourselves to prove \eqref{energyintro} for sets with uniformly bounded diameter. This is a consequence of the following expedient result.
\begin{lemma}
\label{lm:bounded}
There exist positive constants \(C=C(N)\), \(\mathcal{T}=\mathcal{T}(N)\) and \(D=D(N)\) such that for every open set  \(\Omega\subset\mathbb{R}^N\) with 
\[
|\Omega|=|B_1|\qquad \mbox{ and }\qquad T(B_1)-T(\Omega)\le \mathcal{T},
\]
we can find another open set \(\widetilde \Omega\subset\mathbb{R}^N\) with 
\[
|\widetilde \Omega|=|B_1|\qquad \mbox{ and }\qquad {\rm diam}(\widetilde \Omega)\le D,
\] 
such that
\begin{equation}
\label{bounded}
\mathcal{A}(\Omega)\le \mathcal{A}(\widetilde \Omega)+C\,\Big(T(B_1)-T(\Omega)\Big)\qquad\mbox{ and }\qquad \Big(T(B_1)-T(\widetilde\Omega)\Big)\le C\, \Big(T(B_1)-T(\Omega)\Big).
\end{equation}
\end{lemma}
The proof of this result is quite tricky and we refer the reader to \cite[Lemma 5.3]{MR3357184}. It is however quite interesting to remark that one of the key ingredients of the proof is the knowledge of some suitable non-sharp quantitative Saint-Venant inequality, where the deficit $T(\Omega)-T(B_1)$ controls a power of the Fraenkel asymmetry. For example, in \cite{MR3357184} a prior result by Fusco, Maggi and Pratelli is used, with exponent $4$ on the asymmetry (see Section \ref{sec:71} below for more comments on their result).
\vskip.2cm\noindent
With the previous result in force, the main output of this step is the following result. 
\begin{proposition}
Let $D$ be the same constant as in Lemma \ref{lm:bounded}. Let us suppose that there exist $c>0$ such that 
\begin{equation}
\label{staidentro}
T(B_1)-T(\Omega)\ge c\,\mathcal A(\Omega)^2,\quad\mbox{ for $\Omega$ such that}\quad |\Omega|=|B_1|\ \mbox{ and }\ \mathrm{diam}(\Omega)\le D.
\end{equation}
Then \eqref{energyintro} holds true. 
\end{proposition}
\begin{proof}
We suppose that $\mathrm{diam}(\Omega)>D$, otherwise there is nothing to prove. Let $\mathcal{T}$ be as in the statement of Lemma \ref{lm:bounded}, we observe that if $T(B_1)-T(\Omega)>\mathcal{T}$,
then \eqref{energyintro} trivially holds true with constant $\mathcal{T}/4$. 
\par
We can thus suppose that $\Omega$ satisfies the hypotheses of Lemma \ref{lm:bounded} and find a new open set $\widetilde\Omega$ for which \eqref{staidentro} holds true. By using \eqref{bounded} and \eqref{staidentro} we get
\[
\begin{split}
\sqrt{T(B_1)-T(\Omega)}\ge \sqrt{\frac{1}{C}\,\Big(T(B_1)-T(\widetilde\Omega)\Big)}&\ge \sqrt{\frac{c}{C}}\,\mathcal{A}(\widetilde\Omega)\\
&\ge\sqrt{\frac{c}{C}}\,\left(\mathcal{A}(\Omega)-C\,\Big(T(B_1)-T(\Omega)\Big)\right).
\end{split}
\]
Since we can always suppose that $T(B_1)-T(\Omega)\le 1$, this shows \eqref{energyintro} for $\Omega$, as desired.
\end{proof}
\vskip.2cm\noindent
\centerline{$\boxed{\mbox{\bf Step 5: sharp stability for bounded sets with small asymmetry}}$} 
 \vskip.2cm\noindent
This is the core of the proof and the most delicate step. Thanks to {\bf Step 1}, {\bf Step 3} and {\bf Step 4}, in order to prove Theorem \ref{thm:fkstability}, we have to prove the following.
\begin{theorem}
\label{thm:stability_lim}
For every $R\ge 2$, there exist two constants $\widehat c=\widehat c(N,R)>0$ and $\widehat\varepsilon=\widehat \varepsilon(N,R)>0$ such that
\begin{equation}
\label{stability_lim}
T(B_1)- T(\Omega)\ge \widehat c\,\mathcal{A}(\Omega)^2, \quad\mbox{ for $\Omega\subset B_R$ such that}\quad |\Omega|=|B_1|\ \mbox{ and }\ \mathcal{A}(\Omega)\le \widehat \varepsilon.
\end{equation}
\end{theorem}
The idea of the proof is to proceed by contradiction. Indeed, let us suppose that \eqref{stability_lim} is false. Thus we may find a sequence of open sets \(\{\Omega_j\}_{j\in\mathbb{N}}\subset B_R(0)\) such that
\begin{equation}
\label{contraintro}
|\Omega_j|=|B_1|, \qquad \varepsilon_j:=\mathcal A (\Omega_j)\to 0\qquad \text{and} \qquad T(B_1)-T(\Omega_j)\le c\, \mathcal A(\Omega_j)^2,
\end{equation}
with \(c>0\) as small as we wish. The idea is to use a variational procedure to replace the sequence $\{\Omega_j\}_{j\in\mathbb{N}}$ with an ``improved''  one  \(\{U_j\}_{j\in\mathbb{N}}\) which still contradicts \eqref{stability_lim} and enjoys some additional smoothness properties. 
\par
In the spirit of the celebrated {\it Ekeland's variational principle}, the idea is to select such a sequence through some penalized minimization problem. Roughly speaking we look for sets \(U_j\) which solve the following 
\begin{equation}
\label{intromin}
\min\left\{T(B_1)-T(\Omega)+\sqrt{\varepsilon_j^2+\eta\,(\mathcal A(\Omega)-\varepsilon_j)^2}\,:\, \Omega\subset B_R,\ |\Omega|=|B_1(0)| \right\},
\end{equation}
where $\eta>0$ is a suitably small parameter, which will allow to get the final contradiction.
\par
One can easily show that the sequence \(U_j\) still contradicts \eqref{stability_lim} and  that \(\mathcal{A}(U_j)\to 0\). Relying on the minimality  of \(U_j\), one then would like to show that the \(L^1\) convergence to \(B_1\) can be improved to a $C^{2,\gamma}$ convergence. If this is the case, then the stability result for smooth nearly spherical sets Proposition \ref{prop:fine} applies and shows that \eqref{contraintro} cannot hold true if \(c\) in \eqref{contraintro} is sufficiently small. 
\par
The key point is thus to prove (uniform) regularity estimates for sets solving \eqref{intromin}. For this, first one would like to get rid of volume constraints applying some sort of Lagrange multiplier principle to show that \(U_j\) solves
\begin{equation}
\label{lagrange}
\min\left\{T(B_1)-T(\Omega)+\sqrt{\varepsilon_j^2+\eta\,(\mathcal A(\Omega)-\varepsilon_j)^2}+\Lambda\,|\Omega|\,:\, \Omega\subset B_R\right\}.
\end{equation}
Then, recalling the formulation \eqref{energiaaa} for $-T(\Omega)$, we can take advantage of the fact that we are considering  a ``min--min'' problem. Thus the previous problem is equivalent to require that the torsion function $w_j:=w_{U_j}$ of $U_j$ minimizes  
\begin{equation}
\label{funzioni}
\begin{split}
\int_{\mathbb{R}^N} |\nabla v|^2\,dx -2\,\int_{\mathbb{R}^N} v\, dx +\Lambda\,\big |\{v>0\}\big|+\sqrt{\varepsilon_j^2+\eta\,(\mathcal A(\{v>0\})-\varepsilon_j)^2},
\end{split}
\end{equation}
among all functions with compact support in $B_R$.
Since we are now facing a perturbed free boundary type problem, we aim to apply the techniques of Alt and Caffarelli \cite{MR0618549} (see also \cite{MR2084257,MR2542718}) to show the regularity of \(\partial U_j=\partial\{u_{j}>0\}\) and  to obtain  the smooth convergence of \(U_j\) to \(B_1\).
\vskip 0.2cm
This is the general strategy, but several non-trivial modifications have to be done to the above sketched proof. A first technical difficulty is that no global Lagrange multiplier principle is available. Indeed, since by scaling
\[
-T(t\, \Omega)-=-t^{N+2}\, T(\Omega)\qquad \mbox{ and }\qquad |t\, \Omega|=t^N\, |\Omega|,\qquad t>0,
\]
by a simple scaling argument one sees that the infimum of the energy in \eqref{lagrange} would be identically  \(-\infty\) in the uncostrained case. This can be fixed by following \cite{MR0826512} and
replacing the term \(\Lambda\, |\Omega|\) with a term of the form \(f(|\Omega|-|B_1|)\), for a suitable strictly increasing function $f$ vanishing at \(0\) only. 
\par
A more serious obstruction is due to the lack of regularity of the Fraenkel asymmetry.  Although solutions to \eqref{funzioni} enjoy some mild regularity properties, we cannot expect  \(\partial \{u_j>0\}\) to be smooth.  Indeed, by formally computing the optimality condition\footnote{That is differentiating the functional along  perturbation of the form \(v_t=u_j\circ ({\rm Id}+tV)\) where \(V\) is a smooth vector field.} of \eqref{funzioni} and assuming that \(B_1\) is the unique optimal ball for the Fraenkel asymmetry of \(\{w_j>0\}\), one gets that \(w_j\) should satisfy
\[
\left|\frac{\partial w_j}{\partial \nu}\right|^2=\Lambda+\frac{\eta\,(\mathcal A(\{w_j>0\})-\varepsilon_j)}{\sqrt{\varepsilon_j^2+\eta\,(\mathcal A(\{w_j>0\})-\varepsilon_j)^2}}\,\big(1_{\mathbb{R}^N\setminus \overline B_1}-1_{B_1}\big)\quad \text{on}\quad \partial \{w_j>0\},
\]
where $1_A$ denotes the characteristic function of a set $A$ and $\nu$ is the outer normal versor.
This means that  the normal derivative of \(w_j\) is discontinuous at points where \(U_j=\{w_j>0\}\) crosses \(\partial B_1\). Since 
classical elliptic regularity implies that if \(\partial U_j\) is \(C^{1,\gamma}\) then \(u_j\in C^{1,\gamma}(\overline{U_j})\), it is clear that the sets \(U_j\) can not enjoy too much smoothness properties. In particular, it seems difficult to obtain the regularity $C^{2,\gamma}$ needed to apply Proposition \ref{prop:fine}. 
\par
To overcome this difficulty, we replace the Fraenkel asymmetry with a new asymmetry functional, which behaves like a squared \(L^2\) distance between the boundaries and whose definition is inspired by \cite{MR1205983}. 
For a bounded set \(\Omega\subset\mathbb{R}^N\), this is defined by
\[
\alpha(\Omega)=\int_{\Omega\Delta B_1(x_\Omega)} \big|1-|x-x_\Omega|\big|\, dx,
\]
where \(x_\Omega\) is the barycenter of \(\Omega\).
Notice that \(\alpha(\Omega)=0\) if and only if \(\Omega\) is a ball of radius \(1\).
This asymmetry is differentiable with respect to the variations needed to compute the optimality conditions (differently from the Fraenkel asymmetry), moreover it enjoys the following crucial properties:
\begin{enumerate}
\item[(i)] there exists a constant \(C_1=C_1(N)>0\) such that for every \(\Omega\)
\begin{equation}
\label{domina!}
C_1\,\alpha(\Omega)\ge |\Omega\Delta B_1(x_\Omega)|^2;
\end{equation}
\vskip.2cm
\item[(ii)] there exists two constants $\delta_2=\delta_2(N)>0$ and $C_{2}=C_2(N)>0$ such that for every nearly spherical set $\Omega$ parametrized by $\varphi$ with $\|\varphi\|_{L^\infty}\le \delta_2$, we have
\begin{equation}
\label{domina!2}
\alpha(\Omega)\le C_2\, \|\varphi\|^2_{L^2(\partial B_1)}.
\end{equation}
\end{enumerate}
By using the strategy described above and replacing $\mathcal{A}(\Omega)$ with $\alpha(\Omega)$, one can obtain the following.
\begin{proposition}[Selection Principle]
Let $R\ge 2$ then there exists $\widetilde \eta=\widetilde \eta(N,R)>0$ such that if \(0<\eta \le \widetilde \eta(N,R)\) and \(\{\Omega_j\}_{j\in\mathbb{N}}\subset\mathbb{R}^N\) verify
\[
|\Omega_j|=|B_1|\qquad\hbox{and}\qquad\varepsilon_j:=\alpha(\Omega_j)\to 0, \qquad  \mbox{  while } \qquad
T(B_1)-T(\Omega_j)\le \eta^4\, \varepsilon_j,
\]
then we can find a sequence of smooth open sets \(\{U_j\}_{j\in\mathbb{N}}\subset B_R\) satisfying: 
\begin{enumerate}
\item[(i)] \(|U_j|=|B_1|\); 
\item[(ii)] \(x_{U_j}=0\);
\item[(iii)] \(\partial U_j\) are converging to \(\partial B_1\) in \(C^k\) for every \(k\);
\item[(iv)] there holds
\[
\limsup_{j \to \infty}\frac{T(B_1)-T(U_j)}{\alpha(U_j)}\le C_3\,\eta,
\]
for some constant \(C_3=C_3(N,R)>0\).
\end{enumerate}
\end{proposition}
In turn, this permits to prove the following alternative version of Theorem \ref{thm:stability_lim}, by following the contradiction scheme sketched above. Indeed, we can apply Proposition  \ref{prop:fine} to the sets $U_j$ and \eqref{domina!2} in order to get
\[
\frac{1}{32\,N^2}\le \limsup_{j \to \infty}\frac{T(B_1)-T(U_j)}{\|\varphi_j\|_{L^2(\partial B)}}\le C_2\,\limsup_{j \to \infty}\frac{T(B_1)-T(U_j)}{\alpha(U_j)}\le C_2\,C_3\,\eta,
\]
where $\varphi_j$ is the parametrization of $\partial U_j$. By choosing $\eta>0$ suitably small, we obtain a contradiction and this proves the following result.
\begin{bis}
For every $R\ge 2$, there exist $\widetilde c=\widetilde c(N,R)>0$ and $\widetilde\varepsilon=\widetilde \varepsilon(N,R)>0$ such that
\[
T(B_1)- T(\Omega)\ge \widetilde c\,\alpha(\Omega), \quad\mbox{ for $\Omega\subset B_R$ such that }\ |\Omega|=|B_1|\ \mbox{ and }\ \alpha(\Omega)\le \widetilde \varepsilon.
\]
\end{bis}
Finally, Theorem \ref{thm:stability_lim} can be now obtained as a consequence of the previous result, by appealing to the properties of $\alpha(\Omega)$. Indeed, by \eqref{domina!} we can assure that $\alpha(\Omega)$ dominates the Fraenkel asymmetry raised to power $2$.
\begin{openpb}[Sharp quantitative Faber-Krahn with explicit constant]
Prove inequality \eqref{eq:fkstability} with a computable constant. Again, it would be sufficient to prove it for the torsional rigidity, still thanks to Proposition \ref{prop:gerarchia}.
\end{openpb}
We conclude this part by remarking that the Fraenkel asymmetry $\mathcal{A}(\Omega)$ is not affected by removing from $\Omega$ a set with positive capacity and zero $N-$dimensional Lebesgue measure, while this is the case for the Faber-Krahn deficit
\[
|\Omega|^{2/N}\,\lambda_1(\Omega)-|B|^{2/N}\,\lambda_1(B).
\] 
In particular, if $\lambda_1(\Omega)=\lambda_1(B)$ and $|\Omega|=|B|$, from Theorem \ref{thm:fkstability} we can only infer that $\Omega$ is a ball {\it up to a set of zero measure}. It could be interesting to have a stronger version of Theorem \ref{thm:fkstability}, where the Fraenkel asymmetry is replaced by a stronger notion of asymmetry, coinciding on sets which differ for a set with zero capacity. Observe that the two asymmetries $d_\mathcal{M}$ and $d_\mathcal{N}$ suffer from the opposite problem, i.e. they are too rigid and affected by removing sets with zero capacity (like points, for example).

\begin{openpb}[Sharp quantitative Faber-Krahn with capacitary asymmetry]
Prove a quantitative Faber-Krahn inequality with a suitable {\rm capacitary asymmetry} $d$, i.e. a scaling invariant shape functional $\Omega\mapsto d(\Omega)$ vanishing on balls only and such that
\[
d(\Omega')=d(\Omega)\qquad \mbox{ if }\quad \mathrm{cap}\,(\Omega\Delta \Omega')=0.
\]
\end{openpb}

\subsection{Checking the sharpness}
\label{sec:sciarpaFK}
The heuristic idea behind the sharpness of the estimate
\[
|\Omega|^{2/N}\,\lambda_1(\Omega)-|B|^{2/N}\,\lambda_1(B)\ge \gamma_{N,2}\,\mathcal{A}(\Omega)^2,
\] 
is quite easy to understand. It is just the standard fact that a smooth function behaves quadratically near a non degenerate minimum point.
\par
Indeed, $\lambda_1$ is twice differentiable in the sense of the {\it shape derivative} (see \cite{ MR2512810}). Then any perturbation of the type $\Omega_t:=X_t(B)$, where $X_t$ is a measure preserving smooth vector field, should provide a Taylor expansion of the form
\[
\lambda_1(\Omega_t)\simeq\lambda_1(B)+O(t^2),\qquad t\ll 1.
\]
since the first derivative of $\lambda_1$ has to vanish at the ``minimum point'' $B$. By observing that the Fraenkel asymmetry satisfies $\mathcal{A}(\Omega_t)=O(t)$, one would prove sharpness of the exponent $2$. 
\par
Rather than giving the detailed proof of the previous argument, we prefer to give an elementary proof of the sharpness, just based on the variational characterization of $\lambda_1^q$ and valid for every $1\le q<2^*$. We believe it to be of independent interest.
\vskip.2cm
We still denote by $B_1\subset\mathbb{R}^N$ the ball with unit radius and centered at the origin. For every $\varepsilon>0$, we consider the $N\times N$ diagonal matrix
\[
M_\varepsilon={\rm diag}\Big((1+\varepsilon), (1+\varepsilon)^{-1}, 1, \dots,1\Big),
\]
and we take the family of ellipsoids $E_\varepsilon=M_\varepsilon\, B_1$. Observe that by construction we have\footnote{We recall that for a $N-$dimensional convex sets having $N$ axes of symmetry, the optimal ball for the Fraenkel asymmetry can be centered at the intersection of these axes, see for example \cite[Corollary 2 \& Remark 6]{MR3050226}.} 
\begin{equation}
\label{ellissoide}
|E_\varepsilon|=|B_1|\qquad \mbox{ and }\qquad \mathcal{A}(E_\varepsilon)=\frac{|E_\varepsilon\Delta B_1|}{|E_\varepsilon|}=O(\varepsilon). 
\end{equation}
Let us fix $q\ge 1$, with a simple change of variables the first semilinear eigenvalue $\lambda^q_1(E_\varepsilon)$ can be written as
\begin{equation}
\label{dream}
\begin{split}
\lambda^q_1(E_\varepsilon)&=\min_{v\in W^{1,2}_0(E_\varepsilon)\setminus\{0\}} \frac{\displaystyle\int_{E_\varepsilon} |\nabla v|^2\, dx}{\displaystyle\left(\int_{E_\varepsilon} |v|^q\, dx\right)^\frac{2}{q}}=\min_{u\in W^{1,2}_0(B_1)\setminus\{0\}}\frac{\displaystyle\int_{B_1}\langle \widetilde M_\varepsilon \nabla u,\nabla u\rangle\,dx}{\displaystyle\left(\int_{B_1} |u|^q\, dx\right)^\frac{2}{q}},
\end{split}
\end{equation}
where $\widetilde M_\varepsilon=M_\varepsilon^{-1}\,M_\varepsilon^{-1}$. We now observe that
\[
\langle \widetilde M_\varepsilon\, \xi,\xi\rangle=\frac{\xi_1^2}{(1+\varepsilon)^2}+(1+\varepsilon)^2\,\xi_2^2+\sum_{i=3}^N \xi_i^2,\qquad \xi\in\mathbb{R}^N,
\]
and by Taylor formula
\[
\frac{1}{(1+\varepsilon)^2}=1-2\,\varepsilon+6\,\int_0^\varepsilon \frac{\varepsilon-s}{(1+s)^4}\,ds\le 1-2\,\varepsilon+3\,\varepsilon^2.
\]
Thus for every $u\in W^{1,2}_0(B_1)$ we obtain
\begin{equation}
\label{silvan}
\begin{split}
\int_{B_1}\langle \widetilde M_\varepsilon \nabla u,\nabla u\rangle\,dx\le \int_{B_1} |\nabla u|^2\,dx&+2\,\varepsilon\,\int_{B_1} \left(|u_{x_2}|^2-|u_{x_1}|^2\right)\,dx\\
&+\varepsilon^2\,\int_{B_1} \left(3\,|u_{x_1}|^2+|u_{x_2}|^2\right)\,dx.
\end{split}
\end{equation}
We now take $U\in W^{1,2}_0(B_1)$ a function which attains the minimum in the definition of $\lambda^q_1(B_1)$, with unit $L^q$ norm. From \eqref{dream} and \eqref{silvan} we get
\[
\lambda^q_1(E_\varepsilon)\le \lambda^q_1(B_1)+2\,\varepsilon\,\int_{B_1} \left(|U_{x_2}|^2-|U_{x_1}|^2\right)\,dx+\varepsilon^2\,\int_{B_1} \left(3\,|U_{x_1}|^2+|U_{x_2}|^2\right)\,dx.
\]
By using that $U$ is radially symmetric (again by P\'olya-Szeg\H{o} principle), it is easy to see that  
\[
\int_{B_1} \left(|U_{x_2}|^2-|U_{x_1}|^2\right)\,dx=0,
\]
and thus finally
\[
\lambda^q_1(E_\varepsilon)\le \lambda^q_1(B_1)+C\,\varepsilon^2.
\]
By recalling \eqref{ellissoide}, this finally shows sharpness of Theorem \ref{thm:fkstability} for every $1\le q< 2^*$.

\section{{\it Intermezzo}: quantitative estimates for the harmonic radius}
\label{sec:3}

In this section we present an application of the quantitative Faber-Krahn inequality to estimates for the so-called {\it harmonic radius}. Apart from being interesting in themselves, some of these results will be useful in the next section.
\begin{definition}[Harmonic radius] 
We denote by $G^\Omega_x$ the Green function of $\Omega$ with singularity at $x\in\Omega$, i.e.
\[
-\Delta G^\Omega_x=\delta_x\quad \mbox{ in }\Omega,\qquad G^\Omega_x=0\quad \mbox{ on }\partial\Omega,
\]
where $\delta_x$ is the Dirac Delta centered at $x$.
We recall that 
\[
G^\Omega_x(y)=\varsigma_N\,\Big(\Gamma_N(|x-y|)-H^\Omega_x(y)\Big),
\]
where:
\begin{itemize}
\item $\varsigma_N$ is the following dimensional constant
\[
\varsigma_2=\frac{1}{2\,\pi},\qquad \varsigma_N=\frac{1}{(N-2)\,N\,\omega_N},\ \mbox{ for }N\ge 3;
\]
\item $\Gamma_N$ is the function defined on $(0,+\infty)$
\[
\Gamma_2(t)=-\log t,\qquad \Gamma_N(t)=t^{2-N},\ \mbox{ for }N\ge 3;
\]
\item $H^\Omega_x$ is the regular part, which solves
\[
\Delta H^\Omega_x=0\quad \mbox{ in }\Omega,\qquad H^\Omega_x=\Gamma_N(|x-\cdot|)\quad \mbox{ on }\partial\Omega.
\]
\end{itemize}  
With the notation above, the {\it harmonic radius of $\Omega$} is defined by
\begin{equation}
\label{harmonicradius}
\mathcal{I}_\Omega:=\sup_{x\in\Omega} \Gamma_N^{-1}(H^\Omega_x(x)).
\end{equation}
We refer the reader to the survey paper \cite{MR1391227} for a comprehensive study of the harmonic radius.
\end{definition}
\begin{remark}[Scaling properties]
It is not difficult to see that $\mathcal{I}_\Omega$ scales like a length. This follows from the fact that for every $t>0$
\begin{equation}
\label{greendilata}
G^{t\,\Omega}_x(y)=t^{2-N}\,G^\Omega_{x/t}\left(\frac{y}{t}\right),\qquad y\not= x\in t\,\Omega.
\end{equation}
Then in dimension $N\ge 3$ we get 
\[
H^{t\,\Omega}_x(y)=t^{2-N}\,H^\Omega_{x/t}\left(\frac{y}{t}\right)\qquad \mbox{ and thus }\qquad \mathcal{I}_{t\,\Omega}:=\sup_{x\in t\,\Omega}\left(t^{2-N}\,H^\Omega_{x/t}\left(\frac{x}{t}\right)\right)^\frac{1}{2-N}=t\,\mathcal{I}_\Omega.
\]
In dimension $2$ we proceed similarly, by observing that from \eqref{greendilata}
\[
H^{t\,\Omega}_x(y)=-\log t+H^\Omega_{x/t}\left(\frac{x}{t}\right).
\]
\end{remark}
For our purposes, it is useful to recall the following spectral inequality.
\begin{theorem}[Hersch-P\'olya-Szeg\H{o} inequality]
Let $\Omega\subset\mathbb{R}^N$ be an open bounded set. Then we have the scaling invariant estimate
\begin{equation}
\label{PH}
\lambda_1(\Omega)\le \frac{\lambda_1(B_1)}{\mathcal{I}_\Omega^2}.
\end{equation}
Equality in \eqref{PH} is attained for balls only.
\end{theorem}
\begin{proof}
Under these general assumptions, the result is due to Hersch and is proved by using {\it harmonic transplation}, a technique introduced in \cite{MR0247572}. The original result by P\'olya and Szeg\H{o} is for $N=2$ and $\Omega$ simply connected, by means of {\it conformal transplantation}. We present their proof below, by referring to \cite[Section 6]{MR1391227} for the general case.
\par
Thus, let us take $N=2$ and $\Omega$ simply connected. Without loss of generality, we can assume $|\Omega|=\pi$. For every $x_0\in\Omega$, we consider the holomorphic isomorphism given by the Riemann Mapping Theorem
\[
f_{x_0}:\Omega\to B_1,
\] 
such that\footnote{We recall that this is uniquely defined, up to a rotation.} $f_{x_0}(x_0)=0$. Then we have the following equivalent characterization for the harmonic radius
\begin{equation}
\label{equivalent}
\mathcal{I}_\Omega=\sup_{x_0\in\Omega} \left|(f_{x_0}^{-1})'(0)\right|=\sup_{x_0\in\Omega} \frac{1}{|f_{x_0}'(x_0)|}.
\end{equation}
Here $f'$ denotes the complex derivative.
Indeed, with the notation above the Green function of $\Omega$ with singularity at $x_0$ is given by 
\[
G^\Omega_{x_0}(y)=-\frac{1}{2\,\pi}\,\log|f_{x_0}(y)|,\qquad y\in\Omega\setminus\{x_0\}.
\]
We can rewrite it as 
\[
G^\Omega_{x_0}(y)=-\frac{1}{2\,\pi}\, \log|f_{x_0}(y)-f_{x_0}(x_0)|=-\frac{1}{2\,\pi}\log|y-x_0|-\frac{1}{2\,\pi}\log\frac{|f_{x_0}(y)-f_{x_0}(x_0)|}{|y-x_0|}.
\]
By recalling the definition \eqref{harmonicradius} of harmonic radius, we get
\[
\mathcal{I}_\Omega=\sup_{x_0\in\Omega} \left\{\lim_{y\to x_0}\exp\left(-\log\frac{|f_{x_0}(y)-f_{x_0}(x_0)|}{|y-x_0|}\right)\right\}=\sup_{x_0\in\Omega} \frac{1}{|f'_{x_0}(x_0)|},
\]
which proves \eqref{equivalent}. 
\par
We now prove \eqref{PH}. Let $u\in W^{1,2}_0(B_1)$ be the first positive Dirichlet eigenfunction of $B_1$, with unit $L^2$ norm. For $x_0\in\Omega$, we consider $f_{x_0}:\Omega\to B_1$ as above, then we set
\[
v=u\circ f.
\]
By conformality we preserve the Dirichlet integral, i.e.
\[
\int_\Omega |\nabla v|^2\,dx=\int_{B_1} |\nabla u|^2\,dx=\lambda_1(B_1).
\]
On the other hand, by the change of variable formula we have
\[
\int_\Omega |v|^2\,dx=\int_{B_1} |u|^2\,|(f^{-1}_{x_0})'|^2\,dx.
\]
We now observe that $|(f^{-1}_{x_0})'|^2$ is sub-harmonic, thus the function
\begin{equation}
\label{subarmo}
\Phi(\varrho)=\frac{1}{2\,\pi\,\varrho}\,\int_{\{|x|=\varrho\}} |(f^{-1}_{x_0})'|^2\,d\mathcal{H}^1,
\end{equation}
is non-decreasing. In particular, we have
\[
\Phi(\varrho)\ge \Phi(0)=|(f^{-1}_{x_0})'(0)|^2.
\]
Thus we obtain
\[
\begin{split}
\int_\Omega |v|^2\,dx&=\int_\Omega |u|^2\,|(f^{-1}_{x_0})'|^2\,dx=2\,\pi\,\int_0^1 u^2\,\varrho\,\Phi(\varrho)\,d\varrho\\
&\ge \left(2\,\pi\,\int_0^1 u^2\,\varrho\,d\varrho\right)\,|(f^{-1}_{x_0})'(0)|^2=|(f^{-1}_{x_0})'(0)|^2,
\end{split}
\]
since $u$ has unitary $L^2$ norm.
By using the variational characterization of $\lambda_1(\Omega)$, this finally shows
\[
|(f^{-1}_{x_0})'(0)|^2\,\lambda_1(\Omega)\le \lambda_1(B_1).
\]
By taking the supremum over $\Omega$ and using \eqref{equivalent}, we get the conclusion.
\end{proof}
\begin{remark}[Conformal radius]
Historically, the quantity 
\[
\max_{x_0\in\Omega} \left|(f_{x_0}^{-1})'(0)\right|=\max_{x_0\in\Omega} \frac{1}{|f_{x_0}'(x_0)|},
\]
has been first introduced under the name {\it conformal radius of $\Omega$}. The definition of harmonic radius is due to Hersch \cite{MR0247572}, as we have seen this gives a genuine extension to general sets of the conformal radius.
\end{remark}

Among open sets with given measure, the harmonic radius is maximal on balls. By recalling that for a ball the harmonic radius coincides with the radius {\it tout court}, we thus have the scaling invariant estimate
\begin{equation}
\label{isoradio}
\frac{|\Omega|^{2/N}}{\mathcal{I}_\Omega^2}\ge \omega_N^{2/N}.
\end{equation}
This can be deduced by joining \eqref{PH} and the Faber-Krahn inequality. If we replace the latter by Theorem \ref{thm:fkstability}, we get a quantitative version of \eqref{isoradio}. This is the content of the next result.
\begin{corollary}[Stability of the harmonic radius]
Let $\Omega\subset\mathbb{R}^N$ be an open bounded set. Then we have
\begin{equation}
\label{confra}
\frac{|\Omega|^{2/N}}{\mathcal{I}_\Omega^2}-\omega_N^{2/N}\ge c\,\mathcal{A}(\Omega)^2,
\end{equation}
for some constant $c>0$.
\end{corollary}
\begin{proof}
We multiply both sides of \eqref{PH} by $|\Omega|^{2/N}/\omega_N^{2/N}$, then we get
\[
\frac{|\Omega|^{2/N}\,\lambda_1(\Omega)}{\omega_N^{2/N}\,\lambda_1(B_1)}-1\le \frac{1}{\omega_N^{2/N}}\,\left(\frac{|\Omega|^{2/N}}{\mathcal{I}_\Omega^2}-\omega_N^{2/N}\right).
\]
By recalling that $\omega_N^{2/N}\,\lambda_1(B_1)$ is a universal constant and using the sharp quantitative Faber-Krahn inequality of Theorem \ref{thm:fkstability}, we get the conclusion.
\end{proof}
For simply connected sets in the plane, the previous result has an interesting geometrical consequence, which will be exploited in Section \ref{sec:4}. Indeed, observe that with the notation above we have 
\[
|\Omega|=\int_{B_1} |(f^{-1}_{x_0})'|^2\,dx\ge \pi\,|(f^{-1}_{x_0})'(0)|^2,
\]
where we used again monotonicity of the function \eqref{subarmo}.
If we assume for simplicity that $|\Omega|=\pi$, thus we get
\[
\frac{1}{|f'_{x_0}(x_0)|}=|(f^{-1}_{x_0})'(0)|\le 1
\]
with equality if $\Omega$ is a disc. If $\Omega$ is not a disc, then the inequality is strict and we can add a remainder term. In other words, {\it the local stretching at the origin of the conformal map $f^{-1}_{x_0}$ can tell whether $\Omega$ is a disc or not}.
This is the content of the next result.
\begin{corollary}
\label{lm:kindof}
Let $\Omega\subset\mathbb{R}^2$ be an open bounded simply connected set such that $|\Omega|=\pi$. For every $x_0\in\Omega$, we consider the holomorphic isomorphism
\[
f_{x_0}:\Omega\to B_1,
\] 
such that $f_{x_0}(x_0)=0$. For every $x_0\in\Omega$ we have
\[
\frac{1}{|f'_{x_0}(x_0)|}=|(f_{x_0}^{-1})'(0)|\le \sqrt{1-\frac{1}{C}\, \mathcal{A}(\Omega)^2},
\]
for some $C>4$.
\end{corollary}
\begin{proof}
We observe that from \eqref{confra} we get
\[
\frac{1}{\mathcal{I}_\Omega^2}\ge 1+\frac{c}{\pi}\,\mathcal{A}(\Omega)^2,
\]
where we used that $|\Omega|=\pi$.
From this, with simple manipulations we get
\[
\mathcal{I}_\Omega^2\le 1-\frac{1}{C}\,\mathcal{A}(\Omega)^2.
\]
It is now sufficient to use the characterization \eqref{equivalent} to conclude.
\end{proof}

\section{Stability for the Szeg\H{o}-Weinberger inequality}
\label{sec:4}

\subsection{A quick overview of the Neumann spectrum}
In the case of homogeneous Neumann boundary conditions, the first eigenvalue $\mu_1(\Omega)$ is always $0$ and corresponds to constant functions. This reflects the fact that the Poincar\'e inequality
\[
c\, \int_\Omega |u|^2\, dx\le \int_\Omega |\nabla u|^2\, dx ,\qquad u\in W^{1,2}(\Omega),
\]
can hold only in the trivial case $c=0$. For an open set $\Omega\subset\mathbb{R}^N$ with finite measure, 
we define its {\it first non trivial Neumann eigenvalue} by
\[
\mu_2(\Omega):=\inf_{u\in W^{1,2}(\Omega)\setminus\{0\}}\left\{\frac{\displaystyle\int_\Omega |\nabla u|^2\, dx}{\displaystyle\int_\Omega |u|^2\, dx}\, :\, \int_\Omega u\, dx=0\right\}.
\] 
In other words, this is the sharp constant in the Poincar\'e-Wirtinger inequality
\[
c\,\int_{\Omega} \left|u-\fint_\Omega u\right|^2\,dx\le \int_\Omega |\nabla u|^2\,dx,\qquad u\in W^{1,2}_0(\Omega).
\]
When $\Omega\subset\mathbb{R}^N$ has Lipschitz boundary, the embedding $W^{1,2}(\Omega)\hookrightarrow L^2(\Omega)$ is compact (see \cite[Theorem 5.8.2]{MR0482102}) and the infimum above is attained.
In this case the Neumann Laplacian has a discrete spectrum $\{\mu_1(\Omega),\mu_2(\Omega),\dots\}$. The successive Neumann eigenvalues can be defined similarly, that is $\mu_k(\Omega)$ is obtained by minimizing the same Rayleigh quotient, among functions orthogonal (in the $L^2(\Omega)$ sense) to the first $k-1$ eigenfunctions.
\par
If $\Omega$ has $k$ connected components, we have $\mu_1(\Omega)=\dots=\mu_k(\Omega)=0$, with corresponding eigenfunctions given by a constant function on each connected component of $\Omega$. We still have the scaling property
\[
\mu_k(t\,\Omega)=t^{-2}\, \mu_k(\Omega),\qquad t>0,
\]
and there holds the Szeg\H{o}-Weinberger inequality\footnote{We point out that Szeg\H{o}-Weinberger inequality holds for every open set with finite measure, without smoothness assumptions. In other words, the proof does not use neither discreteness of the Neumann spectrum of $\Omega$, nor that the infimum in the definition of $\mu_2(\Omega)$ is attained.}
\begin{equation}
\label{segowine}
|\Omega|^{2/N}\,\mu_2(\Omega)\le |B|^{2/N}\,\mu_2(B),
\end{equation}
with equality if and only if $\Omega$ is a ball. 
\par
For a ball $B_r$ of radius $r$, $\mu_2(B_r)$ has multiplicity $N$, that is $\mu_2(B_r)=\dots=\mu_{N+1}(B_r)$. This value can be explicitely computed, together with its corresponding eigenfunctions. Indeed, these are given by (see \cite{MR2310200})
\begin{equation}
\label{autopalle}
\xi_i(x):=|x|^\frac{2-N}{2}\,J_\frac{N}{2}\left(\frac{\beta_{N/2,1}|x|}{r}\right)\,\frac{x_i}{|x|},\qquad i=1,\,\dots,\,N.
\end{equation}
Here $J_{N/2}$ is still a Bessel function of the first kind, 
while $\beta_{N/2,1}$ denotes the first positive zero of the derivative of $t\mapsto t^{(2-N)/2}\,J_{N/2}(t)$, i.e. it verifies
\[
\beta_{N/2,1}\,J'_\frac{N}{2}(\beta_{N/2,1})+\left(\frac{2-N}{2}\right) J_\frac{N}{2}(\beta_{N/2,1})=0\,.
\]
Observe in particular that the radial part of $\xi_i$ 
\begin{equation}
\label{Q}
\varphi_N(|x|):=|x|^{1-\frac{N}{2}}\,J_{N/2}\left(\frac{\beta_{N/2,1}|x|}{r}\right)\,,
\end{equation}
satisfies the ODE (of Bessel type)
\[
g''(t)+\frac{N-1}{t}\,g'(t)+\left(\mu_2(B_r)-\frac{N-1}{t^2}\right)\,g(t)=0,
\]
and one can compute
\[
\mu_2(B_r)=\left(\frac{\beta_{N/2,1}}{r}\right)^2.
\]
Finally, we recall that in dimension $N=2$ inequality \eqref{segowine} can be sharpened. Namely, for every $\Omega\subset\mathbb{R}^2$ simply connected open set we have
\begin{equation}
\label{sego}
\frac{1}{|\Omega|}\,\left(\frac{1}{\mu_2(\Omega)}+\frac{1}{\mu_3(\Omega)}\right)\ge \frac{1}{|B|}\,\left(\frac{1}{\mu_2(B)}+\frac{1}{\mu_3(B)}\right),
\end{equation}
where $B\subset\mathbb{R}^2$ is any open disc.
This result has been proved by Szeg\H{o} in \cite{MR0061749} by means of conformal maps, we will recall his proof below.
 By recalling that for a disc $\mu_2=\mu_3$, from \eqref{sego} we immediately get \eqref{segowine} for simply connected sets in $\mathbb{R}^2$.

\begin{remark}
The higher dimensional analogue of \eqref{sego} would be
\[
\frac{1}{|\Omega|^{2/N}}\,\sum_{k=2}^{N+1} \frac{1}{\mu_k(\Omega)}\ge \frac{1}{|B|^{2/N}}\,\sum_{k=2}^{N+1} \frac{1}{\mu_k(B)}.
\]
However, the validity of such an inequality is still an open problem, see \cite[page 106]{MR2251558}.
\end{remark}

\subsection{A two-dimensional result by Nadirashvili}

One of the first quantitative improvements of the Szeg\H{o}-Weinberger inequality was due to Nadirashvili, see \cite{MR1476715}. Even if his result is limited to simply connected sets in the plane, this is valid for the stronger inequality \eqref{sego}. We reproduce the original proof, up to some modifications (see Remark \ref{oss:capacity} below). We will also highlight a quicker strategy suggested to us by Mark S. Ashbaugh (see Remark \ref{oss:mark} below).
\begin{theorem}[Nadirashvili]
\label{thm:nikolaisego}
There exists a constant $C>0$ such that for every $\Omega\subset\mathbb{R}^2$ smooth simply connected open set we have
\begin{equation}
\label{nikolaisego}
\frac{1}{|\Omega|}\,\left(\frac{1}{\mu_2(\Omega)}+\frac{1}{\mu_3(\Omega)}\right)-\frac{1}{|B|}\,\left(\frac{1}{\mu_2(B)}+\frac{1}{\mu_3(B)}\right)\ge \frac{1}{C}\, \mathcal{A}(\Omega)^2.
\end{equation}
Here $B\subset\mathbb{R}^2$ is any open disc.
\end{theorem}
\begin{proof}
The proof of \eqref{nikolaisego} introduces some quantitative ingredients in the original proof by Szeg\H{o}. 
For the reader's convenience, it is thus useful to recall at first the proof of \eqref{sego}.
\par
By scale invariance, we can suppose that $|\Omega|=|B|=\pi$ and we may take the disc $B$ to be centered at the origin.
From \eqref{autopalle} above, we know that 
\[
\xi_1(x)=c\,J_{1}(\beta_{1,1}\,|x|)\,\frac{x_1}{|x|}\qquad \mbox{ and }\qquad \xi_2(x)=c\,J_{1}(\beta_{1,1}\,|x|)\,\frac{x_2}{|x|},
\]
are two linearly independent Neumann eigenfunctions in $B$, corresponding to $\mu_2(B)=\mu_3(B)$. The normalization constant $c$ is chosen so to guarantee that $\xi_1$ and $\xi_2$ have unit $L^2$ norm.
\par
Since $\Omega\subset\mathbb{R}^2$ is simply connected, given $x_0\in \Omega$ by the Riemann Mapping Theorem there exists an analytic isomorphism $f_{x_0}:\Omega\to B$ such that $f_{x_0}(x_0)=0$. For notational simplicity, we will omit the index $x_0$ and simply write $f$. Szeg\H{o} proved that we can choose $x_0\in\Omega$ in such a way that if we set $v_i=\xi_i\circ f$ ($i=1,2$) then
\[
\int_\Omega v_i\,dx=0,\qquad i=1,2.
\]
Then if we set $h=f^{-1}$ we have 
\begin{equation}
\label{olomorfa}
\int_\Omega |v_i|^2\,dx=\int_B \xi_i^2\,|h'|^2\,dx,\qquad \int_\Omega |\nabla v_i|^2\,dx=\int_B |\nabla \xi_i|^2\,dx,\qquad i=1,2,
\end{equation}
where $h'$ denotes the complex derivative. Also observe that by conformality we have
\[
\int_\Omega \langle \nabla v_1,\nabla v_2\rangle\,dx=\int_B \langle \nabla \xi_1,\nabla \xi_2\rangle\,dx=0.
\]
By recalling the following variational formulation for sum of inverses of Neumann eigenvalues (see for example \cite[Theorem 1]{MR1251868})
\[
\frac{1}{\mu_2(\Omega)}+\frac{1}{\mu_3(\Omega)}=\max_{u\in W^{1,2}(\Omega)\setminus\{0\}}\left\{\frac{\displaystyle\int_\Omega |u_1|^2\,dx}{\displaystyle \int_\Omega |\nabla u_1|^2\,dx}+\frac{\displaystyle\int_\Omega |u_2|^2\,dx}{\displaystyle \int_\Omega |\nabla u_2|^2\,dx}\, :\, \begin{array}{c}\displaystyle \int_{\Omega} u_1\,dx=\int_\Omega u_2\,dx=0\\
\displaystyle\int_\Omega \langle \nabla u_1,\nabla u_2\rangle\,dx=0
\end{array}\right\},
\]
and using that $\mu_2(B)=\mu_3(B)$, from \eqref{olomorfa} we get
\begin{equation}
\label{szego}
\begin{split}
\frac{1}{\mu_2(\Omega)}+\frac{1}{\mu_3(\Omega)}\ge \frac{\displaystyle\int_\Omega |v_1|^2\,dx}{\displaystyle \int_\Omega |\nabla v_1|^2\,dx}+\frac{\displaystyle \int_\Omega |v_2|^2\,dx}{\displaystyle\int_\Omega |\nabla v_2|^2\,dx}&=\frac{\displaystyle\int_B |\xi_1|^2\,|h'|^2\,dx}{\displaystyle \mu_2(B)}+\frac{\displaystyle \int_B |\xi_2|^2\,|h'|^2\,dx}{\displaystyle \mu_3(B)}\\
&= \frac{\displaystyle\int_B \Big(|\xi_1|^2+|\xi_2|^2\Big)\,|h'|^2\,dx}{\displaystyle \mu_2(B)}.
\end{split}
\end{equation}
Since $h'$ is holomorphic, the function $|h'|^2$ is subharmonic, thus 
\[
r\mapsto \Phi(r):=\fint_{\{|x|=r\}} |h'|^2\,d\mathcal{H}^1,
\]
is a monotone nondecreasing function.
The same is true for the radial function 
\[
|\xi_1|^2+|\xi_2|^2=c^2 J_{1}(\beta_{1,1}\,|x|)^2,
\]
thus by Lemma \ref{lm:hersch} we have
\begin{equation}
\label{notte}
\begin{split}
\int_B \Big(|\xi_1|^2+|\xi_2|^2\Big)\,|h'|^2\,dx&=2\,\pi\,\int_0^1 \Big(|\xi_1|^2+|\xi_2|^2\Big)\,\Phi(\varrho)\,\varrho\,d\varrho\\
&\ge 2\,\pi\,\frac{\displaystyle\int_0^1 \Big(|\xi_1|^2+|\xi_2|^2\Big)\,\varrho\,d\varrho}{\displaystyle\int_0^1 \varrho\,d\varrho}\,\int_0^1\, \Phi(\varrho)\,\varrho\,d\varrho \\
&=2\,\int_0^1 \Big(|\xi_1|^2+|\xi_2|^2\Big)\,\varrho\,d\varrho\,\left(\int_0^1\int_{\{|x|=\varrho\}} |h'|^2\,d\mathcal{H}^1\,d\varrho\right)\,\\
&=2\,\pi\,\int_0^1 \Big(|\xi_1|^2+|\xi_2|^2\Big)\,\varrho\,d\varrho=\int_B \Big(|\xi_1|^2+|\xi_2|^2\Big)\,dx=2,
\end{split}
\end{equation}
where we used that
\[
\int_0^1\int_{\{|x|=\varrho\}} |h'|^2\,d\mathcal{H}^1\,d\varrho=\int_B |h'|^2\,dx=\pi.
\] 
By using the previous estimate in \eqref{szego}, we finally get \eqref{sego}.
\vskip.2cm\noindent
We now come to the proof of \eqref{nikolaisego}. By using Corollary \ref{lm:kindof} from the previous section, we get
\begin{equation}
\label{primocoeff}
|h'(0)|^2\le 1-\frac{1}{C}\,\mathcal{A}(\Omega)^2.
\end{equation}
Since $h$ is analytic, we have
\[
h'(z)=\sum_{n=1}^\infty n\,a_n\,z^{n-1},
\]
and thus
\[
\Phi(\varrho)=\fint_{\{|x|=\varrho\}} |h'|^2\,d\mathcal{H}^1=\sum_{n=1}^\infty n^2\,a_n^2\,\varrho^{2\,(n-1)}.
\]
The latter can be rewritten as
\[
\Phi(\varrho)=\sum_{n=0}^\infty \alpha_n\,\varrho^n,\qquad \mbox{ where } \alpha_n=\begin{cases}\left(\frac{n+2}{2}\right)^2\,a^2_{\frac{n+2}{2}}, & n \mbox{ even},\\
0, & n \mbox{ odd},
\end{cases}
\]
and from \eqref{primocoeff}
\[
\alpha_0=a_1^2=|h'(0)|^2\le 1-\frac{1}{C}\,\mathcal{A}(\Omega)^2=2\,\left(1-\frac{1}{C}\,\mathcal{A}(\Omega)^2\right)\,\int_0^1 \Phi(\varrho)\,\varrho\,d\varrho.
\]
We can thus apply Lemma \ref{lm:hersch+}, with the choices
\[
f=|\xi_1|^2+|\xi_2|^2=c^2 J_1^2,\qquad \Phi(\varrho)=\fint_{\{|x|=\varrho\}} |h'|^2\,d\mathcal{H}^1\quad \mbox{ and }\quad \gamma=2\,\left(1-\frac{1}{C}\,\mathcal{A}(\Omega)^2\right).
\]
Thus in place of \eqref{notte} we now obtain
\[
\begin{split}
\int_B \Big(|\xi_1|^2+|\xi_2|^2\Big)\,|h'|^2\,dx&\ge 2\,\pi\,\frac{\displaystyle\int_0^1 \Big(|\xi_1|^2+|\xi_2|^2\Big)\,\varrho\,d\varrho}{\displaystyle\int_0^1 \varrho\,d\varrho}\,\int_0^1\, \Phi(\varrho)\,\varrho\,d\varrho\\
&+2\,c'\,\mathcal{A}(\Omega)^2\,\int_0^1 \Phi(\varrho)\,\varrho\,d\varrho=2+c'\,\mathcal{A}(\Omega)^2.
\end{split}
\]
By using this improved estimate in \eqref{szego}, we get
\[
\frac{1}{\mu_2(\Omega)}+\frac{1}{\mu_3(\Omega)}\ge \frac{2}{\mu_2(B)}+\frac{c'}{\mu_2(B)}\,\mathcal{A}(\Omega)^2,
\]
which concludes the proof.
\end{proof}
\begin{remark}
\label{oss:capacity}
The crucial point of the previous proof is to obtain estimate \eqref{primocoeff} on $h'(0)=(f^{-1})'(0)$. The argument we used to obtain it is slightly different with respect to the original one by Nadirashvili. The latter exploits a stability result of Hansen and Nadirashvili (see \cite[Corollary 2]{MR1222458}) for the logarithmic capacity in dimension $N=2$, which assures that\footnote{As explained in the Introduction of \cite{MR1243100}, for {\it connected} open sets in $\mathbb{R}^2$ inequality \eqref{incapacità} follows from an inequality linking capacity and moment of inertia which can be found in the book \cite{MR0043486}. This observation is attributed to Keady. In \cite{MR1222458} the result is extended to general open sets in $\mathbb{R}^2$.}
\begin{equation}
\label{incapacità}
\mathrm{Cap}(\Omega)-\mathrm{Cap}(B)\ge c\, \mathcal{A}(\Omega)^2,\qquad \mbox{ if }|\Omega|=|B|.
\end{equation}
Here on the contrary we rely on the stability estimate of Corollary \ref{lm:kindof}, which in turn is a consequence of the quantitative Faber-Krahn inequality, as we saw in Section \ref{sec:2}.
\end{remark}
\begin{remark}[An overlooked inequality]
\label{oss:mark}
Inequality \eqref{sego} in turn can be sharpened. Indeed, in \cite{MR0295172} Hersch and Monkewitz have shown that there exists a constant $c>0$ such that for every $\Omega\subset\mathbb{R}^2$ simply connected open set we have
\begin{equation}
\label{HMsego}
\frac{1}{|\Omega|}\,\left(\frac{1}{\mu_2(\Omega)}+\frac{1}{\mu_3(\Omega)}+\frac{c}{\lambda_1(\Omega)}\right)\ge \frac{1}{|B|}\,\left(\frac{1}{\mu_2(B)}+\frac{1}{\mu_3(B)}+\frac{c}{\lambda_1(B)}\right).
\end{equation}
By using this inequality, we can provide a quicker proof of Theorem \ref{thm:nikolaisego}.
Indeed, let us suppose for simplicity that $|\Omega|=1$, from \eqref{HMsego} we get
\[
\left(\frac{1}{\mu_2(\Omega)}+\frac{1}{\mu_3(\Omega)}\right)-\left(\frac{1}{\mu_2(\Omega^*)}+\frac{1}{\mu_3(\Omega^*)}\right)\ge \frac{c}{\lambda_1(\Omega^*)\,\lambda_1(\Omega)}\,\Big(\lambda_1(\Omega)-\lambda_1(\Omega^*)\Big),
\]
where $\Omega^*$ is a disc such that $|\Omega^*|=|\Omega|=1$. We now observe that if $\lambda_1(\Omega)\ge 2\,\lambda_1(\Omega^*)$, the right-hand side above can be bounded from below as follows
\[
\frac{c}{\lambda_1(\Omega^*)\,\lambda_1(\Omega)}\,\Big(\lambda_1(\Omega)-\lambda_1(\Omega^*)\Big)\ge \frac{c}{2\,\lambda_1(\Omega^*)}\ge \frac{c}{8\,\lambda_1(\Omega^*)}\,\mathcal{A}(\Omega)^2,
\]
where we used that $\mathcal{A}(\Omega)<2$.
If on the contrary $\lambda_1(\Omega)<2\,\lambda_1(\Omega^*)$, then from the sharp quantitative Faber-Krahn inequality (Theorem \ref{thm:fkstability}) we get
\[
\frac{c}{\lambda_1(\Omega^*)\,\lambda_1(\Omega)}\,\Big(\lambda_1(\Omega)-\lambda_1(\Omega^*)\Big)\ge \frac{c'}{\lambda_1(\Omega^*)^2}\,\mathcal{A}(\Omega)^2.
\] 
In conclusion, we can infer the existence of a constant $c''>0$ such that
\[
\left(\frac{1}{\mu_2(\Omega)}+\frac{1}{\mu_3(\Omega)}\right)-\left(\frac{1}{\mu_2(\Omega^*)}+\frac{1}{\mu_3(\Omega^*)}\right)\ge c''\mathcal{A}(\Omega)^2,
\]
thus proving Theorem \ref{thm:nikolaisego}. We thank Mark S. Ashbaugh for kindly pointing out the reference \cite{MR0295172}.
\end{remark}

\subsection{The Szeg\H{o}-Weinberger inequality in sharp quantitative form}

From Theorem \ref{thm:nikolaisego}, one can easily get a quantitative improvement of the Szeg\H{o}-Weinberger inequality, in the case of simply connected sets in the plane. 
For general open sets in any dimension, we have the following result proved by Brasco and Pratelli in \cite[Theorem 4.1]{MR2899684}.
\begin{theorem}
\label{th:quanto_sw}
For every $\Omega\subset\mathbb{R}^N$ open set with finite measure, we have
\begin{equation}
\label{quantSW}
|B|^{2/N}\,\mu_2(B)\, -|\Omega|^{2/N}\,\mu_2(\Omega)\ge \rho_N\, \mathcal{A}(\Omega)^2,
\end{equation}
where $\rho_N>0$ is an explicit dimensional constant (see Remark \ref{oss:costanteSW} below).
\end{theorem}
\begin{proof}
Here as well, we first recall the proof of \eqref{segowine}.
As always, we denote by $\Omega^*$ the ball centered at the origin and such that $|\Omega^*|=|\Omega|$. Since \eqref{quantSW} is scaling invariant, we can suppose that $|\Omega|=\omega_N$, i.e. the radius of $\Omega^*$ is $1$. Observe that the eigenfunctions $\xi_i$ of $\Omega^*$ defined in \eqref{autopalle} have the following property, which will be crucially exploited:
\[
x\mapsto \sum_{i=1}^N |\xi_i(x)|^2\qquad \mbox{ and }\qquad x\mapsto \sum_{i=1}^N |\nabla \xi_i(x)|^2\qquad \mbox{ are monotone radial functions}.
\]
Indeed, we have
\begin{equation}
\label{radiali}
\sum_{i=1}^N |\xi_i(x)|^2=\varphi_N(|x|)^2 \qquad \mbox{ and }\qquad \sum_{i=1}^N |\nabla \xi_i(x)|^2=\varphi'_N(|x|)^2+(N-1)\, \frac{\varphi_N(|x|)^2}{|x|^2},
\end{equation}
and the first one is radially increasing, while the second is decreasing.
Moreover, since each $\xi_i$ is an eigenfunction of the ball, we have
\[
\mu_2(\Omega^*)\, \int_{\Omega^*} |\xi_i|^2\,dx=\int_{\Omega^*} |\nabla \xi_i|^2\,dx,\qquad i=1,\dots,N.
\]
If we sum the previous identities and use \eqref{radiali}, we thus end up with
\begin{equation}
\label{palla}
\mu_2(\Omega^*)=\frac{\displaystyle\int_{\Omega^*} \left[\varphi'_N(|x|)^2+(N-1)\,\frac{\varphi_N(|x|)^2}{|x|^2}\right]\, dx}{\displaystyle\int_{\Omega^*}\varphi_N(|x|)^2\, dx}.
\end{equation}
We then extend $\varphi_N$ to the whole $[0,+\infty)$ as follows
\[
\phi_N(t)=\left\{\begin{array}{rl}
\varphi_N(t),& 0\le t\le 1,\\
\varphi_N(1),& t>1,\\
\end{array}
\right.
\]
and consider the new functions defined on $\mathbb{R}^N$
\[
\Xi_i(x)=\phi_N(|x|)\,\frac{x_i}{|x|},\qquad i=1,\, \dots\, ,\,N.
\]
Observe that if we define 
\[
F_N(t)=\int_0^t \phi_N(s)\,ds,\qquad t\ge 0,
\]
this is a $C^1$ convex increasing function, which diverges at infinity. This means that the function 
\[
x\mapsto \int_\Omega F_N(|x-y|)\, dy,
\] 
admits a global minimum point $x_0\in\mathbb{R}^N$ and thus
\[
(0,\dots,0)=\int_\Omega F_N'(|x_0-y|)\, \frac{x_0-y}{|x_0-y|}\, dy=\left(\int_\Omega \Xi_1(x_0-y)\, dy,\dots,\,\int_\Omega \Xi_N(x_0-y)\, dy\right). 
\]
Thus it is always possible to choose the origin of the coordinate axes in such a way that\footnote{We avoid here the original argument based on the Brouwer Fixed Point Theorem.} 
\[
\int_\Omega \Xi_i(x)\, dx=0,\ i=1,\dots,N.
\] 
By making such a choice for the origin, the functions $\Xi_i$ can be used to estimate $\mu_2(\Omega)$ and we can infer
\[
\mu_2(\Omega)\le \frac{\displaystyle\int_\Omega |\nabla \Xi_i|^2\, dx}{\displaystyle\int_\Omega \Xi_i^2\, dx},\qquad   i=1,\,\dots,\,N\,.
\]
Again, a summation over $i=1,\dots,N$ yields
\[
\mu_2(\Omega)\le \frac{\displaystyle\sum\limits_{i=1}^N \int_\Omega |\nabla \Xi_i|^2\, dx}{\displaystyle\sum\limits_{i=1}^N\int_\Omega \Xi^2_i\, dx},
\]
and the summation trick makes the angular variables disappear and one ends up with
\begin{equation}
\label{nonpalla}
\mu_2(\Omega)\leq \frac{\displaystyle\int_\Omega \left[\phi'_N(|x|)^2+(N-1)\,\frac{\phi_N(|x|)^2}{|x|^2}\right]\, dx}{\displaystyle\int_\Omega \phi_N(|x|)^2\, dx}.
\end{equation}
We set
\[
f(t)=\phi'_N(t)^2+(N-1)\,\frac{\phi_N(t)^2}{t^2}\qquad \mbox{ and }\qquad g(t)=\phi_N(t)^2,\qquad t\ge 0,
\]
and recall that $f$ is non-increasing, while $g$ is non-decreasing.
Then from \eqref{palla} and \eqref{nonpalla} we get
\begin{equation}
\label{P}
\begin{split}
\mu_2(\Omega^*) &\int_{\Omega^*} g(|x|)\, dx - \mu_2(\Omega) \int_\Omega g(|x|)\, dx \ge \int_{\Omega^*} f(|x|)\, dx-\int_\Omega f(|x|)\, dx.
\end{split}
\end{equation}
By using the weak Hardy-Littlewood inequality (see Lemma \ref{lm:HL}) and the monotonicity of $g$, we have
\[
\int_\Omega g(|x|)\, dx\ge \int_{\Omega^*} g(|x|)\, dx=\int_{\{|y|\le 1\}} |y|^{2-N}\, J_{N/2}(\beta_{N/2,1}\,|y|)^2\, dy=\omega_N^{2/N}\, \eta_N,
\]
where we used the definition of $\phi_N$ and that of $\varphi_N$, see \eqref{Q}.
The dimensional constant $\eta_N$ is defined by
\[
\eta_N:=N\,\omega_N^\frac{N-2}{N}\, \int_0^1 J_{N/2}(\beta_{N/2,1}\,\varrho)^2\,\varrho\, d\varrho>0.
\] 
Thus, by recalling that $|\Omega|=|\Omega^*|=\omega_N$, inequality \eqref{P} yields
\begin{equation}
\label{P2}
\begin{split}
|\Omega^*|^{2/N}\,\mu_2(\Omega^*) &- |\Omega|^{2/N}\,\mu_2(\Omega)\ge \frac{1}{\eta_N}\,\left[\int_{\Omega^*} f(|x|)\, dx-\int_\Omega f(|x|)\, dx\right].
\end{split}
\end{equation}
The proof by Weinberger now uses Lemma \ref{lm:HL} again to ensure that the right-hand side of \eqref{P2} is positive, which leads to \eqref{segowine}. 
\par
If on the contrary we replace Lemma \ref{lm:HL} by its improved version Lemma \ref{lm:HLquanto}, we can get a quantitative lower bound. 
Since $f$ is non-increasing, by using \eqref{HLquanto} in \eqref{P2} we get
\begin{equation}
\label{P''}
\begin{split}
|\Omega^*|^{2/N}\, \mu_2(\Omega^*) - |\Omega|^{2/N}\,\mu_2(\Omega)& \ge \frac{N\,\omega_N}{\eta_N}\, \int_{R_1}^{R_2} |f(\varrho)-f(1)|\,\varrho^{N-1}\, dx\\
&\ge \frac{N\,\omega_N}{\eta_N}\,\int_{1}^{R_2} [f(1)-f(\varrho)]\,d\varrho. 
\end{split}
\end{equation}
The radii $R_1<1<R_2$ are such that 
\[
|\Omega^*|-|B_{R_1}|=|\Omega^*\setminus \Omega|\qquad \mbox{ and }\qquad |B_{R_2}|-|\Omega^*|=|\Omega\setminus \Omega^*|.
\]
By recalling that $|\Omega|=\omega_N$, they are defined by 
\[
R_1=\left(\frac{|\Omega\cap\Omega^*|}{\omega_N}\right)^\frac{1}{N}\qquad \mbox{ and }\qquad R_2=\left(\frac{|\Omega\setminus\Omega^*|}{\omega_N}+1\right)^\frac{1}{N}
\]
In order to conclude it is now sufficient to observe that
\[
f(1)-f(\varrho)\ge (N-1)\,\phi_N(1)^2\, \left[\frac{\varrho^2-1}{\varrho^2}\right]\ge\frac{N-1}{2^{1/N}}\,\phi_N(1)^2\, (\varrho-1),\qquad \mbox{ for }R_2\ge\varrho\ge 1,
\]
where we also used that $\varrho\le R_2\le 2^{1/N}$.
Thus from \eqref{P''} we get
\[
\begin{split}
|\Omega^*|^{2/N}\, \mu_2(\Omega^*) - |\Omega|^{2/N}\,\mu_2(\Omega)
&\ge\frac{N\,(N-1)\,\omega_N}{2^{(N+1)/N}\,\eta_N}\,\varphi_N(1)^2\,\left(R_2-1\right)^2. 
\end{split}
\]
By using the definition of $R_2$ we have
\begin{equation}
\label{anellosuper}
\left(R_2-1\right)^2=\left(\left(1+\frac{|\Omega\setminus\Omega^*|}{\omega_N}\right)^\frac{1}{N}-1\right)^2\ge (2^{1/N}-1)^2\,\left(\, \frac{|\Omega\setminus\Omega^*|}{\omega_N}\right)^2,
\end{equation}
thanks to the elementary inequality
\[
(1+t)^{1/N}\ge 1+(2^{1/N}-1)\,t,\qquad \mbox{ for every } t\in[0,1],
\]
which follows from concavity. By observing that $|\Omega\Delta\Omega^*|=2\,|\Omega\setminus \Omega^*|$ and recalling the definition of Fraenkel asymmetry, we get the conclusion.
\end{proof}
\begin{remark}
\label{oss:costanteSW}
An inspection of the proof reveals that a feasible choice for the constant $\rho_N$ appearing in \eqref{quantSW} is
\[
\rho_N=(N-1)\,\frac{(2^{1/N}-1)^2}{8\cdot\,2^{1/N}}\,\, \frac{\omega_N^{2/N}\, J_{N/2}(\beta_{N/2,1})^2}{\displaystyle\int_0^1 J_{N/2}(\beta_{N/2,1}\,\varrho)^2\, \varrho\,d\varrho}.
\]
By observing that $\varrho\mapsto J_{N/2}(\beta_{N/2,1}\,\varrho)^2\, \varrho$ is increasing on $(0,1)$, we can estimate this constant from below by
\[
\rho_N\ge (N-1)\,\frac{(2^{1/N}-1)^2}{8\cdot\,2^{1/N}}\,\omega_N^{2/N}.
\]
\end{remark}

\subsection{Checking the sharpness}

As one may see, the proof of the sharp quantitative Szeg\H{o}-Weinberger inequality is considerably simpler than that for the Faber-Krahn inequality. But there is a subtlety here:
indeed, checking sharpness of Theorem \ref{th:quanto_sw} is now much more complicate. The argument used for $\lambda_1$ can not be applied here: indeed, the shape functional
\[
\Omega\mapsto \mu_2(\Omega),
\]
{\it is not differentiable} at the ``maximum point'', i.e. for a ball $B$. This is due to the fact that $\mu_2(B)$ is a multiple eigenvalue (see \cite[Chapter 2]{MR2251558}). Thus what now can happen is that $\mu_2(B)-\mu_2(\Omega_n)$ behaves {\it linearly} along some family $\{\Omega_\varepsilon\}_{\varepsilon>0}$ converging to $B$, i.e.  
\[
\mu_2(B)-\mu_2(\Omega_\varepsilon)\simeq \mathcal{A}(\Omega_\varepsilon),\qquad |\Omega_\varepsilon|=|B|.
\]
Quite surprisingly, the familiy of ellipsoids $\{E_\varepsilon\}_{\varepsilon>0}$ from the previous section exactly exhibits this behaviour. Indeed, by using the same notation as in Section \ref{sec:sciarpaFK}, we have
\[
\mu_2(E_\varepsilon)=\min_{u\in W^{1,2}(B_1)\setminus\{0\}} \left\{\frac{\displaystyle \int_{B_1} \langle \widetilde M_\varepsilon \nabla u,\nabla u\rangle \,dx}{\displaystyle\int_{B_1} |u|^2\,dx}\, :\, \int_{B_1} u\,dx=0\right\}.
\]
By recalling that
\[
\int_{B_1}\langle \widetilde M_\varepsilon \nabla u,\nabla u\rangle\,dx=\int_{B_1}\left(\frac{|u_{x_1}|^2}{(1+\varepsilon)^2}+(1+\varepsilon)^2\,|u_{x_2}|^2+\sum_{i=3}^N |u_{x_i}|^2\right)\,dx,
\]
and 
\[
\frac{1}{(1+\varepsilon)^2}\le 1-2\,\varepsilon+3\,\varepsilon^2,
\]
if we use a $L^2-$normalized eigenfunction of the ball $\xi$ relative to $\mu_2(B)$, we obtain
\[
\begin{split}
\mu_2(E_\varepsilon)\le\mu_2(B_1)&+2\,\varepsilon\,\int_{B_1} \left(|\xi_{x_2}|^2-|\xi_{x_1}|^2\right)\,dx+\varepsilon^2\,\int_{B_1} \left(3\,|\xi_{x_1}|^2+|\xi_{x_2}|^2\right)\,dx.
\end{split}
\]
An important difference with respect to the Dirichlet case now arises. Indeed, $\xi$ is {\it not} radial and with a suitable choice of $\xi$ we can obtain
\[
\int_{B_1} \left(|\xi_{x_2}|^2-|\xi_{x_1}|^2\right)\,dx<0.
\] 
Thus we finally get for $0<\varepsilon\ll 1$
\[
|B_1|^{2/N}\,\mu_2(B_1)-|E_\varepsilon|^{2/N}\,\mu_2(E_\varepsilon)\ge \frac{1}{C}\,\varepsilon\simeq \frac{1}{C}\, \mathcal{A}(E_\varepsilon).
\]
This shows that the family of ellipsoids $\{E_\varepsilon\}_{\varepsilon>0}$ has (at most) a linear decay rate and thus it can not be used to show optimality of the estimate \eqref{quantSW}.
\vskip.2cm
The difficult point is to detect families of deformations of a ball such that $\mu_2(B)-\mu_2(\Omega_\varepsilon)$ behaves quadratically. In other words, we need to identify directions along which $\Omega\mapsto\mu_2(\Omega)$ is smooth around the maximum point. The next result presents a general way to construct such families. This statement generalizes the one in \cite[Section 6]{MR2899684} and comes from the analogous discussion for the Steklov case, treated in \cite[Section 6]{MR2913683}.
\begin{theorem}[Sharpness of the quantitative Szeg\H{o}-Weinberger inequality]
\label{thm:brapradepruf}
Let the function $\psi\in C^\infty(\partial B_1)$ satisfy the following assumptions: 
\begin{itemize}
\item for every $a\in\mathbb{R}^N$, there holds
\begin{equation}
\label{ka1}
\int_{\partial B_1} \langle a,x\rangle\,\psi\, d\mathcal H^{N-1}=0;
\end{equation}
\item for every $a\in\mathbb{R}^N$, there holds
\begin{equation}
\label{ka2}
\int_{\partial B_1} \langle a,x\rangle^2\, \psi\, d\mathcal H^{N-1}=0.
\end{equation}
\end{itemize}
Then the corresponding family $\{\Omega_\varepsilon\}_{\varepsilon>0}$ of nearly spherical domains
\[
\Omega_\varepsilon=\left\{x\in\mathbb{R}^N\, :\, x=0\ \mbox{ or }\ |x|< 1+\varepsilon\,\psi\left(\frac{x}{|x|}\right)\right\},
\]
is such that
\[
\mathcal{A}(\Omega_\varepsilon) \simeq \frac{\big| \Omega_\varepsilon \Delta B_1\big|}{\big| \Omega_\varepsilon\big|}\simeq \varepsilon\qquad \mbox{ and } \qquad |B_1|^{2/N}\,\mu_2(B_1)-|\Omega_\varepsilon|^{2/N}\,\mu_2(\Omega_\varepsilon) \simeq \varepsilon^2,\qquad \varepsilon\ll 1.
\]
\end{theorem}
\begin{remark}
We may notice that the second condition \eqref{ka2} implies also
\begin{equation}
\label{ka3}
\int_{\partial B_1} \psi\, d\mathcal H^{N-1}=0.
\end{equation}
Indeed, we have
\[
\int_{\partial B_1} \psi\, d\mathcal H^{N-1}=\sum_{i=1}^N \int_{\partial B_1}\langle x,\mathbf{e}_i\rangle^2\,\psi\, d\mathcal H^{N-1}=0.
\]
\begin{figure}
\includegraphics[scale=.3]{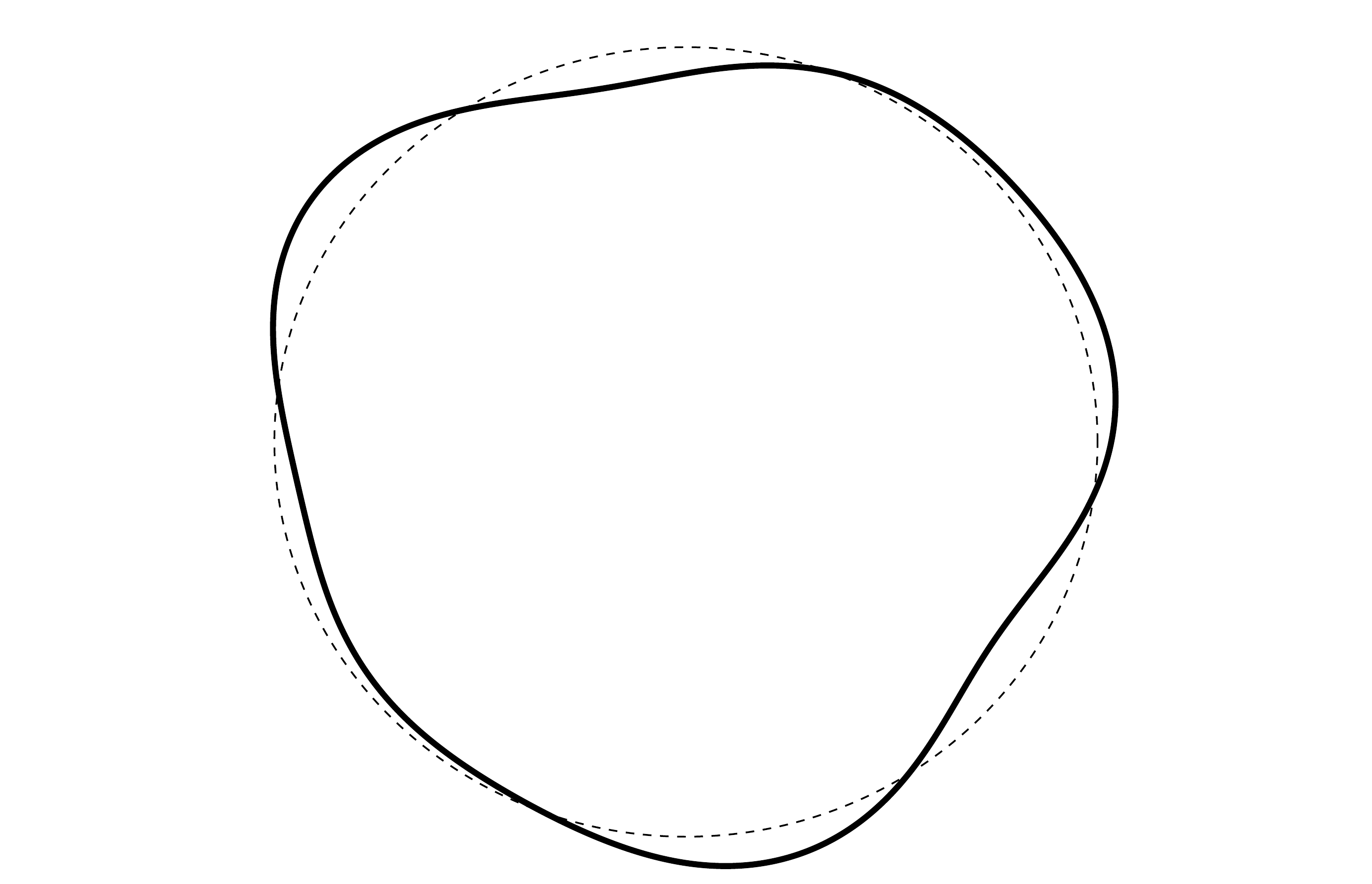}
\caption{The sets $\Omega_\varepsilon$ corresponding to the choice $\psi(\vartheta)=2\,\sin 3\vartheta+\cos5\vartheta$. Such a function satisfies \eqref{ka1} and \eqref{ka2}.}
\end{figure}
\end{remark}
\begin{remark}[Meaning of the assumptions on $\psi$]
Conditions \eqref{ka1}, \eqref{ka2} and \eqref{ka3} are equivalent to say that $\psi$ is orthogonal in the $L^2(\partial B_1)$ sense to the first three eigenspace of the Laplace-Beltrami operator on $\partial B_1$, i.e. to spherical harmonics of order $0, 1$ and $2$ respectively (see \cite{MR1483320} for a comprehensive account on spherical harmonics).
\par
Each of these conditions plays a precise role in the construction: \eqref{ka3} implies that $|\Omega_\varepsilon|-|B_1|\simeq\varepsilon^2$. The first condition \eqref{ka1} implies that $\Omega_\varepsilon$ has the same barycenter as $B_1$, still up to an error of order $\varepsilon^2$. Then this order coincides with the magnitude of $\mathcal{A}(\Omega_\varepsilon)^2$. 
\par
In order to understand the second condition \eqref{ka2}, one should recall that every Neumann eigenfunction $\xi$ relative to $\mu_2(B_1)$ is a linear combination of those defined by \eqref{autopalle}. Thus it has the form
\[
\xi(x)=\varphi_N(|x|)\,\sum_{i=1}^N a_i\,x_i=\varphi_N(|x|)\,\langle a,x\rangle,\qquad x\in B_1,
\]
where $\varphi_N$ is the radial function appearing in \eqref{Q} and $a\in\mathbb{R}^N\setminus\{0\}$. We then obtain for $x\in\partial B_1$
\[
|\xi(x)|^2=\varphi_N(1)^2\,\langle a,x\rangle^2,
\]
and for the tangential gradient $\nabla_\tau$
\[
\begin{split}
|\nabla_\tau \xi|^2=|\nabla \xi-\langle \nabla \xi,x\rangle\,x|^2&=|\nabla \xi|^2-\langle \nabla\xi,x\rangle^2\\
&=-\varphi_N(1)^2\,\langle a,x\rangle^2+\varphi_N(1)^2\,|a|^2.\\
\end{split}
\]
Thus condition \eqref{ka2} implies
\begin{equation}
\label{crucial}
\int_{\partial B_1} \psi\, |\xi|^2\, d\mathcal{H}^{N-1}=0\qquad \mbox{ and } \qquad \int_{\partial B_1} \psi\, |\nabla_\tau \xi|^2\, d\mathcal{H}^{N-1}=0.
\end{equation}
Relations \eqref{crucial} are crucial in order to prove that $\mu_2(B_1)-\mu_2(\Omega_\varepsilon)\simeq\varepsilon^2$.
\end{remark}
We now sketch the proof of Theorem \ref{thm:brapradepruf}. In order to compare $\mu_2(\Omega_\varepsilon)$ with $\mu_2(B_1)$, we define an admissible test function in $B_1$, starting from an eigenfunction $u_\varepsilon$ of $\Omega_\varepsilon$. First of all, we smoothly extend $u_\varepsilon$ outside $\Omega_\varepsilon$, in order to have it defined in a set containing $\Omega_\varepsilon\cup B_1$. Then we take the test function
\[
v_\varepsilon=u_\varepsilonì\cdot 1_{B_1}-\delta_\varepsilon,\qquad \mbox{ where }\delta_\varepsilon =\fint_{B_1} u_\varepsilon\, dx=O(\varepsilon).
\]
By construction, it is not difficult to see that
\begin{equation}
\label{wb}
\left|\int_{B_1} v_\varepsilon^2\,dx-\int_{B_1} u_\varepsilon^2\,dx\right|\le K\,\varepsilon^2,
\end{equation}
By \eqref{wb} and assuming that $u_\varepsilon$ has unit $L^2$ norm on $\Omega_\varepsilon$, we can estimate $\mu_2(B_1)$ 
\begin{equation}
\label{fundamental}
\begin{split}
\mu_2(B_1)\le \frac{\displaystyle\int_{B_1} |\nabla v_\varepsilon|^2\, dx}{\displaystyle\int_{B_1} v_\varepsilon^2\, dx}&\le\frac{\displaystyle\int_{B_1\cap \Omega_\varepsilon} |\nabla u_\varepsilon|^2\, dx+\int_{B_1\setminus \Omega_\varepsilon}|\nabla u_\varepsilon|^2\, dx}{\displaystyle\int_{B_1\cap \Omega_\varepsilon}  u_\varepsilon^2\, dx +\int_{B_1\setminus \Omega_\varepsilon} u_\varepsilon^2\, dx-K\,\varepsilon^2}\\
&\le \frac{\mu_2(\Omega_\varepsilon)+\mathcal{R}_1(\varepsilon)}{1+\mathcal{R}_2(\varepsilon)-K\varepsilon^2}.
\end{split}
\end{equation}
The two {\it error terms} $\mathcal{R}_1(\varepsilon)$ and $\mathcal{R}_2(\varepsilon)$ above are given by
\[
\mathcal{R}_1(\varepsilon)=\int_{B_1\setminus \Omega_\varepsilon}|\nabla u_\varepsilon|^2\,dx-\int_{\Omega_\varepsilon\setminus B_1}|\nabla u_\varepsilon|^2\,dx \qquad\mbox{ and }\qquad \mathcal{R}_2(\varepsilon)=\int_{B_1\setminus \Omega_\varepsilon} u_\varepsilon^2\,dx-\int_{\Omega_\varepsilon\setminus B_1} u_\varepsilon^2\,dx.
\]
It is not difficult to see that the following rough estimate holds
\begin{align}
\label{easyexp1}
\big|\mathcal{R}_1(\varepsilon)\big| \leq K'\,\varepsilon,\qquad \big|\mathcal{R}_2(\varepsilon)\big| \leq K'\,\varepsilon,
\end{align}
for some $K'>0$. Indeed, as $\Omega_\varepsilon$ is a small smooth deformation of $B_1$, then $u_\varepsilon$ satisfies uniform regularity estimates, thus for example $\|\nabla u_\varepsilon\|_{L^\infty}+\|u_\varepsilon\|_{L^\infty}\le C$ and
\[
|\mathcal{R}_1(\varepsilon)|+|\mathcal{R}_2(\varepsilon)|\le 2\,C^2\, |B_1\Delta \Omega_\varepsilon|,
\]
thus giving \eqref{easyexp1}.
By inserting this in \eqref{fundamental}, one would get
\[
\mu_2(B_1)\leq \mu_2(\Omega_\varepsilon)+ K''\,\varepsilon.
\]
This shows that in order to get the correct decay estimate for the deficit, we need to improve \eqref{easyexp1} by replacing $\varepsilon$ with $\varepsilon^2$.
\par
We now explain how the assumptions on the function $\psi$ (i.e. on the boundary of $\partial\Omega_\varepsilon$) imply that the rough estimate \eqref{easyexp1} can be enhanced.
For ease of readability, we present below the heuristic argument, referring the reader to \cite[Section 6]{MR2899684} and \cite[Section 6]{MR2913683} for the rigorous proof.
We focus on the term $\mathcal{R}_1(\varepsilon)$, the ideas for $\mathcal{R}_2(\varepsilon)$ being exactly the same. By using polar coordinates
\[
\int_{\Omega_\varepsilon\setminus B_1}|\nabla u_\varepsilon|^2\, dx = \int_{\{y\in\partial B_1\,:\, \psi(y)\ge0 \}}\int_{1}^{1+\varepsilon\,\psi(y)}\left((\partial_\varrho u_\varepsilon)^2+\frac{N-1}{\varrho^2}|\nabla_\tau u_\varepsilon|^2\right)\,\varrho^{N-1}\,d\varrho\,d\mathcal{H}^{N-1},
\]
and
\[
\int_{B_1\setminus \Omega_\varepsilon}|\nabla u_\varepsilon|^2\, dx = \int_{\{y\in\partial B_1\,:\, \psi(y)\le0 \}}\int_{1+\varepsilon\,\psi(y)}^1\bigg((\partial_\varrho u_\varepsilon)^2+\frac{N-1}{\varrho^2}\,|\nabla_\tau u_\varepsilon|^2\bigg)\,\varrho^{N-1}\,d\varrho\,d\mathcal{H}^{N-1},
\]
where we recall that $\nabla_\tau$ is the tangential gradient and $\partial_\varrho$ is the derivative in the radial direction.
The homogeneous Neumann condition of $u_\varepsilon$ on $\partial \Omega_\varepsilon$ implies that the gradient $\nabla u_\varepsilon$ is ``almost tangential'' in the small sets $B_1\setminus \Omega_\varepsilon$ and $\Omega_\varepsilon\setminus B_1$. In particular
\[
\partial_\varrho u_\varepsilon=O(\varepsilon)\qquad \mbox{ for }\varrho =1+ O(\varepsilon),
\]
and
\[
|\nabla_\tau u_\varepsilon(\varrho,y)|=|\nabla_\tau u_\varepsilon(1,y)|+O(\varepsilon),\qquad \mbox{ for }\varrho = 1 + O(\varepsilon).
\]
By observing that $|B_1\setminus \Omega_\varepsilon| = |\Omega_\varepsilon\setminus B_1| \simeq \varepsilon$, one can compute
\[
\begin{split}
\int_{\Omega_\varepsilon\setminus B_1}|\nabla u_\varepsilon|^2\, dx
&=\int_{\{y\in\partial B_1\,:\, \psi(y)\ge0 \}}\int_1^{1+\varepsilon\,\psi(y)}|\nabla_\tau u_\varepsilon|^2\,\varrho^{N-3}\,d\varrho\,d\mathcal{H}^{N-1}+o(\varepsilon^2)\\
&=\varepsilon\,\int_{\{y\in\partial B_1\,:\, \psi(y)\ge0 \}}\psi\,|\nabla_\tau u_\varepsilon|^2\, d\mathcal{H}^{N-1}+O(\varepsilon^2),
\end{split}
\]
and similarly
\[
\begin{split}
\int_{B_1\setminus \Omega_\varepsilon}|\nabla u_\varepsilon|^2\, dx
&=\int_{\{y\in\partial B_1\,:\, \psi(y)\le0 \}}\int_{1+\varepsilon\psi(y)}^1\left|\nabla_\tau u_\varepsilon\right|^2\,\varrho^{N-3}\,d\varrho\,d\mathcal{H}^{N-1}+o(\varepsilon^2)\\
&=-\varepsilon\int_{\{y\in\partial B_1\,:\, \psi(y)\le0 \}}\psi\,|\nabla_\tau u_\varepsilon|^2\, d\mathcal{H}^{N-1}+O(\varepsilon^2).
\end{split}
\]
Hence, recalling the definition of $\mathcal{R}_1(\varepsilon)$, one gets
\begin{equation}
\label{lastcalculation}
\begin{split}
\mathcal{R}_1(\varepsilon) &= -\varepsilon\int_{\{y\in\partial B_1\,:\, \psi(y)\le0 \}}\psi\,|\nabla_\tau u_\varepsilon|^2\, d\mathcal{H}^{N-1}-\varepsilon\int_{\{y\in\partial B_1\,:\, \psi(y)\ge0 \}}\psi\,|\nabla_\tau u_\varepsilon|^2\, d\mathcal{H}^{N-1}+O(\varepsilon^2)\\
&=-\varepsilon \int_{\partial B_1}\psi\,|\nabla_\tau u_\varepsilon|^2\,d\mathcal{H}^{N-1}+ O(\varepsilon^2).
\end{split}
\end{equation}
It is precisely here that the condition \eqref{ka2} on $\psi$ enters.
Indeed, since $\Omega_\varepsilon$ is smoothly converging to $B_1$, one can guess that $u_\varepsilon$ is sufficiently close to an eigenfunction $\xi$ for $\mu_2(B)$. 
If we assume that we have 
\begin{equation}
\label{icineicine}
u_\varepsilon=\xi+O(\varepsilon),
\end{equation}
then substituting $u_\varepsilon$ with $\xi$ in \eqref{lastcalculation}, one would get
\[
\mathcal{R}_1(\varepsilon)=O(\varepsilon^2).
\]
Indeed, we have seen that \eqref{ka2} implies \eqref{crucial} and thus
\[
\int_{\partial B} \psi\,|\nabla_\tau \xi|^2\,d\mathcal{H}^{N-1}=0.
\]
This would enhance the rate of convergence to $0$ of the term $\mathcal{R}_1(\varepsilon)$ up to an order $\varepsilon^2$. The same arguments can be applied to $\mathcal{R}_2(\varepsilon)$, this time using the first relation in \eqref{crucial}. By inserting these informations in \eqref{fundamental}, one would finally get
\[
\mu_2(B)\le \mu_2(\Omega_\varepsilon)+K\,\varepsilon^2,
\]
as desired. Of course, the most delicate part of the argument is to prove that the guess \eqref{icineicine} is correct in a $C^1$ sense, i.e. that $\|u_\varepsilon-\xi\|_{C^1}=O(\varepsilon)$.

\begin{remark}[Back to the ellipsoids]
Observe that if on the contrary $\psi$ violates condition \eqref{ka2}, we can not assure that all the first-order term in the previous estimates cancel out. For example, for the case of the ellipsoids $E_\varepsilon$ considered above, their boundaries can be described as follows (let us take $N=2$ for simplicity)
\[
\partial E_\varepsilon=\left\{x=\varrho_\varepsilon(y)\,y\, :\, y\in\partial B_1\quad \mbox{ and }\quad \varrho_\varepsilon(y)=\left((1+\varepsilon)^2\,y_2^2+\frac{y_1^2}{(1+\varepsilon)^2}\right)^{-\frac{1}{2}}\right\}.
\]
Observe that 
\[
\varrho_\varepsilon(y)\simeq 1+\varepsilon\,(y_1^2-y_2^2),\qquad y\in\partial B_1,
\]
and the function $\psi(y)=y_1^2-y_2^2$ crucially {\it fails} to satisfy\footnote{This function is indeed a spherical harmonic of order $2$.} \eqref{ka2}. This confirms that
\[
\mu_2(B)-\mu_2(E_\varepsilon)\simeq \varepsilon,
\] 
i.e. ellipsoids do not exhibit the sharp decay rate for the Szeg\H{o}-Weinberger inequality.
\end{remark}

\section{Stability for the Brock-Weinstock inequality}
\label{sec:5}

\subsection{A quick overview of the Steklov spectrum}

Let $\Omega\subset\mathbb{R}^N$ be an open bounded set with Lipschitz boundary. We define its first nontrivial Steklov eigenvalue by
\[
\sigma_2(\Omega):=\inf_{u\in W^{1,2}(\Omega)\setminus\{0\}}\left\{\frac{\displaystyle\int_\Omega |\nabla u|^2\, dx}{\displaystyle\int_{\partial\Omega} |u|^2\, d\mathcal{H}^{N-1}}\, :\, \int_{\partial\Omega} u\, d\mathcal{H}^{N-1}=0\right\},
\] 
where the boundary integral at the denominator has to be intended in the trace sense.
In other words, this is the sharp constant in the Poincar\'e-Wirtinger trace inequality
\[
c\,\int_{\partial\Omega} \left|u-\fint_{\partial\Omega} u\right|^2\,d\mathcal{H}^{N-1}\le \int_\Omega |\nabla u|^2\,dx,\qquad u\in W^{1,2}(\Omega).
\]
Thanks to the assumptions on $\Omega$, the embedding $W^{1,2}(\Omega)\hookrightarrow L^2(\partial\Omega)$ is compact (see \cite[Section 6.10.5]{MR0482102}) and the infimum above is attained. We have again discreteness of the spectrum of the Steklov Laplacian, that we denote by $\{\sigma_1(\Omega),\sigma_2(\Omega),\dots\}$. The first eigenvalue $\sigma_1(\Omega)$ is $0$ and corresponds to constant eigenfunctions. These are the only real numbers $\sigma$ for which the boundary value problem
\[
\left\{\begin{array}{rcll}
-\Delta u&=&0,& \mbox{ in } \Omega,\\
\langle \nabla u,\nu_\Omega\rangle&=&\sigma\, u,& \mbox{ on }\partial\Omega,
\end{array}
\right.
\]
admits nontrivial solutions. Here $\nu_\Omega$ stands for the exterior normal versor. As always, $\sigma_k(\Omega)$ is obtained by minimizing the same Rayleigh quotient, among functions orthogonal (in the $L^2(\partial\Omega)$ sense, this time) to the first $k-1$ eigenfunctions.
The scaling law of Steklov eigenvalues is now
\[
\sigma_k(t\,\Omega)=t^{-1}\,\sigma_k(\Omega)
\]
and we have the sharp inequality due to Brock
\begin{equation}
\label{brocco}
|\Omega|^{1/N}\, \sigma_2(\Omega)\le |B|^{1/N}\,\sigma_2(B),
\end{equation}
with equality if and only if $\Omega$ is a ball. 
\par
As in the case of the Neumann Laplacian, here as well for any ball $B_r$ the first nontrivial eigenvalue has multiplicity $N$. We have $\sigma_2(B_r)=\dots=\sigma_{N+1}(B_r)$ and the corresponding eigenfunctions are just the coordinate functions
\begin{equation}
\label{eigenballstek}
\xi_i(x)=x_i,\qquad i=1,\dots,N.
\end{equation}
Accordingly, we have
\[
\sigma_2(B_r)=\frac{1}{r}.
\]
Actually, in dimension $N=2$ and for simply connected sets, a result stronger than \eqref{brocco} holds. Indeed, if we recall the notation $P(\Omega)$ for the perimeter of a set $\Omega$, for every $\Omega\subset\mathbb{R}^2$ open simply connected bounded set with smooth boundary, we have
\begin{equation}
\label{weinstock}
P(\Omega)\,\sigma_2(\Omega)\le P(B)\,\sigma_2(B),
\end{equation}
where $B$ is any open disc.
This is the {\it Weinstock inequality}, proved in \cite{MR0064989} by means of conformal mappings. Observe that by using the planar isoperimetric inequality
\[
\frac{P(\Omega)}{\sqrt{|\Omega|}}\ge \frac{P(B)}{\sqrt{|B|}},
\]
from \eqref{weinstock} we get
\[
\sqrt{|\Omega|}\,\sigma_2(\Omega)=P(\Omega)\,\sigma_2(\Omega)\,\frac{\sqrt{|\Omega|}}{P(\Omega)}\le P(B)\,\sigma_2(B)\,\frac{\sqrt{|B|}}{P(B)}=\sqrt{|D|}\,\sigma_2(D),
\]
thus for simply connected sets in the plane, inequality \eqref{weinstock} implies \eqref{brocco}. \begin{remark}[The role of topology]
Weinstock inequality is false if we remove the simple connectedness assumption, see \cite[Example 4.2.5]{GP}. On the other hand, the quantity 
\[
P(\Omega)\,\sigma_2(\Omega),
\]
is uniformly bounded from above, but it is still an open problem to compute the sharp bound, see \cite{GP} for more details.
\end{remark}
We also recall that it is possible to provide isoperimetric-like estimates for sums of inverses. For example, for simply connected set in the plane Hersch and Payne in \cite{MR0243210} showed that \eqref{weinstock} can be enforced as follows
\[
\frac{1}{P(\Omega)}\,\left(\frac{1}{\sigma_2(\Omega)}+\frac{1}{\sigma_3(\Omega)}\right)\ge \frac{1}{P(B)}\,\left(\frac{1}{\sigma_2(B)}+\frac{1}{\sigma_3(B)}\right).
\]
In general dimension and without restrictions on the topology of the sets, in \cite{MR1808500} Brock proved that for every $\Omega\subset\mathbb{R}^N$ open bounded set with Lipschitz boundary, we have the stronger inequality
\begin{equation}
\label{inv}
\frac{1}{|\Omega|^{1/N}}\sum\limits_{k=2}^{N+1} \frac{1}{\sigma_k(\Omega)}\ge \frac{1}{|B|^{1/N}}\sum\limits_{k=2}^{N+1} \frac{1}{\sigma_k(B)}.
\end{equation}
Equality in \eqref{inv} holds if and only if $\Omega$ is a ball. By recalling that for a ball we have $\sigma_2(B)=\dots=\sigma_{N+1}(B)$, we see that \eqref{inv} implies \eqref{brocco}.

\subsection{Weighted perimeters}

The proof of \eqref{brocco} is similar to Weinberger's proof of \eqref{segowine}. Namely, one obtains an upper bound on $\sigma_2(\Omega)$ by inserting in the relevant Rayleigh quotient the Steklov eigenfunctions \eqref{eigenballstek} of the ball. This would give
\[
\sigma_2(\Omega)\le \frac{\displaystyle\int_\Omega |\nabla x_i|^2\, dx}{\displaystyle\int_{\partial\Omega} |x_i|^2\, d\mathcal{H}^{N-1}}=\frac{\displaystyle|\Omega|}{\displaystyle\int_{\partial\Omega} |x_i|^2\, d\mathcal{H}^{N-1}},\qquad i=1,\dots,N.
\]
Observe that the chosen test functions are admissible, up to translate $\Omega$ so that its boundary has the barycenter at the origin.
By summing up all the inequalities above, one gets 
\[
\sigma_2(\Omega)\le \frac{\displaystyle N\,|\Omega|}{\displaystyle\int_{\partial\Omega} |x|^2\, d\mathcal{H}^{N-1}}.
\]
Since for a ball we have equality in the previous estimate, in order to conclude the key ingredient is the following weighted isoperimetric inequality
\begin{equation}
\label{BBMP}
\int_{\partial\Omega} |x|^2\,d\mathcal{H}^{N-1}\ge \int_{\partial \Omega^*} |x|^2\,d\mathcal{H}^{N-1},
\end{equation}
where $\Omega^*$ is the ball centered at the origin such that $|\Omega^*|=|\Omega|$. Inequality \eqref{BBMP} has been proved by Betta, Brock, Mercaldo and Posteraro in \cite{MR1734159}. If we use the notation
\[
P_2(\Omega)=\int_{\partial \Omega} |x|^2\, \mathcal H^{N-1},
\] 
and observe that $P_2$ scales like a length to the power $N+1$, \eqref{BBMP} can be rephrased in scaling invariant form as
\begin{equation}
\label{BBMPscale}
|\Omega|^{-\frac{N+1}{N}}\,P_2(\Omega)\ge |B|^{-\frac{N+1}{N}}\,P_2(B),
\end{equation}
where $B$ is any ball {\it centered at the origin}. Equality in \eqref{BBMPscale} holds if and only if $\Omega$ is a ball centered at the origin.
\par
In order to get a quantitative improvement of the Brock-Weinstock inequality, it is sufficient to prove stability of \eqref{BBMPscale}. This has been done in \cite{MR2913683}, by means of an alternative proof of \eqref{BBMPscale} based on a sort of {\it calibration technique} (related ideas can be found in the paper \cite{MR2858470}). 
\begin{theorem}
\label{teo:bradepruf}
For every  $\Omega\subset\mathbb{R}^N$ open bounded set with Lipschitz boundary, we have 
\begin{equation}
\label{stability}
|\Omega|^{-\frac{N+1}{N}}\,P_2(\Omega)-|B|^{-\frac{N+1}{N}}\,P_2(B)\ge \mathfrak{c}_{N}\,\left(\frac{|\Omega\Delta \Omega^*|}{|\Omega|}\right)^2,
\end{equation}
where $\mathfrak{c}_{N}>0$ is an explicit dimensional constant (see Remark \ref{oss:costanteberry} below).
\end{theorem}
\begin{proof}
As always, by scale invariance we can suppose that $|\Omega|=\omega_N$, so that the radius of the ball $\Omega^*$ is $1$.
We start observing that the vector field $x\mapsto |x|\,x$ is such that
\[
\mathrm{div}\left( |x|\,x\right)=(N+1)\,|x|,\quad x\in\mathbb{R}^N.
\]
By integrating the previous quantity on $\Omega$, applying the Divergence Theorem and using Cauchy-Schwarz inequality, we then obtain
\[
\begin{split}
(N+1)\,\int_{\Omega} |x|\, dx&=\int_{\partial\Omega} |x|^2\,\left\langle \frac{x}{|x|},\nu_\Omega(x)\right\rangle\, d\mathcal H^{N-1}\le P_2(\Omega).
\end{split}
\]
On the other hand, by integrating on the ball $\Omega^*$ we get
\[
\begin{split}
(N+1)\,\int_{\Omega^*} |x|\, dx&=\int_{\partial \Omega^*} |x|^2\, d\mathcal H^{N-1}=P_2(\Omega^*),
\end{split}
\]
since $\nu_{\Omega^*}(x)=x/|x|$. We thus obtain the following lower bound for the isoperimetric deficit
\[
\begin{split}
P_2(\Omega)-P_2(\Omega^*)\ge (N+1)\,\left[\int_{\Omega } |x|\, dx-\int_{\Omega^*} |x|\, dx\right].
\end{split}
\]
The proof is now similar to that of the quantitative Szeg\H{o}-Weinberger inequality. By applying again the quantitative Hardy-Littlewood inequality of Lemma \ref{lm:HLquanto}, we get
\[
P_2(\Omega)-P_2(\Omega^*)\ge (N+1)\,N\,\omega_N\,\int_{R_1}^{R_2} |\varrho-1|\, \varrho^{N-1}\, d\varrho.
\]
The radii $R_1<1<R_2$ are still defined by
\[
R_1=\left(\frac{|\Omega\cap\Omega^*|}{\omega_N}\right)^\frac{1}{N}\qquad \mbox{ and }\qquad R_2=\left(\frac{|\Omega\setminus\Omega^*|}{\omega_N}+1\right)^\frac{1}{N}.
\]
With simple manipulations we arrive at
\begin{equation}
\label{deficitPv}
P_2(\Omega)-P_2(\Omega^*)\ge (N+1)\,N\,\omega_N\,\int_{1}^{R_2} (\varrho-1)\, d\varrho.
\end{equation}
As in the proof of the quantitative Szeg\H{o}-Weinberger inequality, we have
\[
\begin{split}
\int_{1}^{R_2} (\varrho-1)\,  d\varrho=\frac{(R_2-1)^2}{2}\ge \frac{(2^{1/N}-1)^2}{2}\,\left(\, \frac{|\Omega\setminus\Omega^*|}{|\Omega|}\right)^2,
\end{split}
\]
where we used again \eqref{anellosuper}. By using this in \eqref{deficitPv} and recalling that $|\Omega\setminus \Omega^*|=|\Omega^*\setminus \Omega|$, we get the desired conclusion.
\end{proof}
\begin{remark}
\label{oss:costanteberry}
The previous proof produces the following constant
\begin{equation}
\label{costanteberry}
\mathfrak{c}_{N}=\frac{(N+1)\,N}{\omega_N^{1/N}}\,\frac{(2^{1/N}-1)^2}{8},
\end{equation}
in inequality \eqref{stability}.
\end{remark}
\begin{remark}
The results of \cite{MR1734159} and \cite{MR2913683} hold for more general weighted perimeters of the form
\[
P_V(\Omega)=\int_{\partial \Omega} V(|x|)\,d\mathcal{H}^{N-1},
\]
under suitable assumptions on the weight $V$. One may also consider anisotropic variants where the Euclidean norm is replaced by a general norm, see \cite[Appendix A]{MR3128695}.
\end{remark}

\subsection{The Brock-Weinstock inequality in sharp quantitative form}

By using Theorem \ref{teo:bradepruf}, one can obtain a quantitative improvement of the stronger inequality \eqref{inv} for the sum of inverses. This has been proved by the Brasco, De Philippis and Ruffini in \cite[Theorem 5.1]{MR2913683}.
\begin{theorem}
For every $\Omega\subset\mathbb{R}^N$ open bounded set with Lipschitz boundary, we have
\begin{equation}
\label{inverso}
\frac{1}{|\Omega|^{1/N}}\,\sum\limits_{k=2}^{N+1} \frac{1}{\sigma_k(\Omega)}-\frac{1}{|B|^{1/N}}\sum\limits_{k=2}^{N+1} \frac{1}{\sigma_k(B)}
\ge \mathfrak{c}_{N}\, \left(\frac{|\Omega\Delta (\Omega^*+x_{\partial\Omega})|}{|\Omega|}\right)^2,
\end{equation}
where the dimensional constant $\mathfrak{c}_{N}>0$ is given by \eqref{costanteberry} and $x_{\partial\Omega}$ is the barycenter of the boundary $\partial\Omega$, i.e.
\[
x_{\partial\Omega}=\fint_{\partial \Omega} x\, d\mathcal H^{N-1}.
\]
\end{theorem}
\begin{proof}
We start by reviewing the proof of Brock. The first ingredient is a variational characterization for the sum of inverses of eigenvalues. In the case of Steklov eigenvalues, the following formula holds (see \cite[Theorem 1]{MR1251868}, for example):
\[
\sum\limits_{k=2}^{N+1} \frac{1}{\sigma_k(\Omega)}=\max\limits_{(v_2,\dots,v_{N+1})\in\mathcal{E}}\sum\limits_{k=2}^{N+1} \int_{\partial\Omega} v_k^2\, d\mathcal H^{N-1},
\]
where the set of admissible functions is given by 
\[
\mathcal{E}=\left\{(v_2,\dots,v_{N+1})\in (W^{1,2}(\Omega))^N\, :\, \int_{\partial\Omega} v_i(x)\, d\mathcal H^{N-1}=0,\, \int_{\Omega} \langle \nabla v_i(x),\nabla v_j(x)\rangle\, dx=\delta_{ij}\right\}.
\]
The quantities $\sigma_i(\Omega)$ are translation invariant, so without loss of generality we can suppose that the barycenter of $\partial \Omega$ is at the origin, i.e. $x_{\partial\Omega}=0$.
This implies that the eigenfunctions relative to $\sigma_2(\Omega^*)=\dots=\sigma_{N+1}(\Omega^*)$ are admissible in the maximization problem above. More precisely, as admissible trial functions we take 
\[
v_i(x)=\frac{x_{i-1}}{\sqrt{|\Omega|}}, \qquad i=2,\dots,N+1.
\] 
In this way, we obtain
\[
\frac{1}{|\Omega|^{1/N}}\,\sum\limits_{k=2}^{N+1} \frac{1}{\sigma_i(\Omega)}\ge \frac{1}{|\Omega|^{1+1/N}}\int_{\partial\Omega} |x|^2\, d\mathcal H^{N-1}=|\Omega|^{-\frac{N+1}{N}}\,P_2(\Omega),
\]
which implies
\begin{equation}
\label{controllino}
\frac{1}{|\Omega|^{1/N}}\,\sum\limits_{k=2}^{N+1} \frac{1}{\sigma_k(\Omega)}-\frac{1}{|\Omega^*|^{1/N}}\sum\limits_{k=2}^{N+1} \frac{1}{\sigma_k(\Omega^*)}\ge |\Omega|^{-\frac{N+1}{N}}\,P_2(\Omega)- |\Omega^*|^{-\frac{N+1}{N}}\,P_2(\Omega^*).
\end{equation}
In the inequality above we used that (recall that $|\Omega^*|^{1/N}\,\sigma_2(\Omega^*)=\omega_N^{1/N}$)
\[
\frac{1}{|\Omega^*|^{1/N}}\sum\limits_{k=2}^{N+1} \frac{1}{\sigma_k(\Omega^*)}=\frac{N}{\omega_N^{1/N}}=|\Omega^*|^{-\frac{N+1}{N}}\,P_2(\Omega^*).
\]
It is then sufficient to use the quantitative estimate \eqref{stability} in \eqref{controllino} in order to conclude.
\end{proof}
As a corollary, we get the following sharp quantitative version of the Brock-Weinstock inequality.
\begin{theorem}
For every $\Omega\subset\mathbb{R}^N$ open bounded set with Lipschitz boundary, we have
\begin{equation}
\label{quanto_bw}
|B|^{1/N}\,\sigma_2(B)-|\Omega|^{1/N}\, \sigma_2(\Omega) \ge \widetilde{\mathfrak{c}}_N\, \left(\frac{|\Omega\Delta(\Omega^*+x_{\partial\Omega})|}{|\Omega|}\right)^2,
\end{equation}
where $\widetilde{\mathfrak{c}}_N>0$ is an explicit constant depending only on $N$ only (see Remark \ref{oss:costantestek} below).
\end{theorem}
\begin{proof}
First of all, we can suppose that
\begin{equation}
\label{unmezzo}
|\Omega|^{1/N}\, \sigma_2(\Omega)\ge \frac{1}{2}\, |B|^{1/N}\,\sigma_2(B),
\end{equation}
otherwise estimate \eqref{quanto_bw} is trivially true with constant $\widetilde\delta_N=1/8\,|B|^{1/N}\,\sigma_2(B)$, just by using the fact that
\[
\frac{|\Omega\Delta (\Omega^*+x_{\partial\Omega})|}{|\Omega|}\le 2.
\]
Let us assume \eqref{unmezzo}. By recalling that $\sigma_2(\Omega)\le\sigma_i(\Omega)$ for every $i\ge 3$, from \eqref{inverso} we can infer
\[
\frac{N}{|\Omega|^{1/N}\, \sigma_2(\Omega)}-\frac{N}{|B|^{1/N}\,\sigma_2(B)}\ge \delta_{N}\, \left(\frac{|\Omega\Delta (\Omega^*+x_{\partial\Omega})|}{|\Omega|}\right)^2.
\]
This can be rewritten as
\[
\frac{|B|^{1/N}\,\sigma_2(B)-|\Omega|^{1/N}\, \sigma_2(\Omega)}{\Big(|\Omega|^{1/N}\, \sigma_2(\Omega)\Big)\,\Big(|B|^{1/N}\,\sigma_2(B)\Big)}\ge \frac{\delta_N}{N}\,\left(\frac{|\Omega\Delta (\Omega^*+x_{\partial\Omega})|}{|\Omega|}\right)^2.
\]
By using \eqref{unmezzo}, the previous inequality easily implies \eqref{quanto_bw}.
\end{proof}
\begin{remark}
\label{oss:costantestek}
By recalling that for every ball $|B|^{1/N}\,\sigma_2(B)=\omega_N^{1/N}$, the constant $\widetilde{\mathfrak{c}}_N$ above is given by
\[
\widetilde{\mathfrak{c}}_N=\frac{\omega_{N}^{1/N}}{2}\,\min\left\{\frac{\delta_N}{N}\,\omega_N^{1/N},\frac{1}{4}\right\},
\]
and $\mathfrak{c}_N$ is the same as in \eqref{costanteberry}.
\end{remark}

\begin{openpb}[Stability of the Weinstock inequality]
Prove that for every $\Omega\subset\mathbb{R}^2$ simply connected open bounded set with smooth boundary, we have
\[
P(B)\,\sigma_2(B)-P(\Omega)\, \sigma_2(\Omega)\ge c_N\,\mathcal{A}(\Omega)^2,
\]
and
\[
\frac{1}{P(\Omega)}\,\left(\frac{1}{\sigma_2(\Omega)}+\frac{1}{\sigma_3(\Omega)}\right)-\frac{1}{P(B)}\,\left(\frac{1}{\sigma_2(B)}+\frac{1}{\sigma_3(B)}\right)\ge c_N\,\mathcal{A}(\Omega)^2.
\]
\end{openpb}

\subsection{Checking the sharpness}
The discussion here is very similar to that of the quantitative Szeg\H{o}-Weinberger inequality. Indeed, the ball is the ``maximum point'' of
\[
\Omega\mapsto |\Omega|^{1/N}\,\sigma_2(\Omega),
\]
and $\sigma_2$ is multiple for a ball, thus again we do not have differentiability. Then verifying that the exponent $2$ on $\mathcal{A}$ is sharp is necessarily involved, exactly like in the Neumann case.
In order to check sharpness of \eqref{th:quanto_sw} we can use exactly the same family of Theorem \ref{thm:brapradepruf}. The heuristic ideas are the same as in the Neumann case, we refer the reader to \cite[Section 6]{MR2913683} for the proof. About the condition \eqref{ka2}, i.e.
\[
\int_{\partial B_1} \langle a,x\rangle^2\,\psi\,d\mathcal{H}^{N-1}=0,\qquad \mbox{ for every }a\in\mathbb{R}^N,
\]
we notice that this is still related to the peculiar form of Steklov eigenfunction of a ball. Indeed, from \eqref{eigenballstek} we know that each eigenfunction $\xi$ corresponding to $\sigma_2(B)$ has the form
\[
\xi=\langle a,x\rangle,\qquad \mbox{ for some } a\in\mathbb{R}^N.
\]
Then we get
\[
|\xi|^2=\langle a,x\rangle^2\qquad \mbox{ and }\qquad |\nabla_\tau\xi|^2=|a|^2-\langle a,x\rangle^2.
\]
Thus condition \eqref{ka2} implies again
\[
\int_{\partial B_1} \psi\,|\xi|^2\,d\mathcal{H}^{N-1}=0\qquad \mbox{ and }\qquad \int_{\partial B_1} \psi\,|\nabla_\tau\xi|^2\,d\mathcal{H}^{N-1}=0,
\]
which are crucial in order to have the sharp decay rate.
\begin{remark}[Sum of inverses]
Observe that 
\[
\begin{split}
\delta_N\,\mathcal{A}(\Omega)^2&\le \frac{1}{|\Omega|^{1/N}}\sum\limits_{k=2}^{N+1} \frac{1}{\sigma_k(\Omega)}-\frac{1}{|B|^{1/N}}\sum\limits_{k=2}^{N+1} \frac{1}{\sigma_k(B)}\\
&\le \frac{N}{|B|^{1/N}\,\sigma_2(B)}\,\left(\frac{|B|^{1/N}\,\sigma_2(B)}{|\Omega|^{1/N}\,\sigma_2(\Omega)}-1\right).
\end{split}
\]
Since the exponent $2$ for $\mathcal{A}(\Omega)$ is sharp in the quantitative Brock-Weinstock inequality, this automatically proves the optimality of inequality \eqref{inverso} as well.
\end{remark}

\section{Some further stability results}
\label{sec:6}

\subsection{The second eigenvalue of the Dirichlet Laplacian}

Up to now we have just considered isoperimetric-like inequalities for ground states energies of the Laplacian, i.e. for first (or first nontrivial) eigenvalues. In each of the cases previously considered, the optimal set was always a ball.
On the contrary, very few facts are known on optimal shapes for successive eigenvalues. In the Dirichlet case, a well-known result states that disjoint pairs of equal balls (uniquely) minimize the second eigenvalue $\lambda_2$, among sets with given volume. This is the so-called {\it Hong-Krahn-Szego inequality}\footnote{This property of balls has been discovered (at least) three times: first by Edgar Krahn (\cite{Kr}) in the '20s. Then the result has been probably neglected, since in 1955 George P\'olya attributes this observation to Peter Szego (see the final remark of \cite{MR0073047}). However, one year before P\'olya's paper, there appeared the paper \cite{MR0070015} by Imsik Hong, giving once again a proof of this result. We owe these informations to the kind courtesy of Mark S. Ashbaugh.}. In scaling invariant form this reads
\begin{equation}
\label{HKS}
|\Omega|^{2/N}\,\lambda_2(\Omega)\ge 2^{2/N}\,|B|^{2/N}\, \lambda_1(B),
\end{equation}
once it is observed that for the disjoint union of two identical balls, the first eigenvalue has multiplicity $2$ and coincides with the first eigenvalue of one of the two balls. Equality in \eqref{HKS} is attained only for disjoint unions of two identical balls, up to sets of zero capacity.
\par
The proof of \eqref{HKS} is quite simple and is based on the following fact.
\begin{lemma}
\label{lm:tecnico}
Let $\Omega\subset\mathbb{R}^N$ be an open set with finite measure. Then there exist two disjoint open sets $\Omega_+,\,\Omega_-\subset \Omega$ such that
\begin{equation}
\label{lemlam2}
\lambda_2(\Omega)=\max\Big\{\lambda_1(\Omega_+),\,\lambda_1(\Omega_-)\Big\}\,. 
\end{equation}
\end{lemma}
For a connected set, the two subsets $\Omega_+$ and $\Omega_-$ above are nothing but the nodal sets of a second eigenfunction. In this case we have
\[
\lambda_1(\Omega_+)=\lambda_2(\Omega)=\lambda_1(\Omega-).
\]
By using information \eqref{lemlam2} and the Faber-Krahn inequality, we get
\begin{equation}
\label{pezzettini}
|\Omega|^{2/N}\,\lambda_2(\Omega)\ge |B|^{2/N}\lambda_1(B)\,\max\left\{\left(\frac{|\Omega|}{|\Omega_+|}\right)^{\frac{2}{N}},\,\left(\frac{|\Omega|}{|\Omega_-|}\right)^{\frac{2}{N}}\right\}.
\end{equation}
By observing that
\begin{equation}
\label{funzioncina}
\max\left\{\left(\frac{|\Omega|}{a}\right)^{\frac{2}{N}},\, \left(\frac{|\Omega|}{b}\right)^{\frac{2}{N}}\right\}\ge 2^{2/N},\qquad \mbox{ for every }a,b> 0,\ a+b\le |\Omega|,
\end{equation}
we obtain inequality \eqref{HKS}. As for equality cases, we observe that if equality holds in \eqref{HKS}, then we must have equality in \eqref{pezzettini} and \eqref{funzioncina}.
The first one implies that $\Omega_+$ and $\Omega_-$ above must be balls (by using equality cases in the Faber-Krahn inequality). But the lower bound in \eqref{funzioncina} is uniquely attained by the pair $a=b=|\Omega|/2$, thus we finally get that $|\Omega_+|=|\Omega_-|=|\Omega|/2$ and $\Omega$ is a disjoint union of two identical balls.
\vskip.2cm
As before, we are interested in improving \eqref{HKS} by means of a quantitative stability estimate. This has been done in \cite[Theorem 3.5]{MR2899684}.
To present this result, we first need to introduce a suitable variant of the Fraenkel asymmetry. This is the {\it Fraenkel $2-$asymmetry}, which measures the $L^1$ distance of a set from the collection of disjoint pairs of equal balls. It is given by
\[
\mathcal{A}_2(\Omega):=\inf\left\{\frac{|\Omega\Delta (B_+\cup B_-)|}{|\Omega|}\, :\, B_+, B_- \mbox{ balls s.\,t. }  \, B_+\cap B_-=\emptyset,\ |B_+|=|B_-|=\frac{|\Omega|}{2}\right\}.
\]
We refer to \cite[Section 2]{1512.00993} for some interesting studies on the functional $\mathcal{A}_2$.
We then have the following quantitative version of the Hong-Krahn-Szego inequality.
We point out that the exponent on the Fraenkel $2-$asymmetry $\mathcal{A}_2$ in \eqref{quanto_ks} is smaller than that in the original statement contained in \cite{MR2899684}, due to the use of the sharp Faber-Krahn inequality of Theorem \ref{thm:fkstability}.
\begin{theorem}
Let $\Omega\subset\mathbb{R}^N$ be an open set with finite measure. Then
\begin{equation}
\label{quanto_ks}
|\Omega|^{2/N}\, \lambda_2(\Omega)-2^{2/N}\, |B|^{2/N}\, \lambda_1(B)\ge \frac{1}{C_N}\,\mathcal{A}_2(\Omega)^{N+1},
\end{equation}
for a constant $C_N>0$ depending on $N$ only.
\end{theorem}
\begin{proof}
Let us set for simplicity
\[
\mathcal{K}(\Omega):=|\Omega|^{2/N}\, \lambda_2(\Omega)-2^{2/N}\, |B|^{2/N}\, \lambda_1(B).
\]
The idea of the proof is to insert quantitative elements in \eqref{pezzettini} and \eqref{funzioncina}, so to obtain an estimate of the type
\begin{equation}
\label{KS}
\mathcal{K}(\Omega)\geq \frac{1}{C_N}\max\left\{\mathcal{A}(\Omega_+)^2+\left|\frac{1}{2}-\frac{|\Omega_+|}{|\Omega|}\right|,\,\mathcal{A}(\Omega_-)^2+\left|\frac{1}{2}-\frac{|\Omega_-|}{|\Omega|}\right|\right\},
\end{equation}
where $\Omega_+$ and $\Omega_-$ are in Lemma \ref{lm:tecnico}. Estimate \eqref{KS} would tell that the deficit on the Hong-Krahn-Szego inequality controls how far $\Omega_+$ and $\Omega_-$ are from being balls having measure $|\Omega|/2$. Once estimate \eqref{KS} is established, the claimed inequality \eqref{quanto_ks} follows from the elementary geometric estimate
\begin{equation}
\label{tassello}
\mathcal{A}_2(\Omega)^\frac{N+1}{2} \leq C_N\, \left(\mathcal{A}(\Omega_+)+\left|\frac{1}{2}-\frac{|\Omega_+|}{|\Omega|}\right|+\mathcal{A}(\Omega_-)+\left|\frac{1}{2}-\frac{|\Omega_-|}{|\Omega|}\right|\right),
\end{equation}
proved in \cite[Lemma 3.3]{MR2899684}.
Observe that since the quantities appearing in the right-hand side of~\eqref{KS} are all bounded by a universal constant, it is not restrictive to prove \eqref{KS} under the further assumption
\begin{equation}
\label{limitazio}
\mathcal{K}(\Omega)\le 2^{2/N}\, |B|^{2/N}\, \lambda_1(B).
\end{equation}
To obtain \eqref{KS}, we need to distinguish two cases.
\vskip.2cm\noindent
{\it Case 1.} Let us suppose that
\[
|\Omega_+|\le \frac{|\Omega|}{2}\qquad \mbox{ and }\qquad |\Omega_-|\le \frac{|\Omega|}{2}.
\]
In this case, let us apply the quantitative Faber-Krahn inequality of Theorem \ref{thm:fkstability} to $\Omega_+$. By recalling \eqref{lemlam2}, we find
\[
\begin{split}
2^{2/N}\,\gamma_{N,2}\, \mathcal{A}(\Omega_+)^2&\leq (2\,|\Omega_+|)^{2/N}\,\lambda_1(\Omega_+)-2^{2/N}\,|B|^{2/N}\,\lambda_1(B)\\
&\leq (2\,|\Omega_+|)^{2/N}\,\lambda_2(\Omega)-2^{2/N}\,|B|^{2/N}\,\lambda_1(B)\\
&=\mathcal{K}(\Omega)+\left((2\,|\Omega_+|)^{2/N}-|\Omega|^{2/N}\right)\,\lambda_2(\Omega)\\
&=\mathcal{K}(\Omega)+|\Omega|^{2/N}\,\lambda_2(\Omega)\,\left[\left(\frac{2\,|\Omega_+|}{|\Omega|}\right)^{2/N}-1\right].\\
\end{split}
\]
By concavity of $\tau\mapsto\tau^{2/N}$, we thus get
\begin{equation}
\label{semprevero}
2^{2/N}\,\gamma_{N,2}\, \mathcal{A}(\Omega_+)^2\le \mathcal{K}(\Omega)+|\Omega|^{2/N}\,\lambda_2(\Omega)\,\frac{4}{N}\,\left(\frac{|\Omega_+|}{|\Omega|}-\frac{1}{2}\right).
\end{equation}
By using the hypothesis on $|\Omega_+|$ and the Hong-Krahn-Szego inequality, we thus obtain
\begin{equation}
\label{avoltevero}
\mathcal{K}(\Omega)\geq c_N\, \mathcal{A}(\Omega_+)^2 +c_N\,\left|\frac{1}{2}-\frac{|\Omega_+|}{|\Omega|}\right|\,.
\end{equation}
Hence, the same computations with $\Omega_-$ in place of $\Omega_+$ yield~(\ref{KS})
\vskip.2cm\noindent
{\it Case 2.} Let us suppose that
\[
|\Omega_+|> \frac{|\Omega|}{2}\qquad \mbox{ and }\qquad |\Omega_-|\le \frac{|\Omega|}{2}.
\]
We still have the estimate \eqref{semprevero} for both $\Omega_+$ and $\Omega_-$. In particular, for the smaller piece $\Omega_-$ we get again \eqref{avoltevero}. On the contrary, this time it is no more true that
\[
\left(\frac{1}{2}-\frac{|\Omega_+|}{|\Omega|}\right) = \left|\frac 12-\frac{|\Omega_+|}{|\Omega|}\right|.
\]
Then for $\Omega_+$ the second term in the right-hand side of \eqref{semprevero} has the wrong sign. The difficulty is that this term could be too big.
However, by recalling that $|\Omega_-|+|\Omega_+|\le |\Omega|$ and using \eqref{avoltevero} for $|\Omega_-|$, we have
\begin{equation}
\label{pezzetto}
\frac{|\Omega_+|}{|\Omega|} -\frac{1}{2}\le \frac{1}{2}-\frac{|\Omega_-|}{|\Omega|}\le \frac{1}{c_N}\, \mathcal{K}(\Omega).
\end{equation}
Therefore, using this information in \eqref{semprevero} and recalling \eqref{limitazio}, we immediately get
\[
\left(\frac{8}{N\,c_N}\,2^{2/N}\,|B|^{2/N}\,\lambda_1(B)+1\right)\mathcal{K}(\Omega)\geq 2^{2/N}\,\gamma_{N,2} \, \mathcal{A}(\Omega_+)^2.
\]
By joining this and \eqref{pezzetto}, we thus obtain estimate \eqref{avoltevero} for $\Omega_+$ as well, possibly with a different constant. Thus we obtain that \eqref{KS} holds in this case as well.
\end{proof}

Concerning the sharpness of estimate \eqref{quanto_ks}, some remarks are in order.
\begin{remark}[Sharpness?]
The proof of \eqref{quanto_ks} consisted of two steps: the first one is the application of the quantitative Faber-Krahn inequality to the two relevant pieces $\Omega_+$ and $\Omega_-$; the second one is the geometric estimate \eqref{tassello}, which enables to switch from the error terms of $\Omega_+$ and $\Omega_-$ to $\mathcal{A}_2(\Omega)$.  Both steps are optimal (for the second one, see \cite[Example 3.4]{MR2899684}), but unfortunately this is of course not a warranty of the sharpness of estimate \eqref{quanto_ks}.\par
Indeed, we are not able to decide whether the exponent for $\mathcal{A}_2$ in~\eqref{quanto_ks} is optimal or not. In any case, we point out that the optimal exponent for the quantitative Hong-Krahn-Szego inequality {\it has to be dimension-dependent}. This follows from the next example.
\end{remark}
\begin{example}
\label{exa:dupalle}
For every $\varepsilon>0$ sufficiently small, we indicate with $B^+_\varepsilon$ and $B^-_\varepsilon$ the open balls of radius $1$, centered at $(1-\varepsilon)\,\mathbf{e}_1$ and $(\varepsilon-1)\,\mathbf{e}_1$ respectively. We also set
\[
\Omega_\varepsilon^+=B_\varepsilon^+\cap \{x_1 \ge 0\} \qquad \mbox{ and }\qquad\Omega_\varepsilon^-=B_\varepsilon^-\cap \{x_1\le 0\},
\]
then we define the set $\Omega_\varepsilon:=
\Omega_\varepsilon^+\cup \Omega_\varepsilon^-\subseteq\mathbb{R}^N$,
for every $\varepsilon>0$ sufficiently small. Observe that we have 
\[
|\Omega^+_\varepsilon|-|B^+_\varepsilon|=O\left(\varepsilon^\frac{N+1}{2}\right)\qquad \mbox{ and }\qquad \lambda_2(\Omega_\varepsilon)-\lambda_1(B^+_\varepsilon)\le O\left(\varepsilon^\frac{N+1}{2}\right),
\]
for the second estimate see for example \cite[Lemma 2.2]{MR3068840}. 
\par
As for asymmetries, it is not difficult to see that 
\[
\mathcal{A}(\Omega_\varepsilon^+)=\mathcal{A}(\Omega^-_\varepsilon)=O\left(\varepsilon^\frac{N+1}{2}\right)\qquad \mbox{ and }\qquad \mathcal{A}_2(\Omega_\varepsilon)=O(\varepsilon).
\]
(see Figure \ref{fig:fendi}). 
\begin{figure}
\includegraphics[scale=.3]{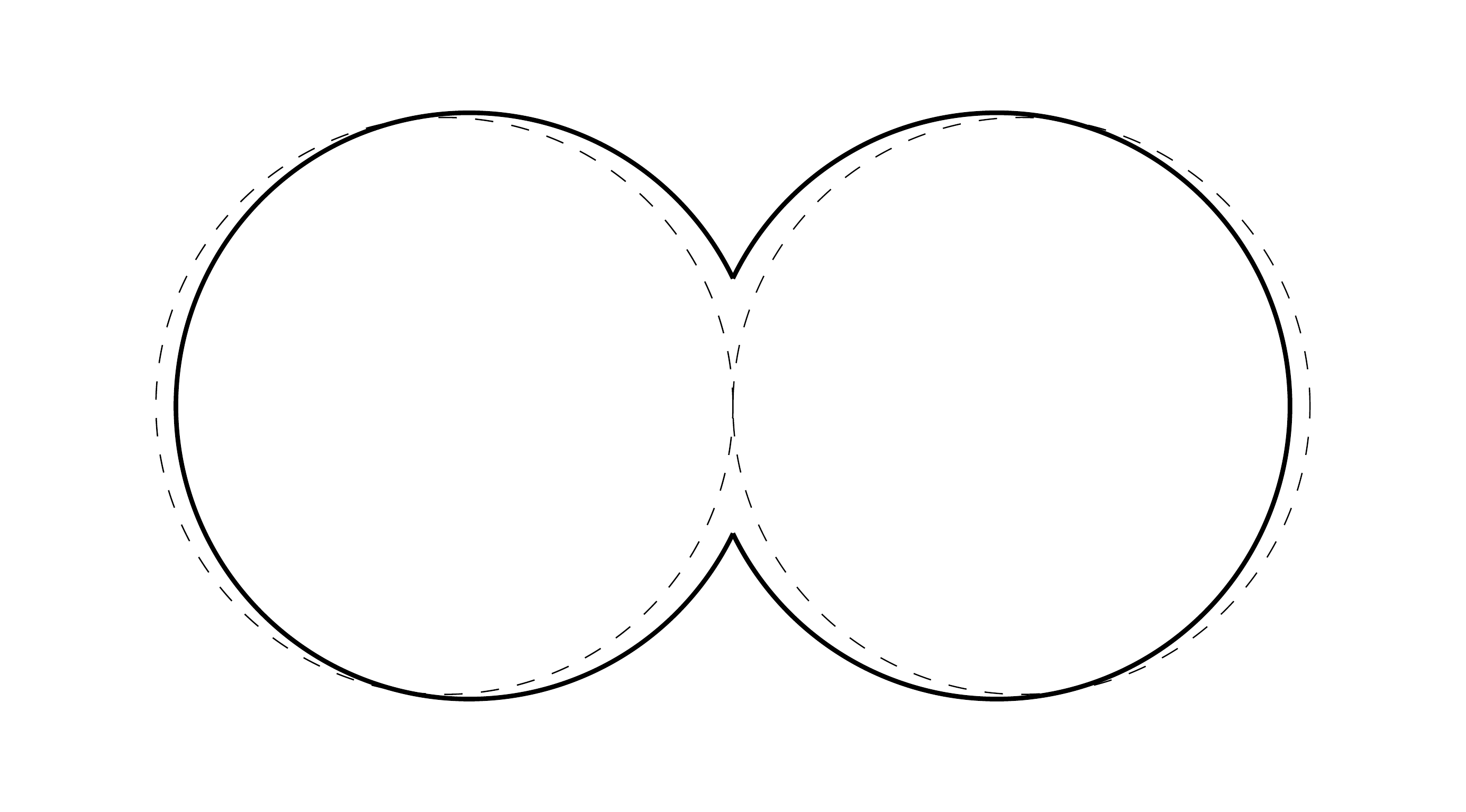}
\caption{The set $\Omega_\varepsilon$ of Example \ref{exa:dupalle} and the pair of balls achieving $\mathcal{A}_2(\Omega_\varepsilon)$.}
\label{fig:fendi}
\end{figure}
Then we get
\[
\begin{split}
|\Omega_\varepsilon|^{2/N}\, \lambda_2(\Omega_\varepsilon)-2^{2/N}\, |B|^{2/N}\, \lambda_1(B)&=2^{2/N}\,\Big(|\Omega_\varepsilon^+|^{2/N}\, \lambda_2(\Omega_\varepsilon)-|B|^{2/N}\, \lambda_1(B)\Big)\\
&=O\left(\varepsilon^\frac{N+1}{2}\right)=O\left(\mathcal{A}_2(\Omega_\varepsilon)^{(N+1)/2}\right).
\end{split}
\] 
This shows that the sharp exponent in \eqref{quanto_ks} has to depend on the dimension and is comprised between $(N+1)/2$ and $N+1$.
\end{example}

\begin{openpb}[Sharp quantitative Hong-Krahn-Szego inequality]
Prove or disprove that the exponent $N+1$ in \eqref{quanto_ks} is sharp. If $N+1$ is not sharp, find the optimal exponent.
\end{openpb}

\subsection{The ratio of the first two Dirichlet eigenvalues}

Another well-known spectral inequality which involves the second Dirichlet eigenvalue $\lambda_2$ is the so-called {\it Ashbaugh-Benguria inequality}. This asserts that the ratio $\lambda_2/\lambda_1$ is maximal on balls and has been proved in \cite{MR1166646, MR1085824}. In other words, for every open set $\Omega\subset\mathbb{R}^N$ with finite measure we have
\begin{equation}
\label{AB}
\frac{\lambda_2(\Omega)}{\lambda_1(\Omega)}\le \frac{\lambda_2(B)}{\lambda_1(B)}.
\end{equation}
\begin{remark}[Equality cases]
Equality cases in \eqref{AB} are a bit subtle: indeed, in general it is not true that equality in \eqref{AB} is attained for balls only. As a counter-example, it is sufficient to consider any disjoint union of the type
\[
\Omega=B\cup \Omega',
\]
with $\Omega'$ open set such that $\lambda_1(\Omega')>\lambda_2(B)$. 
In general equality in \eqref{AB} only implies that the connected component of $\Omega$ supporting $\lambda_1$ and $\lambda_2$ is a ball. 
\end{remark}
\begin{remark}
Inequality \eqref{AB} is an example of {\it universal inequality}. With this name we usually designate spectral inequalities involving eigenvalues only, without any other geometric quantity (see for example \cite{MR1894540}). In particular, inequality \eqref{AB} is valid in the larger class of open sets having discrete spectrum, but not necessarily finite measure.
\end{remark}
The first stability result for \eqref{AB} is due to Melas, see \cite[Theorem 3.1]{MR1168980}. To the best of our knowledge, this is still the best known result on the subject. The original statement was for the asymmetry $d_\mathcal{M}$ defined in \eqref{melas}. Here on the contrary we state the result for the Fraenkel asymmetry.
\begin{theorem}[Melas]
\label{thm:melasrap}
Let $\Omega\subset\mathbb{R}^N$ be an open bounded convex set. Then we have
\begin{equation}
\label{nam}
\frac{\lambda_2(B)}{\lambda_1(B)}-\frac{\lambda_2(\Omega)}{\lambda_1(\Omega)}\ge \frac{1}{C}\, \mathcal{A}(\Omega)^{m},
\end{equation}
for some $C=C(N)>0$ and $m=m(N)>10$ (see Remark \eqref{oss:esponente_melas} below) depending on the dimension $N$ only.
\end{theorem}
We are going to present the core of the proof of Theorem \ref{thm:melasrap} below. At first, one needs a handful of technical results.
\begin{lemma}
\label{lm:porzione}
Let $\Omega\subset\mathbb{R}^N$ be an open set with finite measure. Let $B\subset\mathbb{R}^N$ be a ball such that 
\[
\lambda_1(B)=\lambda_1(\Omega).
\]
There exists a constant $C=C(N)>0$ such that
\begin{equation}
\label{porzione}
\frac{|\Omega|-|B|}{|\Omega|}\ge \frac{1}{C}\,\mathcal{A}(\Omega)^2.
\end{equation}
\end{lemma}
\begin{proof}
Observe that thanks to Theorem \ref{thm:fkstability}, we have
\[
\lambda_1(B)=\lambda_1(\Omega)\ge \frac{|B|^{2/N}}{|\Omega|^{2/N}}\,\lambda_1(B)+\frac{\gamma_{N,2}}{|\Omega|^{2/N}}\,\mathcal{A}(\Omega)^2.
\]
Thus we get
\[
\left(\frac{|\Omega|}{|B|}\right)^\frac{2}{N}-1\ge \frac{\gamma_{N,2}}{|B|^{2/N}\,\lambda_1(B)}\,\mathcal{A}(\Omega)^2.
\]
From the previous inequality, by concavity of the function $\tau\mapsto \tau^{2/N}$ we obtain
\begin{equation}
\label{intermezzo}
c_N'\,\mathcal{A}(\Omega)^2\le \frac{2}{N}\, \frac{|\Omega|-|B|}{|B|}.
\end{equation}
We now distinguish two possibilities: if $|\Omega|\le 2\,|B|$, we have
\[
\frac{|\Omega|-|B|}{|B|}\le 2\, \frac{|\Omega|-|B|}{|\Omega|}.
\]
By inserting this information in the right-hand side of \eqref{intermezzo}, we get \eqref{porzione} as desired.
\par
The case $|\Omega|>2\,|B|$ is even simpler. Indeed, in this case
\[
\frac{|\Omega|-|B|}{|\Omega|}>\frac{1}{2}\ge \frac{1}{8}\,\mathcal{A}(\Omega)^2,
\]
since the asymmetry of a set does not exceed $2$. 
\end{proof}
The key ingredient in the proof by Melas is the following result. It asserts that for non degenerating convex sets with given measure, the values of the first Dirichlet eigenfunction control in a quantitative way the measure of the corresponding sublevel sets. Namely, we have the following.
\begin{lemma}
\label{lm:keyofthekey}
Let $\Lambda>0$, there exist $C=C(N,\Lambda)>0$ and $\beta=\beta(N,\Lambda)>1$ such that for every open convex set $\Omega\subset\mathbb{R}^N$ with
\begin{equation}
\label{nonsqueezing}
|\Omega|=1\qquad \mbox{ and }\qquad \lambda_1(\Omega)\le \Lambda,
\end{equation}
and every $t> 0$ we have
\begin{equation}
\label{controllone}
\frac{1}{C}\,\Big|\{x\in\Omega\, :\, u_1(x)\le t\}\Big|^\beta\le t. 
\end{equation}
Here $u_1$ is the first (positive) Dirichlet eigenfunction of $\Omega$ with unit $L^2$ norm.
\end{lemma}
We omit the proof of the previous result, the interested reader may find it in \cite[Lemma 3.5]{MR1168980}.
We just mention that \eqref{controllone} follows by proving the comparison estimate
\begin{equation}
\label{dalbassotto}
u_1(x)\ge c\,\mathrm{dist}(x,\partial\Omega)^\beta,\qquad x\in\Omega.
\end{equation}
\begin{remark}[The exponent $\beta$]
By recalling that, on a convex set, the first eigenfunction $u_1$ is always globally Lipschitz continuous, we know that the exponent $\beta$ above can not be smaller than $1$. Moreover, it is quite clear that $\beta$ in \eqref{dalbassotto} heavily depends on the regularity of the boundary $\partial\Omega$. To clarify this point, let us stick for simplicity to the case $N=2$.  If $\partial\Omega$ contains a corner at $x_0\in\partial\Omega$ of opening $\alpha<\pi/2$, classical asymptotic estimates based on comparisons with harmonic homogeneous functions imply that
\[
u_1(x)\simeq \mathrm{dist}(x,\partial\Omega)^\frac{\pi}{2\,\alpha},
\]
around the corner $x_0$. This in particular shows that the smaller the angle $\alpha$ is, the larger the exponent $\beta$ in \eqref{dalbassotto} must be. In particular, without taking any further restriction on the convex sets $\Omega$, it would be impossible to get \eqref{dalbassotto}.
\par
The hypothesis \eqref{nonsqueezing} exactly prevents the convex sets considered to become too narrow and permits to have a control like \eqref{dalbassotto}, with a uniform $\beta$.
\end{remark}
Finally, one also needs the following interesting result, whose proof can be found in \cite[Section 6.1]{MR2439395}. This permits to reduce the proof of Theorem \ref{thm:melasrap} to the case of convex sets satisfying the hypothesis \eqref{nonsqueezing} of the previous result.
A similar statement was contained in the original paper by Melas (this is essentially \cite[Proposition 3.1]{MR1168980}), but the proof in \cite{MR2439395} is quicker and simpler. We reproduce it here, with some minor modifications.
\begin{lemma}
\label{lm:ABC}
Let $\{\Omega_n\}_{n\in\mathbb{N}}\subset\mathbb{R}^N$ be a sequence of open convex sets such that
\[
|\Omega_n|=1\qquad \mbox{ and}\qquad \lim_{n\to\infty} \lambda_1(\Omega_n)=+\infty.
\]
Then we have 
\begin{equation}
\label{auno}
\lim_{n\to\infty} \frac{\lambda_2(\Omega_n)}{\lambda_1(\Omega_n)}=1.
\end{equation}
In particular, for every $\delta>0$ there exists $\Lambda=\Lambda(\delta)>0$ such that
\begin{equation}
\label{fermi!}
\sup\left\{\lambda_1(\Omega)\, : \, \Omega\subset\mathbb{R}^N \mbox{ open convex such that } \frac{\lambda_2(\Omega)}{\lambda_1(\Omega)}\ge (1+\delta)\right\}\le \Lambda.
\end{equation}
\end{lemma}
\begin{proof}
We first observe that \eqref{fermi!} easily follows from the first part of the statement. Thus we just need to prove \eqref{auno}. For every $n\in\mathbb{N}$, we take a pair of points $(a_n,b_n)\in\partial\Omega_n$ such that
\[
|a_n-b_n|=\mathrm{diam}(\Omega_n).
\]
Up to rigid motions, we can suppose that 
\[
a_n=(0,\dots,0,\mathrm{diam}(\Omega_n))\qquad \mbox{ and }\qquad b_n=(0,\dots,0).
\]
Observe that the hypotheses on the sequence $\{\Omega_n\}_{n\in\mathbb{N}}$ implies that
\[
\lim_{n\to+\infty}|a_n-b_n|=\lim_{n\to+\infty}\mathrm{diam}(\Omega_n)=+\infty,
\]
see Remark \ref{rem:ABC}.
We now need to prove that for every $n\in\mathbb{N}$ there exists $0<t_n<\mathrm{diam}(\Omega_n)$ such that
\begin{equation}
\label{sezione}
\lambda_1(\Omega_n)\ge \lambda_1(\Omega_n\cap\{x_N=t_n\}).
\end{equation}
Indeed, let us consider the first (positive) eigenfunction $u_{n}\in W^{1,2}_0(\Omega_n)$ with unit $L^2$ norm. By Fubini Theorem we have
\[
\begin{split}
\lambda_1(\Omega_n)=\int_{\Omega_n} |\nabla u_{n}|^2\,dx&\ge\int_0^{\mathrm{diam}(\Omega_n)} \int_{\Omega_n\cap \{x_N=t\}} |\nabla' u_{n}|^2\,dx'\,dt\\
&\ge \int_0^{\mathrm{diam}(\Omega_n)} \left(\lambda_1(\Omega_n\cap\{x_N=t\})\int_{\Omega_n\cap \{x_N=t\}} |u_{n}|^2\,dx'\right)\,dt,
\end{split}
\]
where we used the notation $x'=(x_1,\dots,x_{N-1})$ and $\nabla'=(\partial_{x_1},\dots,\partial_{x_{N-1}})$. Since we assumed
\[
\int_0^{\mathrm{diam}(\Omega_n)} \int_{\Omega_n\cap \{x_N=t\}} |u_{n}|^2\,dx'\,dt=\int_{\Omega_n} |u_{n}|^2\,dx=1,
\] 
from the previous estimate we get \eqref{sezione}. From the fact that $0<t_n<|a_n|=\mathrm{diam}(\Omega_n)$, we have
\begin{itemize}
\item either
\[
|a_n|-t_n=\mathrm{dist}(a_n,\Omega_n\cap\{x_N=t_n\})\ge \frac{\mathrm{diam}(\Omega_n)}{2};
\]
\item or 
\[
t_n=\mathrm{dist}(b_n,\Omega_n\cap\{x_N=t_n\})\ge \frac{\mathrm{diam}(\Omega_n)}{2}.
\]
\end{itemize}
Without loss of generality we can assume that the first condition is verified, then we consider the cone $\mathcal{C}_n$ given by the convex hull of $\{a_n\}\cup(\Omega_n\cap\{x_N=t_n\})$. By convexity, we have $\mathcal{C}_n\subset\Omega_n$ and for every $0<h<|a_n|-t_n$
\[
T_n(h):=\frac{|a_n|-t_n-h}{|a_n|-t_n}\,(\Omega_n\cup\{x_N=t_n\})\times(t_n,t_n+h)\subset\mathcal{C}_n\subset\Omega_n.
\] 
In other words, $\Omega_n$ contains a cylinder having height $h$ and with basis a scaled copy of the $(N-1)-$dimensional section $\Omega_n\cap\{x_N=t_n\}$, see Figure \ref{fig:A}. 
\begin{figure}
\leftline{\includegraphics[scale=.4]{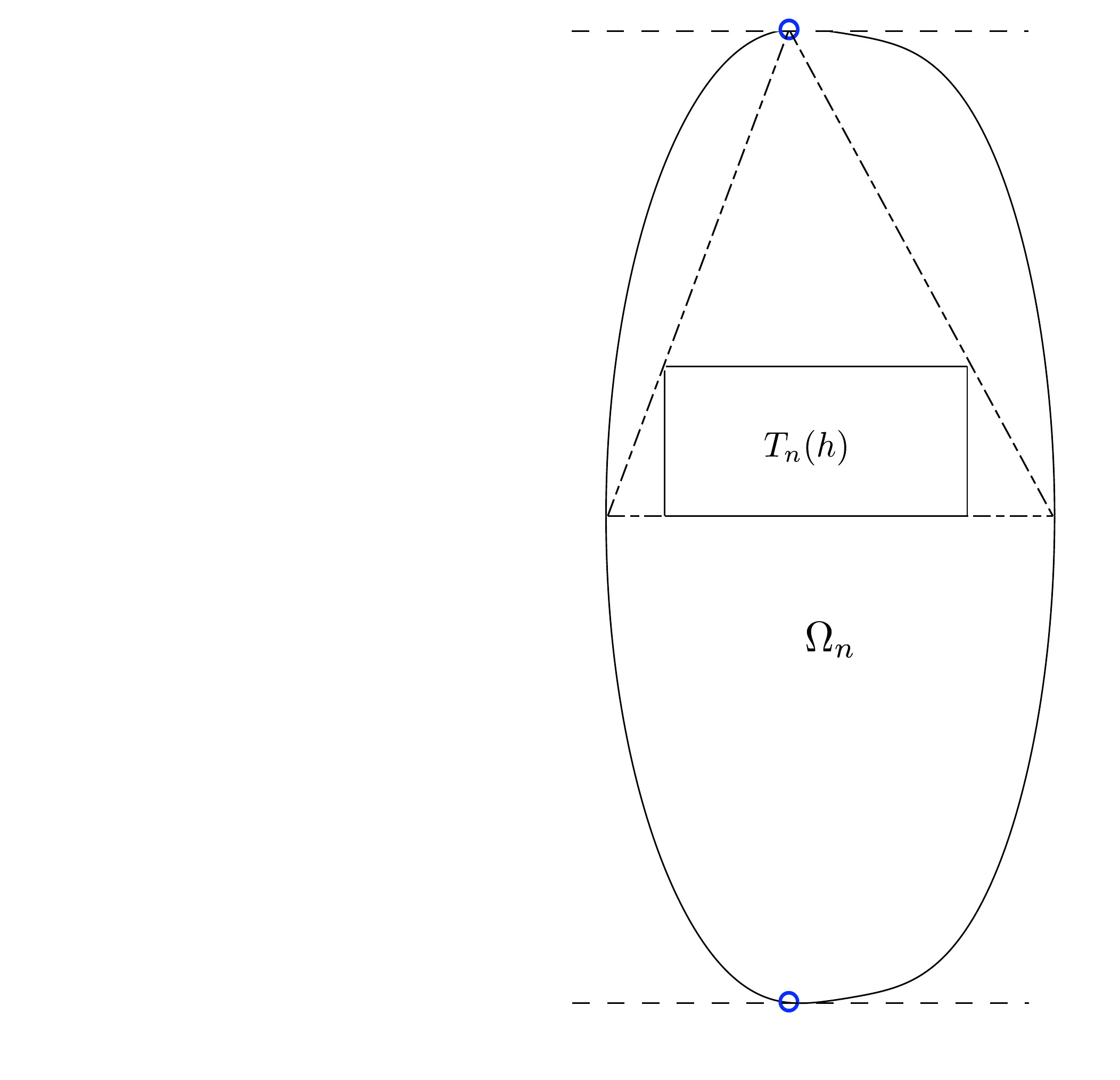}}
\caption{The construction of the cylinder $T_n(h)$ for the proof of Lemma \ref{lm:ABC}.}
\label{fig:A}
\end{figure}
By monotonicity and scaling properties of Dirichlet eigenvalues and \eqref{sezione}, we thus obtain\footnote{We also use that for a cylindric set $\mathcal{O}\times(a,b)$ its Dirichlet eigenfunctions have the form $U(x',x_N)=u(x')\cdot\,v(x_N)$, with $u$ Dirichlet eigenfunction of $\mathcal{O}$ and $v$ Dirichlet eigenfunction of $(a,b)$. The corresponding eigenvalues take the form
\[
\lambda=\lambda'+\frac{n^2\,\pi^2}{(b-a)^2},
\]
where $\lambda'$ is an eigenvalue of $\mathcal{O}$ and $n\in\mathbb{N}$.}
\[
\begin{split}
1\le \frac{\lambda_2(\Omega_n)}{\lambda_1(\Omega_n)}\le \frac{\lambda_2(T_n(h))}{\lambda_1(\Omega_n)}&\le \frac{1}{\lambda_1(\Omega_n)}\,\left[\lambda_1\left(\frac{|a_n|-t_n-h}{|a_n|-t_n}\,(\Omega_n\cup\{x_N=t_n\})\right)+\frac{4\,\pi^2}{h^2}\right]\\
&\le\left[\left(\frac{|a_n|-t_n}{|a_n|-t_n-h}\right)^{2}+\frac{4\,\pi^2}{h^2}\,\frac{1}{\lambda_1(\Omega_n)}\right].
\end{split}
\]
By recalling that $\lambda_1(\Omega_n)$ and $|a_n|-t_n$ are diverging to $+\infty$, we get \eqref{auno} as desired.
\end{proof}
\vskip.2cm\noindent
We now come to the proof of the quantitative Ashbaugh-Benguria inequality.
\begin{proof}[Proof of Theorem \ref{thm:melasrap}]
We first observe that since the functional $\lambda_2/\lambda_1$ is scaling invariant, we can suppose that 
\begin{equation}
\label{1}
|\Omega|=1.
\end{equation}
Moreover, we can always suppose
\begin{equation}
\label{delta}
\lambda_2(\Omega)\ge (1+\delta)\,\lambda_1(\Omega),
\end{equation}
where $\delta$ is the dimensional constant
\[
\delta:=\frac{1}{2}\,\left(\frac{\lambda_2(B)}{\lambda_1(B)}-1\right)>0.
\] 
Indeed, when \eqref{delta} is not verified, then we have
\[
\frac{\lambda_2(B)}{\lambda_1(B)}-\frac{\lambda_2(\Omega)}{\lambda_1(\Omega)}\ge \frac{\lambda_2(B)}{\lambda_1(B)}-(1+\delta)=\frac{1}{2}\,\left(\frac{\lambda_2(B)}{\lambda_1(B)}-1\right)>0,
\]
and the stability estimate is trivially true, with a constant depending on the dimensional constant $\lambda_2(B)/\lambda_1(B)$ only. Finally, thanks to hypothesis \eqref{delta} and Lemma \ref{lm:ABC}, we obtain 
\begin{equation}
\label{nosqueeze}
\lambda_1(\Omega)\le \Lambda,
\end{equation}
with $\Lambda$ depending on $\delta$ only and thus only on the dimension $N$.
\par
We now take the ball $B$ centered at the origin and such that $\lambda_1(\Omega)=\lambda_1(B)$. By \eqref{lambda1palla}, its radius $R$ is given by
\[
R=\frac{j_{N/2-1,1}}{\sqrt{\lambda_1(\Omega)}}.
\]
We set $u_1$ and $z$ for the eigenfunctions corresponding to $\lambda_1(\Omega)$ and $\lambda_1(B)$, normalized by the conditions
\[
\int_\Omega u_1^2\, dx=\int_B z^2\, dx=1.
\] 
We recall that
\[
z(x)=\alpha\,|x|^\frac{2-N}{2}\,J_\frac{N-2}{2}\left(\frac{j_{(N-2)/2,1}}{R}\,|x|\right),
\] 
with the normalization constant $\alpha$ given by
\begin{equation}
\label{costantedim}
\begin{split}
\alpha^2&=\frac{1}{\displaystyle\int_B|x|^{2-N}\,J_{N/2-1}\left(\frac{j_{N/2-1,1}}{R} |x|\right)^2\, dx }\\
&=\frac{1}{R^2\,\displaystyle\int_{\{|y|<1\}} |y|^{2-N}\, J_{N/2-1}\left(j_{N/2-1,1}\, |y|\right)^2\, dy}=:c_N\, R^{-2}. 
\end{split}
\end{equation}
We then compare $\lambda_2(\Omega)$ and $\lambda_2(B)$: since $B$ and $\Omega$ have the same $\lambda_1$, we get
\begin{equation}
\label{lambdi2}
\lambda_2(B)-\lambda_2(\Omega)=\left(\lambda_2(B)-\lambda_1(B)\right)-
\left(\lambda_2(\Omega)-\lambda_1(\Omega)\right).
\end{equation}
We introduce the functions $P_i$ defined as follows 
\[
P_i(x)=g(|x|)\, \frac{x_i}{|x|},\qquad i=1,\dots,N,
\]
with $g$ being the ratio of (the radial parts of) the eigenfunctions corresponding to $\lambda_1(B)$ and $\lambda_2(B)$, that is
\[
g(t)=\frac{J_\frac{N}{2}\left(\frac{j_{N/2,1}}{R}\, t\right)}{J_\frac{N-2}{2}\left(\frac{j_{N/2-1,1}}{R}\, t\right)},
\]
extended as $g(t)\equiv \lim_{s\to R^-} g(s)$ for $t\ge R$. In this way, the functions $P_i$ are defined over $\mathbb{R}^N$. With a suitable choice of the origin of the axes, one can guarantee
(see \cite[Lemma 6.2.2]{MR2251558})
\[
\int_\Omega P_i\, u_1^2\, dx=0,\qquad i=1,\dots,N.
\]
This implies that the functions $P_i\,u_1$ are $L^2-$orthogonal to $u_1$, thus we can use them in the variational problem defining $\lambda_2(\Omega)$. We get
\[
\lambda_2(\Omega)\le \frac{\displaystyle\int_\Omega |\nabla P_i|^2\,u_1^2 dx+\int_\Omega |\nabla u_1|^2\,P_i^2\,dx+2\,\int_\Omega \langle \nabla P_i,\nabla u_1\rangle\,P_i\,u_1\,dx}{\displaystyle\int_\Omega P_i^2\, u_1^2\, dx},
\]
then we observe that by testing the equation $-\Delta u_1=\lambda_1(\Omega)\,u_1$ against $P_i^2\,u_1$, we obtain
\[
\int_\Omega |\nabla u_1|^2\,P_i^2\,dx+2\,\int_\Omega \langle \nabla P_i,\nabla u_1\rangle\,P_i\,u_1\,dx=\lambda_1(\Omega)\,\int_\Omega P_i^2\,u_1^2\,dx.
\] 
This permits to infer
\begin{equation}
\label{peromega}
\lambda_2(\Omega)-\lambda_1(\Omega)\le \frac{\displaystyle\int_\Omega |\nabla P_i|^2\,u_1^2 dx}{\displaystyle\int_\Omega P_i^2\, u_1^2\, dx}.
\end{equation}
The same computations in the ball $B$ give of course equalities everywhere, since in this case $P_i\,u_1$ would coincide with a second Dirichlet eigenfunction of $B$. Thus
\begin{equation}
\label{perpalla}
\frac{\displaystyle\int_B |\nabla P_i|^2\,z^2 dx}{\displaystyle\int_B P_i^2\, z^2\, dx}=\lambda_2(B)-\lambda_1(B).
\end{equation}
We then perform the standard trick of adding these (in)equalities for $i=1,\dots,N$, which let the angular variables disappear, as in the proof of the Szeg\H{o}-Weinberger inequality.
Thus from \eqref{peromega} and \eqref{perpalla}, we obtain
\begin{equation}
\label{lambdi21}
\left(\int_\Omega g^2\, u_1^2\, dx\right)[\lambda_2(\Omega)-\lambda_1(\Omega)]\le \int_\Omega \left[|g'(|x|)|^2+(N-1)\, \left(\frac{g(|x|)}{|x|}\right)^2\right]\, u_1^2\, dx,
\end{equation}
and
\begin{equation}
\label{lambdi22}
\int_B \left[|g'(|x|)|^2+(N-1)\, \left(\frac{g(|x|)}{|x|}\right)^2\right]\, z^2\, dx=
\left(\int_B g^2\, z^2\, dx\right)[\lambda_2(B)-\lambda_1(B)].
\end{equation}
We now use the fact that the function $g$ is monotone non-decreasing on $\mathbb{R}^N$ (see \cite[Theorem 3.3]{MR1166646}), so that by Hardy-Littlewood inequality we obtain\footnote{There is a small subtility here. Indeed, even if $g$ is a radially non-decreasing function, in general we just have 
\[
g_*(r)\ge g(r),\qquad \mbox{ for } 0\le r\le R_\Omega,
\]
and inequality could be strict. This is due to the fact that $g_*$ is constructed by rearranging the super-level sets of $g$ on $\Omega$ and not on the whole $\mathbb{R}^N$ (see \cite[page 606]{MR1166646}).}
\[
\int_\Omega g^2\, u_1^2\, dx\ge \int_{\Omega^*} g_*^2\, (u^*_1)^2\, dx\ge \int_{\Omega^*} g^2\,(u^*_1)^2\,dx,
\]
where $g_*$ denotes the spherically {\it increasing} rearrangement. As always, $\Omega^*$ is the ball centered at the origin such that $|\Omega^*|=|\Omega|=1$. We denote its radius by $R_\Omega$.
\par
Another essential ingredient in the proof by Ashbaugh and Benguria is a {\it comparison result} between $z$ and the spherical rearrangement $u_1^*$ of $u_1$. This is due to Chiti, who proved (see \cite{MR0652376}) that there exists a radius $r_1\in(0,R)$ such that
\begin{equation}
\label{chiti}
\left\{\begin{array}{cc} 
u^*_1(x)\le z(x),& \mbox{ for } 0\le|x|\le r_1,\\ 
&\\
u^*_1(x)\ge z(x),& \mbox{ for } r_1\le|x|\le R.  
\end{array}
\right.
\end{equation}
This in turn implies that
\[
\int_{\Omega^*} g^2\, (u^*_1)^2\, dx\ge \int_B g^2\, z^2\, dx. 
\]
More precisely, by using \eqref{chiti}, the monotonicity of $g$ and polar coordinates we have
\[
\begin{split}
\int_{\Omega^*} g^2\, [(u^*_1)^2-z^2]\, dx&=N\,\omega_N\, \int_0^{r_1} g^2\, [(u^*_1)^2-z^2]\, \varrho^{N-1}\, d\varrho\\
&+N\,\omega_N\, \int_{r_1}^{R} g^2\, [(u^*_1)^2-z^2]\, \varrho^{N-1}\, d\varrho\\
&+N\,\omega_N\,\int_R^{R_\Omega} g^2\, (u^*_1)^2\, \varrho^{N-1}\, d\varrho\\
&\ge g(r_1)^2\, \int_{B} [(u_1^*)^2-z^2]\, dx +g(R)^2\,\int_{\Omega^*\setminus B} (u_1^*)^2\, dx\\
&= \left(g(R)^2-g(r_1)^2\right) \int_{\Omega^*\setminus B} u^*_1(x)^2\, dx.
\end{split}
\]
In the last equality we used that
\[
\int_B z^2\,dx=\int_{\Omega^*} (u^*_1)^2\,dx=1.
\]
We now observe that, using the definition both of $z$ and $g$, we get
\[
\int_B g^2\, z^2\, dx=\alpha^2\,R^{2}\,\int_{\{|y|<1\}} J_\frac{N}{2}(j_{N/2,1}\,|y|)^2\, dy:= C_N,
\]
thanks to \eqref{costantedim}. 
In this way, we have shown
\[
\int_{\Omega^*} g^2\, |u^*_1|^2\, dx\ge C_N +\left(g(R)^2-g(r_1)^2\right) \int_{\Omega^*\setminus B} u^*_1(x)^2\, dx. 
\]
In the end, by using \eqref{lambdi2}, \eqref{lambdi21} and \eqref{lambdi22} we have obtained that
\[
\begin{split}
C_N\,\Big(\lambda_2(B)-\lambda_2(\Omega)\Big)&\ge\left(\int_B g^2\, z^2\, dx\right)\,\Big(\lambda_2(B)-\lambda_1(B)\Big)-\left(\int_\Omega g^2\, |u_1^*|^2\, dx\right)\,\Big(\lambda_2(\Omega)-\lambda_1(\Omega)\Big)\\
&+\left(g(R)^2-g(r_1)^2\right)\,\left(\int_{\Omega^*\setminus B} u^*_1(x)^2\, dx\right)\,\Big(\lambda_2(\Omega)-\lambda_1(\Omega)\Big)\\
&\ge \displaystyle\int_B \left[|g'|^2+(N-1)\, \left(\frac{g}{|x|}\right)^2\right]\,z^2\, dx-\displaystyle\int_\Omega \left[|g'|^2+(N-1)\, \left(\frac{g}{|x|}\right)^2\right]\,u_1^2\, dx\\
&+\delta\,\left(g(R)^2-g(r_1)^2\right)\,\left(\int_{\Omega^*\setminus B} (u^*_1)^2\, dx\right)\,\lambda_1(\Omega),
\end{split}
\]
where we also used hypothesis \eqref{delta} for the last term.
It is now time to use the monotonicity properties of the function
\[
G(t)=g'(t)^2+(N-1)\left(\frac{g(t)}{t}\right)^2,
\]
which is monotone non-increasing (see \cite[Corollary 3.4]{MR1166646}), so that again by Hardy-Littlewood inequality
\[
\int_\Omega G\, u_1^2\, dx\le \int_{\Omega^*} G^*\, (u_1^*)^2\, dx\le \int_{\Omega^*} G\,(u_1^*)^2\,dx.
\]
and we thus have
\[
\begin{split}
C_N\Big(\lambda_2(B)-\lambda_2(\Omega)\Big)&\ge \int_B G\,z^2\, dx-\int_{\Omega^*} G\, (u^*_1)^2\, dx+\delta\,\left(g(R)^2-g(r_1)^2\right)\,\left(\int_{\Omega^*\setminus B} (u^*_1)^2\, dx\right)\lambda_1(\Omega).
\end{split}
\]
Proceeding as before, by using \eqref{chiti}, the monotonicity of $G$ and indicating as always with $R_\Omega$ the radius of $\Omega^*$, we obtain (we omit the details)
\[
\begin{split}
\int_B G\, z^2\, dx-\int_{\Omega^*} G\, (u^*_1)^2\, dx\ge 0. 
\end{split}
\]
What we have obtained so far is the following
\[
\lambda_2(B)-\lambda_2(\Omega)\ge \frac{\delta}{C_N}\, \left(g(R)^2-g(r_1)^2\right)\,\left(\int_{\Omega^*\setminus B} |u^*_1|^2 \, dx\right)\,\lambda_1(\Omega).
\]
Dividing by $\lambda_1(\Omega)=\lambda_1(B)$, the previous can be rewritten as
\begin{equation}
\label{stabile?}
\frac{\lambda_2(B)}{\lambda_1(B)}-\frac{\lambda_2(\Omega)}{\lambda_1(\Omega)}\ge c\,\left(g(R)^2-g(r_1)^2\right)\,\int_{\Omega^*\setminus B} |u^*_1|^2 \, dx,
\end{equation}
where $c>0$ depends on $N$ only.
We now choose $\widetilde R=(R+R_\Omega)/2$, the monotone behaviour of $u^*_1$ permits to infer\footnote{The paper \cite{MR1168980} contains a misprint in this part of the proof. Indeed, it is claimed that (in our notation)
\[
\int_{\Omega^*\setminus B} u^*_1(x)^2 \, dx\ge u^*_1(R)^2\,|\Omega^*\setminus B|,
\]
which is of course not true, since $u_1^*$ is non-increasing.}
\[
\begin{split}
\int_{\Omega^*\setminus B} |u^*_1|^2 \, dx&\ge \int_{B_{\widetilde R}\setminus B} |u^*_1|^2\, dx\ge |B_{\widetilde R}\setminus B|\, u_1^*(\widetilde R)^2\\
&= \frac{\omega_N}{2^N}\,\left[(R_\Omega+R)^N-(2\,R)^N\right] u_1^*(\widetilde R)^2\\
&\ge \frac{\omega_N}{2^N}\,\Big(R_\Omega^N-R^N\Big)\,u_1^*(\widetilde R)^2= \frac{1}{2^N}\, |\Omega^*\setminus B|\,u_1^*(\widetilde R)^2.
\end{split}
\]
In the second inequality above we used the elementary fact
\begin{equation}
\label{increment}
(a+h)^N-(b+h)^N\ge a^N-b^N,\qquad a\ge b\ge 0,\ h\ge 0,
\end{equation}
which follows from convexity.
Thus from \eqref{stabile?} we obtain
\[
\begin{split}
\frac{\lambda_2(B)}{\lambda_1(B)}-\frac{\lambda_2(\Omega)}{\lambda_1(\Omega)}&\ge c\,\left(g(R)^2-g(r_1)^2\right)\,|\Omega^*\setminus B|\, u^*_1(\widetilde R)^2.
\end{split}
\]
We now use further properties of $g$: indeed, in addition to being increasing on $[0,R]$, this is also concave on the same interval, with $g''<0$ on $(0,R]$ (see \cite[proof of Theorem 3.3]{{MR1166646}}). Thus we get  
\[
g(R)^2-g(r_1)^2\ge g(R)\,(g(R)-g(r_1))\ge C\,\int_{r_1}^R (-g''(\tau))\,(\tau-r_1)\,d\tau,
\]
by Taylor formula and the fact that $C=g(R)>0$ is a constant depending on $N$ only. This in particular implies
\[
g(R)^2-g(r_1)^2\ge c\,(R-r_1)^2,
\]
for a possibly different constant $c>0$. We join this with the fact that $z$ is a Lipschitz functions. Thus for some $c'>0$ depending on $N$ only we have 
\[
(R-r_1)^2\ge c'\,\Big(z(r_1)-z(R)\Big)^2=c'\,z(r_1)^2=c'\,u^*_1(r_1)^2\ge c'\,u^*_1(\widetilde R)^2.
\]
By resuming, we finally obtain
\begin{equation}
\label{3.25}
\frac{\lambda_2(B)}{\lambda_1(B)}-\frac{\lambda_2(\Omega)}{\lambda_1(\Omega)}\ge c\,|\Omega^*\setminus B|\, u^*_1(\widetilde R)^4,
\end{equation}
for some $c>0$, still depending on $N$ only. In order to conclude, we now use the key Lemma \ref{lm:keyofthekey}. Indeed, by recalling \eqref{1} and \eqref{nosqueeze}, we can infer from \eqref{controllone}
\[
\frac{1}{C}\,\Big|\{x\in\Omega\, :\, u_1(x)\le u_1^*(\widetilde R)\}\Big|^\beta\le u_1^*(\widetilde R).
\]
Moreover, by definition of $u_1^*$ we have
\[
\begin{split}
\Big|\{x\in\Omega\, :\, u_1(x)\le u_1^*(\widetilde R)\}\Big|=|\Omega^*\setminus B_{\widetilde{R}}|&=\frac{\omega_N}{2^N}\,\left((2\,R_\Omega)^N-(R+R_\Omega)^N\right) \\
&\ge \frac{\omega_N}{2^N}\,\Big(R_\Omega^N-R^N\Big)=\frac{1}{2^N}\, |\Omega^*\setminus B|,
\end{split}
\]
again thanks to \eqref{increment}. By using the last two estimates in \eqref{3.25}, we get
\[
\frac{\lambda_2(B)}{\lambda_1(B)}-\frac{\lambda_2(\Omega)}{\lambda_1(\Omega)}\ge c\,|\Omega^*\setminus B|^{4\,\beta+1},
\]
for some constant $c=c(N)>0$. By recalling that $|\Omega^*\setminus B|=|\Omega|-|B|=1-|B|$ and that $\lambda_1(B)=\lambda_1(\Omega)$, we can use Lemma \ref{lm:porzione} and finally get the conclusion by \eqref{porzione}.
\end{proof}
\begin{remark}
\label{oss:esponente_melas}
An inspection of the proof reveals that the exponent $m$ appearing in \eqref{nam} is given by
\[
m=2\,(4\,\beta+1)>10,
\]
with $\beta$ being the exponent coming from Lemma \ref{lm:keyofthekey}.
\end{remark}

\begin{openpb}[Sharp quantitative Ashbaugh-Benguria inequality]
Prove that there exists a dimensional constant $c_N>0$ such that for every open bounded convex set $\Omega\subset\mathbb{R}^N$ 
\[
\frac{\lambda_2(B)}{\lambda_1(B)}-\frac{\lambda_2(\Omega)}{\lambda_1(\Omega)}\ge c_N\, \mathcal{A}(\Omega)^2.
\]
The same family of sets in Theorem \ref{thm:brapradepruf} should give that the exponent $2$ is the best possible (recall that $\lambda_2$ is multiple for a ball and thus not differentiable).
\end{openpb}

\subsection{Neumann VS. Dirichlet}
It is immediate to see that by joining Faber-Krahn and Szeg\H{o}-Weinberger inequalities, one gets the universal inequality
\begin{equation}
\label{against}
\frac{\mu_2(\Omega)}{\lambda_1(\Omega)}\le \frac{\mu_2(B)}{\lambda_1(B)}.
\end{equation}
Equality in \eqref{against} holds only for balls. By observing that\footnote{This can be obtained by direct computation. However, it is also possible to prove directly that $\theta_N<1$, without computing its explicit value. This is indeed a consequence of the sharp estimate for convex sets
\[
\mu_2(\Omega)<\lambda_1(B)\,\left(\frac{\mathrm{diam\,}(B)}{\mathrm{diam\,}(\Omega)}\right)^2,
\]
which holds with strict inequality sign, see \cite[Theorem 3.1 \& Corollary 3.2]{BraNitTro}.}
\begin{equation}
\label{thetaN}
\theta_N:=\frac{\mu_2(B)}{\lambda_1(B)}=\left(\frac{\beta_{N/2,1}}{j_{N/2-1,1}}\right)^2<1,
\end{equation}
one can obtain
\[
\mu_2(\Omega)<\lambda_1(\Omega).
\]
We have the following quantitative improvement of \eqref{against}. 
\begin{corollary}
For every $\Omega\subset\mathbb{R}^N$ open set with finite measure, we have
\begin{equation}
\label{mu2la1}
\frac{\mu_2(B)}{\lambda_1(B)}-\frac{\mu_2(\Omega)}{\lambda_1(\Omega)} \geq \kappa_N\,\mathcal{A}(\Omega)^2,
\end{equation}
where $\kappa_N>0$ is an explicit constant depending on $N$ only.
\end{corollary}
\begin{proof}
Let $B$ be a ball such that $|B|=|\Omega|$, then we write
\begin{equation}
\label{becille}
\begin{split}
\frac{\mu_2(B)}{\lambda_1(B)}-\frac{\mu_2(\Omega)}{\lambda_1(\Omega)}=\left(\frac{\mu_2(B)}{\lambda_1(B)}-\frac{\mu_2(B)}{\lambda_1(\Omega)}\right)+\left(\frac{\mu_2(B)}{\lambda_1(\Omega)}-\frac{\mu_2(\Omega)}{\lambda_1(\Omega)}\right).
\end{split}
\end{equation}
If we suppose that $\lambda_1(\Omega)>2\,\lambda_1(B)$, by using Szeg\H{o}-Weinberger for the second term in the right-hand side of \eqref{becille} we get
\[
\frac{\mu_2(B)}{\lambda_1(B)}-\frac{\mu_2(\Omega)}{\lambda_1(\Omega)}\ge \mu_2(B)\,\left(\frac{1}{\lambda_1(B)}-\frac{1}{\lambda_1(\Omega)}\right)\ge \frac{\mu_2(B)}{2\,\lambda_1(B)}\ge \frac{\mu_2(B)}{8\,\lambda_1(B)}\,\mathcal{A}(\Omega)^2.
\]
As always, we used that $\mathcal{A}(\Omega)<2$. Thus the conclusion follows in this case. 
\par
If on the contrary $\lambda_1(\Omega)\le 2\,\lambda_1(B)$, then by using the Faber-Krahn inequality for the first term in the right-hand side of \eqref{becille}, we obtain
\[
\frac{\mu_2(B)}{\lambda_1(B)}-\frac{\mu_2(\Omega)}{\lambda_1(\Omega)}\ge \left(\frac{\mu_2(B)}{\lambda_1(\Omega)}-\frac{\mu_2(\Omega)}{\lambda_1(\Omega)}\right)\ge \frac{1}{2\,\lambda_1(B)}\,\Big(\mu_2(B)-\mu_2(\Omega)\Big).
\]
It is now sufficient to use Theorem \ref{th:quanto_sw} to conclude. 
\end{proof}

\begin{remark}
A feasible value for the constant $\kappa_N$ appearing in \eqref{mu2la1} is 
\[
\kappa_N=\frac{1}{2}\,\min\left\{\frac{\theta_N}{4},\, \frac{\rho_N}{|B|^{2/N}\,\lambda_1(B)}\right\}.
\]
Here $\theta_N$ is defined in \eqref{thetaN} and $\rho_N$ is the same constant appearing in \eqref{quantSW}.
\end{remark}

\section{Notes and comments}
\label{sec:7}

\subsection{Other references}
\label{sec:71}
We wish to mention that one of the first paper to introduce quantitative elements in the P\'olya-Szeg\H{o} principle \eqref{PS} was \cite{MR0202053} by Makai. There the scope was to add some remainder term in order to infer uniqueness of balls as extremals for the Saint-Venant inequality \eqref{sv}. More recently, sophisticated quantitative improvements of the P\'olya-Szeg\H{o} principle have been proven in \cite{MR3250365} and \cite{MR2376285}. 
\par
The first papers to prove quantitative improvements of the Faber-Krahn inequality for general open sets with respect to the Fraenkel asymmetry have been \cite{MR1480023} by Sznitman (for $N=2$) and \cite{MR1701519} by Povel (for $N\ge 3$). It is interesting to notice that both papers prove such a kind of estimates for probabilistic purposes. In \cite{MR1480023} these estimates are employed to study the asimptotic behaviour of the first eigenvalue of $-\Delta+V_\omega$ of a square $(0,\ell)\times(0,\ell)$ for large $\ell$. Here $V_\omega$ is a soft repulsive random potential. In \cite{MR1701519} a quantitative Faber-Krahn inequality is used to estimate the extinction time of a Brownian motion in presence of (random) absorbing obstacles.
\par
Among the contributions to the subject, it is mandatory to mention the papers \cite{MR1836803} by Bhattacharya and \cite{MR2512200} by Fusco, Maggi and Pratelli. Both papers consider the more general case of the $p-$Laplacian operator $\Delta_p$, defined by $\Delta_p u=\mathrm{div}(|\nabla u|^{p-2}\,\nabla u)$. More precisely, they consider the quantities
\[
\min_{u\in W^{1,p}_0(\Omega)}\left\{\int_\Omega |\nabla u|^p\,dx\, :\, \|u\|_{L^q(\Omega)}=1\right\},\qquad 1<q<p^*=\left\{\begin{array}{rl}
\displaystyle\frac{N\,p}{N-p},& \mbox{ if } 1<p<N,\\
+\infty, & \mbox{ if }p\ge N.
\end{array}\right.
\]
It is not difficult to generalize the Faber-Krahn inequalities \eqref{FKgen} to these quantities, again thanks to the P\'olya-Szeg\H{o} principle.
Then in \cite{MR1836803, MR2512200} some (non sharp) quantitative versions of these Faber-Krahn inequalities are proved, similar to Theorem \ref{teo:HN}.
\par 
Finally, we wish to cite the recent paper \cite{1512.00993} by Mazzoleni and Zucco. There it is shown that quantitative versions of the Faber-Krahn and Hong-Krahn-Szego inequalities can be used to show topological properties of minimizers for a particular spectral optimization problem (see \cite[Theorem 1.2]{1512.00993}). Namley, the minimization of the convex combination $t\,\lambda_1+(1-t)\,\lambda_2$, with volume constraint. 

\subsection{Nodal domains and Pleijel's Theorem}
Let $\Omega\subset\mathbb{R}^2$ be an open bounded set. For every $k\in\mathbb{N}$ let us note by $\mathcal{N}(k)$ the number of nodal domains of the Dirichlet eigenfunction $\varphi_k$ corresponding to $\lambda_k(\Omega)$. A classical result by Pleijel (see \cite{MR0080861}) asserts that 
\begin{equation}
\label{pleijel}
\limsup_{k\to\infty} \frac{\mathcal{N}(k)}{k}\le \left(\frac{2}{j_0}\right)^2.
\end{equation}
By observing that $2/j_0<1$, this results in particular asymptotically improves the classical {\it Courant nodal Theorem} (see \cite[Theorem 1.3.2]{MR2251558}), which asserts that $\mathcal{N}(k)\le k$.
The proof of \eqref{pleijel} can be obtained by combining the Faber-Krahn inequality on every nodal domain $\Omega_i$
\[
|\Omega|\,\lambda_k(\Omega)=\sum_{i=1}^{\mathcal{N}(k)} |\Omega_i|\,\lambda_1(\Omega_i)\ge \mathcal{N}(k)\,\pi\,j_0^2,
\] 
 and the classical {\it Weyl law}, which describes the asymptotic distribution of eigenvalues, i.e. 
\[
 \lim_{t\to\infty} \frac{\#\{\lambda \mbox{ eigenvalue} \, :\, \lambda\le t\}}{t}=\frac{|\Omega|}{4\,\pi}.
\]
As observed by Bourgain in \cite{MR3340367}, the estimate \eqref{pleijel} can be (slightly) improved by using the quantitative Faber-Krahn inequality and the packing density of balls in the Euclidean space. More precisely, the result of \cite{MR3340367} gives an explicit improvement in dimension $N=2$ by appealing to the Hansen-Nadirashvili result of Theorem \ref{teo:HNoriginal}, which comes indeed with an explicit constant. On the same problem, we also mention the paper \cite{MR3272823} by Steinerberger.
\subsection{Quantitative estimates in space forms}
In this manuscript we only discussed the Euclidean case. We briefly mention that some partial results are known for some special classes of manifolds (essentially the so-called {\it space forms}). 
\par
For example, the paper \cite{MR1315661} by Xu proves a stability result for the Szeg\H{o}-Weinberger inequality for smooth (geodesically) convex domains contained in a nonpositively curved space form (i.e. the hyperbolic space $\mathbb{H}^N$ or $\mathbb{R}^N$). More precisely, \cite[Theorem 4]{MR1315661} proves a {\it pinching result} which shows that if the spectral deficit $\mu_2(B)-\mu_2(\Omega)$ converges to $0$, then the Hausdorff distance from the set of geodesic balls with given volume goes to $0$.
\par
The paper \cite{MR1888036} by Avila proves a quantitative Faber-Krahn inequality for smooth (geodesically) convex domains on the hemisphere $\mathbb{S}^2_+$ (\cite[Theorem 0.1]{MR1888036}) or on the hyperbolic space $\mathbb{H}^2$ (see \cite[Theorem 0.2]{MR1888036}). These can be seen as the natural counterparts of Melas' result Theorem \ref{teo:melas}, indeed stability is measured with a suitable variant of his asymmetry $d_\mathcal{M}$. In \cite[Theorem 0.3]{MR1888036} the aforementioned Xu's result is extended to a sufficiently narrow polar cap contained in $\mathbb{S}_+^N$.
\par
In \cite{MR2439395} Aubry, Bertrand and Colbois prove pinching results for the Faber-Krahn and Ashbaugh-Benguria inequalities for convex sets in $\mathbb{S}^N$, $\mathbb{R}^N$ and $\mathbb{H}^N$. These results show that if the relevant spectral deficit is small, then the set is close to a ball in the Hausdorff metric. As for the hyperbolic space $\mathbb{H}^N$, it should be noticed that the inequality
\[
\frac{\lambda_2(\Omega)}{\lambda_1(\Omega)}\le \frac{\lambda_2(B)}{\lambda_1(B)},
\]  
{\it does not hold true} and that the correct replacement of the Ashbaugh-Benguria inequality is 
\begin{equation}
\label{babascione}
\lambda_2(\Omega)\le \lambda_2(B),\qquad \mbox{ if } \lambda_1(\Omega)=\lambda_1(B),
\end{equation}
where $B$ denotes a geodesic ball (see \cite[Theorem 1.1]{MR2359820}). Then for $\mathbb{H}^N$ the pinching result of \cite[Theorem 1.5]{MR1888036} exactly concerns inequality \eqref{babascione}.
\vskip.2cm

\appendix

\section{The Kohler-Jobin inequality and the Faber-Krahn hierarchy}
\label{sec:KJ}

Let $q>1$ be an exponent satisfying \eqref{q}, for every $\Omega\subset\mathbb{R}^N$ open set with finite measure we still denote $\lambda^q_1(\Omega)$ its first semilinear Dirichlet eigenvalue, see definition \eqref{autolavoro}. The {\it Kohler-Jobin inequality} states that 
\begin{equation}
\label{KJ}
T(\Omega)^{\vartheta(q,N)} \lambda_1^q(\Omega)\ge T(B)^{\vartheta(q,N)}\, \lambda_1^q(B),\qquad \mbox{ where }\ \vartheta(q,N)=\frac{2+\displaystyle\frac{2}{q}\,N-N}{N+2}<1.
\end{equation}
The original statement by Marie-Th\'er\`ese Kohler-Jobin is for the first eigenvalue of the Laplacian, i.e. for $q=2$ (see \cite[Th\'eor\`eme 1]{MR0511908}). This can be equivalently reformulated by saying that balls are the only solutions to the following problem
\[
\min\{\lambda_1(\Omega)\, :\, T(\Omega)=c\}.
\]
We refer to \cite[Theorem 3]{MR0641547} and \cite[Theorem 1.1]{MR3264206} for the general version \eqref{KJ} of Kohler-Jobin inequality.
\par
An important consequence of Kohler-Jobin inequality is that the whole family of Faber-Krahn inequalities \eqref{FKgen} for the first semilinear eigenvalue $\lambda^q_1$
can be derived by combining the Saint-Venant inequality \eqref{sv} and \eqref{KJ}. Indeed, we have
\[
\begin{split}
|\Omega|^{\frac{2}{N}+\frac{2}{q}-1}\, \lambda_1^q(\Omega)&=\left(|\Omega|^{\frac{2}{N}+\frac{2}{q}-1}\,T(\Omega)^{-\vartheta(q,N)}\right)\, \left(T(\Omega)^{\vartheta(q,N)}\,\lambda_1^q(\Omega)\right)\\
&=\left(|\Omega|^{-\frac{N+2}{N}}\,T(\Omega)\right)^{-\vartheta(q,N)}\, \left(T(\Omega)^{\vartheta(q,N)}\,\lambda^q_1(\Omega)\right)\\
&\ge \left(|B|^{-\frac{N+2}{N}}\,T(B)\right)^{-\vartheta(q,N)}\, \left(T(B)^{\vartheta(q,N)}\,\lambda_1^q(B)\right)=|B|^{\frac{2}{N}+\frac{2}{q}-1}\, \lambda_1^q(B).
\end{split}
\]
More interestingly, we can translate every quantitative improvement of the Saint-Venant inequality into a similar statement for $\lambda_1^q$. Namely, we have the following expedient result.
\begin{proposition}[Faber-Krahn hierarchy]
\label{prop:gerarchia}
Let $q\ge 1$ be an exponent verifying \eqref{q}. 
Suppose that there exists $C>0$ and
\begin{itemize} 
\item $\mathcal{G}:[0,\infty)\to [0,\infty)$ a continuous increasing function vanishing at the origin only,
\vskip.2cm
\item $\Omega\mapsto d(\Omega)$  a scaling invariant shape functional vanishing on balls and bounded by some constant $M>0$,
\end{itemize}
such that 
for every open set $\Omega\subset\mathbb{R}^N$ with finite measure we have
\begin{equation}
\label{goal}
|B|^{-\frac{N+2}{N}}\,T(B)-|\Omega|^{-\frac{N+2}{N}}\,T(\Omega) \ge \frac{1}{C}\,\mathcal{G}(d(\Omega)).
\end{equation}
Then we also have
\[
|\Omega|^{\frac{2}{N}+\frac{2}{q}-1}\,\lambda_1^q(\Omega)-|B|^{\frac{2}{N}+\frac{2}{q}-1}\, \lambda_1^q(B)\ge c\, \mathcal{G}(d(\Omega)).
\]
The constant $ c>0$ depends on $C,\,N,\,q$ and $\mathcal{G}(M)$ only and is given by
\[
c=(2^{\vartheta}-1)\,|B|^{\frac{2}{N}+\frac{2}{q}-1}\,\lambda_1^q(B)\,\min\left\{\frac{1}{C}\,\frac{|B|^\frac{N+2}{N}}{T(B)},\,\frac{1}{\mathcal{G}(M)}\right\}.
\]
\end{proposition}
\begin{proof}
Without loss of generality, we can suppose that $|\Omega|=1$ and let $B$ be a ball having unit measure. We also use the shortcut notation $\vartheta=\vartheta(q,N)$. By \eqref{KJ} one obtains
\begin{equation}
\label{cappio}
\frac{\lambda_1^q(\Omega)}{\lambda_1^q(B)}-1\ge \left(\frac{T(B)}{T(\Omega)}\right)^\vartheta-1.
\end{equation}
Since $0<\vartheta\le 1$, by concavity we have
\[
t^\vartheta-1\ge (2^{\vartheta}-1)\, (t-1),\qquad t\in[1,2].
\]
Thus from \eqref{cappio} and \eqref{goal} we can easily infer that if $T(B)\le 2\, T(\Omega)$, then
\[
\frac{\lambda_1^q(\Omega)}{\lambda_1^q(B)}-1\ge (2^\vartheta-1)\, \left(\frac{T(B)}{T(\Omega)}-1\right)\ge \frac{1}{C}\,\frac{2^\vartheta-1}{T(B)}\, \mathcal{G}(d(\Omega)).
\]
In the last inequality we also used that $T(\Omega)\le T(B)$ by Saint-Venant inequality. On the other hand, if $T(B)>2\, T(\Omega)$, still by \eqref{cappio} we get
\[
\frac{\lambda_{1}^q(\Omega)}{\lambda_{1}^q(B)}-1\ge 2^\vartheta-1\ge \frac{2^\vartheta-1}{\mathcal{G}(M)}\, \mathcal{G}(d(\Omega)),
\]
since $d(\Omega)\le M$ and $\mathcal{G}$ is increasing by hypothesis.
\end{proof}
\begin{remark}
Examples of shape functionals $d$ as in the previous statement are: $d_\mathcal{M}$ defined in \eqref{melas}, $d_\mathcal{N}$ defined in \eqref{d1} and the Fraenkel asymmetry $\mathcal{A}$.
\end{remark}

\section{An elementary inequality for monotone functions}

Szeg\H{o}'s proof of the inequality 
\[
\frac{1}{|\Omega|}\,\left(\frac{1}{\mu_2(\Omega)}+\frac{1}{\mu_3(\Omega)}\right)\ge \frac{1}{|B|}\,\left(\frac{1}{\mu_2(B)}+\frac{1}{\mu_3(B)}\right),
\]
is based on the following elementary inequality for monotone functions of one real variable. We give its proof for completeness.
\begin{lemma}[Monotonicity lemma]
\label{lm:hersch}
Let $f:[0,1]\to\mathbb{R}_+$ and $\Phi:[0,1]\to\mathbb{R}_+$ be two non-decreasing functions, not identically vanishing. Then we have
\begin{equation}
\label{hersch}
\displaystyle\int_0^1 f(t)\,\Phi(t)\,t\,dt\ge \left(\frac{\displaystyle\int_0^1 f(t)\,t\,dt}{\displaystyle\int_0^1 t\,dt}\right)\, \int_0^1\, \Phi(t)\,t\,dt
\end{equation}
\end{lemma}
\begin{proof}
By approximation, we can assume that $\Phi$ is $C^1$. We first observe that if we set 
\[
F(t)=\int_0^t f(s)\,s\,ds\qquad \mbox{ and }\qquad \overline f=\frac{\displaystyle\int_0^1 f(t)\,t\,dt}{\displaystyle\int_0^1 t\,dt},
\]
then \eqref{hersch} is equivalent to
\[
\int_0^1 \Big(F'(t)-\overline f\,t\Big)\, \Phi(t)\,dt\ge 0.
\]
If we perform an integration by parts, this in turn is equivalent to
\[
\int_0^1 \Big(\overline f\,\frac{t^2}{2}-F(t)\Big)\, \Phi'(t)\,dt\ge 0.
\]
Since $\Phi'\ge 0$, in order to conclude it would be sufficient to prove 
\begin{equation}
\label{prehersch}
F(t)\le \overline f\,\frac{t^2}{2},\qquad \mbox{ i.\,e. }\quad \frac{2}{t^2}\,\int_0^t f(s)\,s\,ds\le 2\,\int_0^1 f(s)\,s\,ds.
\end{equation}
In turn, in order to show \eqref{prehersch} it would be enough to prove that the function
\[
H(t)=\frac{2}{t^2}\,\int_0^t f(s)\,s\,ds,
\]
is monotone non-decreasing. A direct computation gives
\[
H'(t)=-\frac{4}{t^3}\,\int_0^t f(s)\,s\,ds+\frac{2}{t}\,f(t)=\frac{2}{t}\,\left[f(t)-\frac{2}{t^2}\,\int_0^t f(s)\,s\,ds\right]\ge 0,
\]
where in the last inequality we used that $f$ is non-decreasing. This proves \eqref{prehersch} and thus \eqref{hersch}.
\end{proof}
The result of Theorem \ref{thm:nikolaisego} is based on the following improvement of the previous inequality.
\begin{lemma}[Improved monotonicity lemma]
\label{lm:hersch+}
Let $f:[0,1]\to\mathbb{R}_+$ be a strictly increasing function. Let $\Phi:[0,1]\to\mathbb{R}_+$ be a non-decreasing function, such that
\begin{equation}
\label{fihona}
\Phi(t)=\sum_{n=0}^\infty \alpha_n\,t^n,\qquad \mbox{ with }\qquad\alpha_n\ge 0\quad \mbox{ and }\quad \alpha_0=\Phi(0)\le \gamma\,\int_0^1 \Phi(t)\,t\,dt,
\end{equation}
for some $0\le \gamma<2$.
Then we have
\begin{equation}
\label{hersch+}
\displaystyle\int_0^1 f(t)\,\Phi(t)\,t\,dt\ge \left(\frac{\displaystyle\int_0^1 f(t)\,t\,dt}{\displaystyle\int_0^1 t\,dt}\right)\, \int_0^1\, \Phi(t)\,t\,dt+c\,(2-\gamma)\,\int_0^1 \Phi(t)\,t\,dt,
\end{equation}
for some $c>0$ depending on $f$.
\end{lemma}
\begin{proof}
We use the same notation as in the proof of Lemma \ref{lm:hersch}. We then recall that
\[
\begin{split}
\displaystyle\int_0^1 f(t)\,\Phi(t)\,t\,dt-\left(\frac{\displaystyle\int_0^1 f(t)\,t\,dt}{\displaystyle\int_0^1 t\,dt}\right)\, \int_0^1\, \Phi(t)\,t\,dt&=
\int_0^1 \Big(\overline f\,\frac{t^2}{2}-F'(t)\Big)\, \Phi'(t)\,dt\\
&=\sum_{n=1}^\infty n\,\alpha_n\, \int_0^1\Big(\overline f\,\frac{t^2}{2}-F(t)\Big)\,t^{n-1}\,dt\\
&=\sum_{n=1}^\infty \frac{n\,\alpha_n}{2}\, \int_0^1\Big(H(1)-H(t)\Big)\,t^{n+1}\,dt.
\end{split}
\]
For the last term, we recall that $H'>0$ on the interval $[0,1]$, thus for $n\ge 1$ 
\[
\begin{split}
\int_0^1\Big(H(1)-H(t)\Big)\,t^{n+1}\,dt&\ge\int_{0}^{1-\frac{1}{2\,n}}\Big(H(1)-H(t)\Big)\,t^{n+1}\,dt\\
&\ge \left[H(1)-H\left(1-\frac{1}{2\,n}\right)\right]\,\frac{1}{n+2}\,\left(1-\frac{1}{2\,n}\right)^{n+2}.
\end{split}
\]
Observe that from \eqref{fihona} we get
\[
\sum_{n=1}^\infty \frac{\alpha_n}{n+2}=\int_0^1 \Phi(t)\,t\,dt-\alpha_0\,\int_0^1 t\,dt\ge \left(1-\frac{\gamma}{2}\right)\,\int_0^1 \Phi(t)\,t\,dt.
\]
Thus in order to conclude the proof of \eqref{hersch+} we need to prove that for $n\ge 1$
\[
n\,\left[H(1)-H\left(1-\frac{1}{2\,n}\right)\right]\,\left(1-\frac{1}{2\,n}\right)^{n+2}\ge c>0.
\]
By observing that 
\[
\left(1-\frac{1}{2\,n}\right)^{n+2}\ge \frac{1}{8},\qquad \lim_{n\to\infty}n\,\left[H(1)-H\left(1-\frac{1}{2\,n}\right)\right]=\frac{H'(1)}{2}=f(1)-2\,\int_0^1 f(s)\,s\,ds>0,
\]
we conclude that this holds true.
\end{proof}

\section{A weak version of the Hardy-Littlewood inequality}

The proofs by Weinberger and Brock are based on a test function argument and an isoperimetric-like property of balls with respect to weighted volumes of the type
\[
\Omega\mapsto \int_\Omega f(|x|)\,dx,
\]
with $f$ positive monotone function.
This is encoded in the following result, which can be seen as a particular case of the Hardy-Littlewood inequality.
\begin{lemma}
\label{lm:HL}
Let $f:\mathbb{R}_+\to\mathbb{R}_+$ be a non-increasing function and $g:\mathbb{R}_+\to\mathbb{R}_+$ a non-decreasing function. Let $\Omega\subset\mathbb{R}^N$ be an open set with finite measure, we denote by $\Omega^*$ the ball centered at the origin such that $|\Omega|=|\Omega^*|$. Then we have
\[
\int_{\Omega^*} f(|x|)\, dx\ge \int_{\Omega} f(|x|)\, dx\qquad \mbox{ and }\qquad \int_{\Omega^*} g(|x|)\, dx\le \int_{\Omega} g(|x|)\, dx.
\]
\end{lemma}
The quantitative versions of Szego-Weinberger and Brock-Weinstock inequalities are based on the following simple but useful improved version of Lemma \ref{lm:HL}.
\begin{lemma}
\label{lm:HLquanto}
Let $f:\mathbb{R}_+\to\mathbb{R}_+$ be a nonincreasing function. Let $\Omega\subset\mathbb{R}^N$ be an open set with finite measure, we denote by $\Omega^*$ the ball centered at the origin such that $|\Omega|=|\Omega^*|$. Then we have
\begin{equation}
\label{HLquanto}
\int_{\Omega^*} f(|x|)\, dx-\int_{\Omega} f(|x|)\, dx\ge N\,\omega_N\, \int_{R_1}^{R_2} |f(\varrho)-f(R_\Omega)|\,\varrho^{N-1}\, d\varrho,
\end{equation}
where 
\begin{equation}
\label{raggi}
R_\Omega=\left(\frac{|\Omega|}{\omega_N}\right)^\frac{1}{N},\qquad R_1=\left(\frac{|\Omega\cap\Omega^*|}{\omega_N}\right)^\frac{1}{N}\qquad \mbox{ and }\qquad R_2=\left(\frac{|\Omega\setminus\Omega^*|+|\Omega|}{\omega_N}\right)^\frac{1}{N}.
\end{equation}
\end{lemma}
\begin{proof}
The proof is quite simple, first of all we observe that 
\begin{equation}
\label{elidi}
\int_{\Omega^*} f(|x|)\, dx-\int_{\Omega} f(|x|)\, dx=\int_{\Omega^*\setminus\Omega} f(|x|)\,dx-\int_{\Omega\setminus\Omega^*} f(|x|)\,dx.
\end{equation}
Then the idea is to rearrange the set $\Omega^*\setminus\Omega$ into a spherical shell having radii $R_\Omega$ and $R_2$, and similarly to rearrange the set $\Omega\setminus \Omega^*$ into a spherical shell having radii $R_\Omega$ and $R_1$ (see Figure \ref{fig:anelli}).
\begin{figure}
\includegraphics[scale=.4]{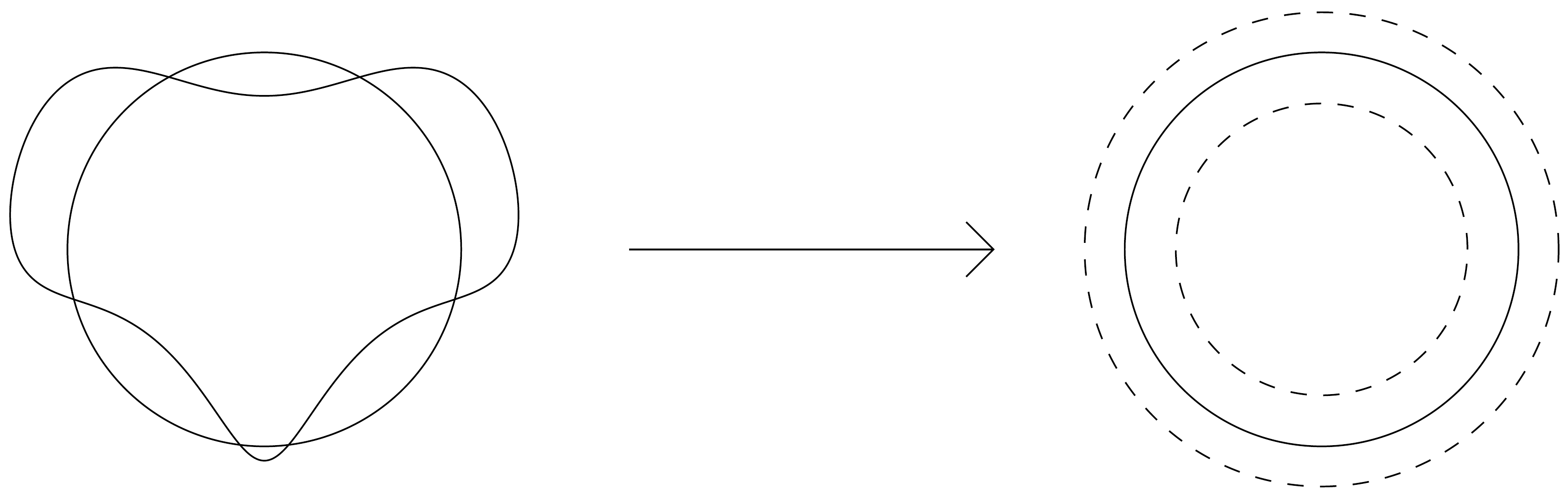}
\caption{The rearrangement of $\Omega\Delta \Omega^*$ for the proof of \eqref{HLquanto}}
\label{fig:anelli}
\end{figure} 
Thanks to the definition \eqref{raggi} of $R_1$ and $R_2$, we have 
\[
|B_{R_2}|-|\Omega^*|=|\Omega^*\setminus\Omega|\qquad \mbox{ and }\qquad |\Omega^*|-|B_{R_1}|=|\Omega^*\setminus \Omega|,
\]
i.e. the two spherical shells mentioned above will preserve the measure. We will prove below that this property and the monotonicity of $f$ entail
\begin{equation}
\label{elidi1}
\int_{\Omega^*\setminus\Omega} f(|x|)\, dx\ge \int_{\Omega^*\setminus B_{R_1}} f(|x|)\, dx\qquad \mbox{ and }\qquad \int_{\Omega\setminus\Omega^*} f(|x|)\, dx\le \int_{B_{R_2}\setminus \Omega^*} f(|x|)\,dx.
\end{equation}
This means that the worst scenario for the right-hand side of \eqref{elidi} is when all the mass is uniformly distributed around $\partial\Omega^*$.
Thus from \eqref{elidi} and \eqref{elidi1}, we can obtain
\[
 \int_{\Omega^*} f(|x|)\, dx-\int_{\Omega} f(|x|)\, dx\ge \int_{\Omega^*\setminus B_{R_1}} f(|x|)\, dx-\int_{B_{R_2}\setminus \Omega^*} f(|x|)\,dx.
\]
In order to conclude, we just observe that since by contruction $|\Omega^*\setminus B_{R_1}|=|B_{R_2}\setminus \Omega^*|$, then we get
\[
\int_{\Omega^*\setminus B_{R_1}} f(|x|)\, dx-\int_{B_{R_2}\setminus \Omega^*} f(|x|)\,dx=\int_{\Omega^*\setminus B_{R_1}} [f(|x|)-f(R_\Omega)]\, dx-\int_{B_{R_2}\setminus \Omega^*} [f(|x|)-f(R_\Omega)]\,dx.
\]
This finally gives \eqref{HLquanto}, by using polar coordinates and using once again that $f$ is nonincreasing.
\vskip.2cm\noindent
Let us now prove \eqref{elidi1}. We first observe that we have
\begin{equation}
\label{misurine}
|(\Omega^*\cap\Omega)\setminus B_{R_1}|=|(\Omega^*\setminus \Omega)\cap B_{R_1}|.
\end{equation}
Indeed, we get
\[
\begin{split}
|(\Omega^*\cap\Omega)\setminus B_{R_1}|+|(\Omega^*\setminus \Omega)\setminus B_{R_1}|=|\Omega^*\setminus B_{R_1}|&=|\Omega^*\setminus\Omega|\\
&=|(\Omega^*\setminus\Omega)\cap B_{R_1}|+|(\Omega^*\setminus\Omega)\setminus B_{R_1}|,
\end{split}
\]
which proves \eqref{misurine}. By using this and the monotonicity of $f$, we get
\[
\begin{split}
\int_{\Omega^*\setminus\Omega} f(|x|)\, dx&=\int_{(\Omega^*\setminus \Omega)\cap B_{R_1}} f(|x|)\,dx+\int_{(\Omega^*\setminus \Omega)\setminus B_{R_1}} f(|x|)\,dx\\
&\ge f(R_1)\,|(\Omega^*\setminus \Omega)\cap B_{R_1}|+\int_{(\Omega^*\setminus \Omega)\setminus B_{R_1}} f(|x|)\,dx\\
&=f(R_1)\,|(\Omega^*\cap\Omega)\setminus B_{R_1}|+\int_{(\Omega^*\setminus \Omega)\setminus B_{R_1}} f(|x|)\,dx\\
&\ge \int_{(\Omega^*\cap\Omega)\setminus B_{R_1}} f(|x|)\,dx+\int_{(\Omega^*\setminus \Omega)\setminus B_{R_1}} f(|x|)\,dx=\int_{\Omega^*\setminus B_{R_1}} f(|x|)\,dx,
\end{split}
\]
which proves the first inequality in \eqref{elidi1}. The second one is proved similarly.
\end{proof}

\section{Some estimates for convex sets}

We still denote by $r_\Omega$ the inradius of a set and by $\mathrm{Haus}$ the Hausdorff distance between sets, defined by \eqref{indahaus}.
\begin{lemma}
\label{lm:blasphemy}
Let $\Omega\subset\mathbb{R}^N$ be an open bounded convex set. For every ball $B_R$ of radius $R$, we have
\[
\mathrm{Haus}(\Omega,B_R)\ge R-r_\Omega.
\]
\end{lemma}
\begin{proof}
We first observe that if $\mathrm{Haus}(\Omega,B)\ge R$ there is nothing to prove. Thus, we set for simplicity $\delta=\mathrm{Haus}(\Omega,B_R)$ and suppose $\delta <R$. By definition of Hausdorff distance, we have that
\begin{equation}
\label{contiene}
B_R\subset \Omega+\delta\,B_1(0)=:\Omega_\delta,
\end{equation}
where $+$ denotes the Minkowski sum of sets. Let $x\in\Omega$ and $x'\in\partial\Omega$ be such that 
\[
|x-x'|=\mathrm{dist}(x,\partial\Omega).
\]
We also consider the point $x'_\delta=x'+\delta\,(x'-x)/|x-x'|\in\partial\Omega_\delta$, then we obtain for every $x\in\Omega$
\[
\mathrm{dist}(x,\partial\Omega)=|x-x'|\ge |x-x'_\delta|-|x'_\delta-x'|\ge \mathrm{dist}(x,\partial\Omega_\delta)-\delta.
\]
Since $r_\Omega$ coincides with the supremum on $\Omega$ of the distance function, this shows
\begin{equation}
\label{prandi}
r_\Omega+\delta\ge \sup_{x\in\Omega}\mathrm{dist}(x,\partial\Omega_\delta).
\end{equation}
We now want to show that
\begin{equation}
\label{posti}
\sup_{x\in\Omega_\delta\setminus \Omega}\mathrm{dist}(x,\partial\Omega_\delta)\le\delta.
\end{equation}
Let us take $x\in \Omega_\delta\setminus\Omega$, then we know that
\[
x=x'+t\,\omega,\qquad \mbox{ for some } x'\in\partial\Omega,\ 0\le t<\delta,\ \omega\in\mathbb{S}^{N-1}.
\] 
The point $x''=x'+\delta\,\omega$ lies on the boundary of $\partial\Omega_\delta$, thus we get
\[
\mathrm{dist}(x,\partial\Omega_\delta)\le |x-x''|=(\delta-t)<\delta.
\]
This shows \eqref{posti}. By putting \eqref{prandi} and \eqref{posti} together, we thus get
\[
r_{\Omega_\delta}=\sup_{x\in\Omega_\delta} \mathrm{dist}(x,\partial\Omega_\delta)=\max\left\{\sup_{\Omega} \mathrm{dist}(x,\partial\Omega_\delta),\,\sup_{\Omega_\delta\setminus \Omega} \mathrm{dist}(x,\partial\Omega_\delta)\right\}\le r_\Omega+\delta.
\]
It is only left to observe that from \eqref{contiene}, we get
\[
R\le r_{\Omega_\delta}\le r_\Omega+\delta,
\]
as desired.
\end{proof}
The following result asserts that for convex sets $\lambda_1$ is equivalent to the inradius.
\begin{proposition}
For every $\Omega\subset\mathbb{R}^N$ open convex set such that $r_\Omega<+\infty$ we have
\begin{equation}
\label{inradius}
\frac{1}{4\,r_\Omega^2}\le\lambda_1(\Omega)\le \frac{\lambda_1(B_1)}{r_\Omega^2},
\end{equation}
where $B_1$ is any $N-$dimensional ball of radius $1$.
\end{proposition}
\begin{proof}
The upper bound easily follows from the monotonicity and scaling properties of $\lambda_1$. For the lower bound, we can use the {\it Hardy inequality} for convex sets (see \cite{MR174788})
\[
\frac{1}{4}\,\int_{\Omega} \left|\frac{u}{d_\Omega}\right|^2\,dx<\int_\Omega |\nabla u|^2\,dx,\qquad u\in W^{1,2}_0(\Omega),
\]
where we used the notation $d_\Omega(x)=\mathrm{dist}(x,\partial\Omega)$. By recalling that the inradius $r_\Omega$ coincides with the maximum of $d_\Omega$, we get $d_\Omega\le r_\Omega$ and taking the infimum over $W^{1,2}_0(\Omega)$ we get the conclusion.
\end{proof}

\begin{lemma}
For every $\Omega\subset\mathbb{R}^N$ open bounded convex set, we have
\begin{equation}
\label{eccentricity}
\frac{1}{N\,\omega_N}\,\frac{|\Omega|}{r_\Omega}\le\mathrm{diam}(\Omega)^{N-1}.
\end{equation}
\end{lemma}
\begin{proof}
By using Coarea formula, we obtain
\[
\begin{split}
|\Omega|=\int_\Omega\,dx=\int_{0}^{r_\Omega} P(\{x\in\Omega\, :\, d_\Omega(x)=t\})\,dt\le r_\Omega\,P(\Omega),
\end{split}
\]
thanks to the convexity of the level sets of the distance function\footnote{We use that on convex sets, the perimeter is monotone with respect to inclusion.}. Since $\Omega$ is contained in a ball with radius $\mathrm{diam}(\Omega)$, we have
\[
P(\Omega)\le N\,\omega_N\,\mathrm{diam}(\Omega)^{N-1}.
\]
This concludes the proof.
\end{proof}
\begin{remark}
\label{rem:ABC}
By joining \eqref{eccentricity} and \eqref{inradius}, we obtain the estimate
\[
\lambda_1(\Omega)\le (N\,\omega_N)^2\,\lambda_1(B_1)\,\left(\frac{\mathrm{diam}(\Omega)^{N-1}}{|\Omega|}\right)^2.
\]
Thus in particular for every sequence of open convex sets $\{\Omega_n\}_{n\in\mathbb{N}}\subset\mathbb{R}^N$ such that
\[
|\Omega_n|=1\qquad \mbox{ and }\qquad \lim_{n\to\infty} \lambda_1(\Omega_n)=+\infty,
\]
then the diameters diverge to $+\infty$ as well. This fact has been used in the proof of Theorem \ref{thm:melasrap}.
\end{remark}

\end{document}